\RequirePackage{fix-cm}
\documentclass[smallextended,envcountsame]{svjour3}       
\smartqed  
\usepackage{graphicx}
\usepackage{geometry} 
\usepackage{amsfonts}
\usepackage{amsmath,amssymb,bm}
\usepackage{xcolor}
\usepackage{mathtools}
\usepackage{ulem}
\usepackage{pgfplots}
\usepackage{subcaption}
\pgfplotsset{compat=1.18}
\usepackage{cancel}
\usepackage{mathrsfs}
\usepackage{comment}
\usepackage{hyperref}

\newcommand{\dual}[2]{\left \langle #1,#2 \right \rangle}
\newcommand{\dotp}[2]{\left ( #1,#2 \right )}

\newcommand{\IR}{\mathbb{R}}
\newcommand{\IC}{\mathbb{C}}

\newcommand{\IN}{\mathbb{N}}

\newcommand{\IV}{\mathcal{V}}

\newcommand{\bx}{\bm{x}}

\newcommand{\cB}{\mathcal{B}}
\newcommand{\cF}{\mathcal{F}}
\newcommand{\cL}{\mathcal{L}}
\newcommand{\cH}{\mathcal{H}}
\newcommand{\cI}{\mathcal{I}}

\newcommand{\cT}{\mathcal{Q}}
\newcommand{\cV}{\mathcal{V}}

\newcommand{\cX}{\mathcal{X}}

\newcommand{\sfA}{\mathsf{A}}

\newcommand{\rb}{{(\mathrm{rb})}}
\newcommand{\hU}{\widehat{U}^{(2)}}

\newcommand{\cVrb}{\cV^\rb_R}
\newcommand{\vrb}{v^\rb_R}

\DeclareMathOperator{\erf}{\mathrm{erf}}
\DeclareMathOperator{\erfc}{\mathrm{erfc}}
\DeclareMathOperator{\Sinc}{\mathrm{Sinc}}

\DeclareMathOperator*{\argmin}{arg\,min}

\DeclareMathOperator{\abs}{\mathrm{abs}}
\DeclareMathOperator{\spanv}{\mathrm{span}}

\newcommand{\norm}[1]{\left \lVert #1 \right \rVert}
\newcommand{\snorm}[1]{\left \lvert #1 \right \rvert}

\DeclareMathOperator*{\esssinf}{ess\,inf}
\DeclareMathOperator*{\esssup}{ess\,sup}
\renewcommand{\Re}{\mathrm{Re}}
\renewcommand{\Im}{\mathrm{Im}}

\numberwithin{equation}{section}
\numberwithin{theorem}{section}

\mathtoolsset{showonlyrefs,showmanualtags}

\def\cdf(#1)(#2)(#3){0.5*(1+(erf((#1-#2)/(#3*sqrt(2)))))}%

\begin{document}

\title{A Model Order Reduction Method for Seismic Applications Using the Laplace Transform
}


\author{Fernando Henr\'iquez         \and
        Matthias Schlottbom 
}


\institute{Fernando Henr\'iquez \at
Institute for Analysis and Scientific Computing, Vienna University of Technology, Wiedner Hauptstra{\ss}e 8-10, A-1040 Wien, Austria. \\
              \email{fernando.henriquez@asc.tuwien.ac.at}           
           \and
           Matthias Schlottbom \at
            Department of Applied Mathematics, University of Twente, 7500 AE Enschede, the Netherlands. \\ \email{m.schlottbom@utwente.nl}  
}

\date{Version of \today.}

\maketitle

\begin{abstract}
We devise and analyze a reduced basis model order reduction (MOR) strategy for an abstract wave problem with vanishing initial conditions and a source term given by the product of a temporal Ricker wavelet and a spatial profile. Such wave problems comprise the acoustic and elastic wave equations, with applications in seismic modeling. Motivated by recent Laplace-domain MOR methodologies, we construct reduced bases that approximate the time-domain solution with exponential accuracy. We prove convergence bounds that are explicit and robust with respect to the parameters controlling the Ricker wavelet’s shape and width and identify an intrinsic accuracy limit dictated by the wavelet’s value at the initial time. In particular, the resulting error bound is independent of the underlying Galerkin discretization space and yields computable criteria for the regime in which exponential convergence is observed.
\end{abstract}

\section{Introduction}

\subsection{Motivation: Seismic Modeling}
\label{sec:motivation}
The computational modeling of seismic phenomena---through techniques such as seismic tomography
\cite{Rawlinson2014SeismicTomography,Tromp2005AdjointBanana} and full waveform inversion
\cite{Bozdag2011MisfitPhaseEnvelope,Tromp2013FullWaveform}---requires the solution of a
{\it many-query} problem, i.e., repeated evaluations of a computational model representing the
underlying physical system. For realistic models, these repeated solves are computationally
demanding and often preclude real-time simulations. 

Seismic wave propagation is typically modeled by the wave equation, which may describe
either acoustic or elastic waves, supplemented with appropriate initial conditions and
a source term represented by a nonzero right-hand side.
To capture the time dependence of realistic seismic sources, a common modeling choice is
the so-called Ricker wavelet, originally introduced in
\cite{ricker1943further,ricker1944wavelet}, which provides a convenient representation of
the spectral properties of seismic phenomena.
The Ricker wavelet has been widely employed as a source term in computational seismic
modeling; see, for instance,
\cite{komatitsch1999introduction,alford1974accuracy}; see also \cite{schneider2004plane,schneider2025understanding,SchneiderChen2011} for applications in electromagnetic problems.
Among these contributions, the most recent work to our knowledge is \cite{hawkins2025model},
where a model order reduction (MOR) method for the parametric elastic wave equation is
proposed, based on the construction of a reduced basis in the Laplace domain, to efficiently compute time-domain seismograms for a two-dimensional subsurface model of Groningen, the Netherlands.
For current approaches tackling this particular problem, we refer to \cite{Hawkinsetal2023} for a model reduction approach based on eigenfunctions.

The main goal of this work is to devise and analyze a model order reduction (MOR)
strategy for an abstract wave problem, which could represent either acoustic or elastic, subject to
vanishing initial conditions and driven by a right-hand side defined as the product
of a temporal Ricker wavelet and a spatial profile.
While this setting shares structural similarities with the framework in \cite{hawkins2025model}, we shift from parametric exploration to focus on the fundamental acceleration of the evolution problem itself.
Motivated by the MOR methodologies introduced in the recent works
\cite{henriquez2024fast,HenriquezHesthaven2024}, which similarly exploit the Laplace
transform for the construction of efficient reduced bases, we propose and
thoroughly analyze reduced basis schemes that approximate the solution of the
evolution problem with exponential accuracy, at a convergence rate that is robust
with respect to the parameters governing the shape and width of the Ricker wavelet.
This robustness is achieved by carefully exploiting the properties of the Ricker wavelet’s spectral content.
Notably, although the original problem is not parametric, after the application of the
Laplace transform, the Laplace variable enters as a parameter. 
In view of this, we borrow ideas from model order reduction for parametric wave problems to accelerate single-instance evolution solves.

\subsection{Model Order Reduction of Parametric PDEs and The Reduced Basis Method}
\label{sec:mor_parametric_PDEs}
Model order reduction (MOR) comprises a broad class of techniques designed to
reduce the computational cost associated with families of parameter-dependent
problems, such as parametric partial differential equations.
A central observation is that, although a high-fidelity discretization may be
very large, the associated {\it solution manifold} often possesses a low-dimensional
structure that can be accurately approximated.

Among MOR approaches, the reduced basis (RB) method is one of the most widely used
frameworks.
It is typically organized into an {\it offline} stage, in which representative
high-fidelity solutions (snapshots) are computed for selected parameter values,
and an {\it online} stage, in which new parameter instances are treated efficiently
by projecting onto the resulting reduced space.
A standard construction of the reduced space is proper orthogonal decomposition (POD)
\cite{berkooz1993proper,kunisch2001galerkin,liang2002proper}, which extracts a basis
of prescribed dimension via a singular value decomposition of the snapshot matrix.
Alternatively, the basis can be generated by greedy procedures
\cite{hesthaven2014efficient,bui-thanh_model_2008,buffa_priori_2012,devore_greedy_2013}
(and, for evolution problems, POD--Greedy strategies
\cite{haasdonk2013convergence,siena2023introduction}).

RB methods for stationary problems are by now well established, including rigorous
a posteriori error estimation and certification procedures; see
\cite{hesthaven2016certified,quarteroni2015reduced,prud2002reliable,rozza2014fundamentals}.
For time-dependent problems, a variety of extensions and alternative strategies have been proposed.
For example, \cite{haasdonk2008reduced} applies RB techniques to linear evolution equations
with a Finite Volume spatial discretization, while non-linear MOR based on local
reduced-order bases is discussed in \cite{lieu2006reduced}.
Further developments include structure-preserving RB methods for Hamiltonian systems
\cite{afkham2017structure,hesthaven2022rank,hesthaven2021structure,hesthaven2023adaptive},
non-intrusive approaches based on radial basis function interpolation
\cite{audouze2013nonintrusive,xiao2017parameterized}, and more recently {\it data-driven}
methodologies such as dynamic mode decomposition (DMD)
\cite{schmid2010dynamic,kutz2016dynamic,duan2023non} and operator learning
\cite{peherstorfer2016data,duan2023machine}; see \cite{Hesthaven2022} for a broader overview.

It is worth noting that the construction of reduced-order models based on linear spaces
may be particularly challenging for wave- and transport-dominated problems.
In \cite{GU19}, it was shown that the Kolmogorov $N$-width of the solution manifold
associated with a one-dimensional parametric wave equation admits the lower bound
$N^{-\frac{1}{2}}$; related results for linear transport can also be found in \cite{OhlbergerRave2016} with refinements in \cite{AGU25}.
A slow decay of the Kolmogorov $N$-width suggests that achieving high accuracy may require reduced spaces of large dimension, which can undermine the efficiency of reduced basis methods as fast surrogates for the underlying parametric problem.
In order to {\it break} this phenomenon, which is usually referred as the 
Kolmogorov barrier \cite{peherstorfer2022breaking}, one needs to resort to {\it non-linear}
MOR techniques. Among these, we highlight online adaptive basis \cite{peherstorfer2020model,iollo2014advection},
dynamic low-rank approximation \cite{koch2007dynamical,sapsis2009dynamically,musharbash2015error},
neural networks \cite{lee2020model}, and the use of non-linear transformations to improve RB error
\cite{reiss2018shifted,taddei2020registration}.
Furthermore, though no longer part of the realm of non-linear MOR techniques, an active line
of research in MOR reduction for transport-dominated parametric problems focuses on the stability
of reduced-order models, see, e.g., \cite{urban2012new,yano2014space}, also \cite{wang2012proper,pan2018data,iollo2000stability} for works on closure models, and \cite{peng2016symplectic,hesthaven2022structure} for symplectic RB construction.
For a comprehensive review of MOR techniques for time-dependent problems, we refer to \cite{hesthaven2022reduced} and, in particular for transport-dominated problems, the recent survey \cite{hesthaven2026nonlinear}.

\subsection{The Laplace Transform for MOR of Evolution Problems}
\label{sec:laplace_MOR}
An alternative route to reduced-order modeling for evolution problems is to work in the
Laplace domain. Applying the Laplace transform to a time-dependent PDE yields a family of
time-independent problems that depend on the complex Laplace variable, which may be viewed as
an additional parameter. This viewpoint makes it possible to leverage established reduced basis
technology for stationary parametric problems.

This idea has been explored, for instance, in \cite{guglielmi2023model} for parametric linear
second-order parabolic equations, where the time-domain solution is recovered by numerically
inverting the Laplace transform via contour-deformation techniques \cite{guglielmi2021pseudospectral}.
Related Laplace-domain reduced basis constructions have also been considered in
\cite{huynh_laplace_2011,Bigoni2020a}. 
In \cite{Bigoni2020a}, the focus is on the wave equation,
for which contour-deformation strategies break down, and the inverse Laplace transform is instead
computed using Weeks' method, see also \cite{hawkins2025model} for an application of Weeks' method in the context of elastic wave equations. We refer to \cite{kuhlman_review_2013} and the references there for an overview of numerical Laplace inversion techniques.

A central difficulty in such approaches is, however, the numerical inversion of the Laplace transform, which
may suffer from stability and accuracy issues, particularly over long time
intervals. The present work is motivated by these challenges and takes a first step toward a stable
and accurate reduced basis methodology for wave problems driven by Ricker wavelet source terms.

\subsection{Contributions}
In this work, we introduce a MOR technique for the abstract {\it non-parametric} wave equation
with vanishing initial conditions and a right-hand side given by the product of a spatial profile
and a temporal Ricker wavelet. As explained in Section~\ref{sec:motivation}, this setting is
motivated by its relevance in the computational modeling of seismic phenomena. 

For the construction of efficient reduced spaces, we follow the Laplace-transform-based approach
of \cite{henriquez2024fast,HenriquezHesthaven2024}, in the following referred to as LT-MOR algorithm. 
In contrast to the Laplace-domain MOR strategies
reviewed in Section~\ref{sec:laplace_MOR}, we do not recover the time-domain solution by explicitly
computing an inverse Laplace transform (e.g., via contour deformation or related inversion
methods). This circumvents stability and accuracy issues associated with numerical Laplace
inversion.

Moreover, \cite{henriquez2024fast,HenriquezHesthaven2024} show, for parabolic problems and the wave
equation, respectively, that reduced spaces constructed in the Laplace domain can be used
successfully for the underlying time-evolution problem and yield exponential convergence with
respect to both the reduced dimension and the number of Laplace-domain samples used to build the
basis. We note, however, that these analyses focus on reduced bases for semi-discrete problems
(discrete in space, continuous in time), whereas \cite{feischl2025optimal} provides a fully discrete
analysis for the parabolic case, including optimal time adaptivity.

From an algorithmic perspective, the proposed approach (building on
\cite{henriquez2024fast,HenriquezHesthaven2024,feischl2025optimal}) follows the standard
offline--online decomposition commonly used in projection-based MOR and summarized in
Section~\ref{sec:mor_parametric_PDEs}.

\begin{itemize}
    \item[(i)]
    {\bf Offline stage.}
    Applying the Laplace transform to evolution problem leads to a family
    of time-independent problems parametrized by the complex Laplace variable.
    A set of sampling points is selected in the Laplace domain, and for each
    sample the corresponding high-fidelity problem is solved after spatial
    finite element discretization.
    The resulting solutions are collected and compressed using a proper
    orthogonal decomposition to construct a reduced basis of low dimension.

    \item[(ii)]
    {\bf Online stage.}
    Once the reduced basis has been constructed, the original time-dependent
    problem is projected onto this space.
    The resulting reduced system is then advanced in time using a suitable
    time-integration scheme, leading to a significantly smaller system of
    ordinary differential equations. 
\end{itemize}

Within this framework, the online computations amount to updating only a small
number of reduced coefficients at each time step, in contrast to standard
time-stepping schemes that operate on the full set of degrees of freedom of the
high-fidelity model.
The computational cost of the offline stage is dominated by the solution of a
limited number of Laplace-domain problems, which can, however, be parallelized straightforwardly.
We emphasize that, in the proposed approach, the time-domain solution is obtained
without numerically inverting the Laplace transform.

Our main contributions are the following:
\begin{itemize}
    \item[(i)]
    We prove that the solution of the abstract wave equation, under the aforementioned
    assumptions on the data, can be approximated by finite-dimensional subspaces
    with exponential accuracy, at a rate independent of the parameters of the Ricker
    wavelet (peak arrival time $t_0$ and width $\alpha$), up to a threshold determined by the
    values of the Ricker wavelet and its time derivative at $t=0$.
    These quantities, in turn, decay like $e^{-(\alpha t_0)^2}$, up to multiplicative powers of $\alpha$. 
    \item[(ii)]
    We provide an explicit algorithm for the practical construction of a finite-dimensional
    space realizing the previously discussed exponential convergence.
    A key tool is the use of the so-called sinc quadrature for the selection of points
    in the Laplace domain, which in turn define the snapshots to be computed.
    We show that, with this strategy, the discretization error in the Laplace domain
    also decays exponentially, with a rate that again does not depend on the parameters
    of the Ricker wavelet.
\end{itemize}

A few remarks are in place. Firstly, our analysis is independent of the underlying Galerkin
discretization space and is valid in an infinite-dimensional setting. Secondly, even though we do
not consider the parametric problem as in, e.g., \cite{hawkins2025model}, our work provides a first
comprehensive analysis of the dynamics of this particular wave propagation model, which is of
central importance in seismic modeling.
The construction of a reduced model for the parametric counterpart
of the problem considered here---for instance, involving layered media with
piecewise-constant material properties following prescribed statistical
distributions---naturally builds upon the theoretical and computational tools developed in this work.

\subsection{Outline}
This work is structured as follows.
In Section~\ref{sec:model_problem}, we introduce the notation used throughout the paper,
together with the abstract wave equation posed in a Gelfand triple of Hilbert spaces.
We also introduce the Laplace transform and its bilateral counterpart, which are used
extensively in this work, in particular through the Paley--Wiener theorem.

In Section~\ref{sec:laplace_transform_MOR}, we present the proposed MOR technique based on
the Laplace transform, which, unlike the method introduced in \cite{HenriquezHesthaven2024},
relies on samples in the Laplace domain for a frequency-controlled problem.
Therein, we also discuss the computational implementation of the method and address
relevant implementation aspects.

Sections~\ref{sec:convergence_analysis} and \ref{sec:fully_discrete} are devoted to a
detailed analysis of the proposed MOR technique from
Section~\ref{sec:laplace_transform_MOR}.
In particular, in Section~\ref{sec:convergence_analysis} we thoroughly analyze the best
possible approximation of the abstract wave problem by finite-dimensional spaces,
whereas in Section~\ref{sec:fully_discrete} we provide a complete fully discrete error
analysis accounting for the sampling in the Laplace domain.

We conclude in Section~\ref{sec:numerical_results} with a comprehensive set of numerical
experiments illustrating the performance of the proposed method, and in
Section~\ref{sec:concluding_remarks} we provide final concluding remarks and discuss
possible extensions of the present work.

\section{Model Problem}
\label{sec:model_problem}
\subsection{Notation}
Set $\mathbb{R}_+=[0,\infty)$ and let $\mathcal{X}$ be a Banach space.
We denote by $L^1_{\text{loc}}(\mathbb{R}_+;\mathcal{X})$ the set of functions $f:\mathbb{R}_+\to\mathcal{X}$ such that $f$ is Bochner integrable on $[0,T]$ for any $T>0$.
For $\mu\geq 0$, we define $L^2_\mu(\mathbb{R}_+;\mathcal{X})$ 
as the space of all functions $f\in L^1_{\text{loc}}(\mathbb{R}_+;\mathcal{X})$ satisfying the condition
\begin{align}
	\|f\|_{L^2_\mu(\mathbb{R}_+;\mathcal{X})}^2 
	=
	\int_0^\infty e^{-2\mu t} \|f(t)\|_{\mathcal{X}}^2\,\textrm{d}t
	<
	\infty.
\end{align}

Let $\mathcal{V}$ be a real separable Hilbert space with inner product $(\cdot,\cdot)_\mathcal{V}$ and induced norm $\norm{\cdot}_{\mathcal{V}}$. 
We denote by $\mathcal{V}^\star$ its dual and by $\dual{\cdot}{\cdot}_{\mathcal{V}^\star \times \mathcal{V}}$ the $\mathcal{V}^{\star} \times \mathcal{V}$
duality pairing. 
In addition, we consider a (real) separable Hilbert space $\mathcal{H}$ equipped with the inner product $\dotp{\cdot}{\cdot}_{\mathcal{H}}$ and
induced norm $\norm{\cdot}_{\mathcal{H}}$.
Furthermore, we assume that $\mathcal{V}\subset\mathcal{H}$ is dense and continuously embedded.
Identifying $\mathcal{H}$ with its dual $\mathcal{H}^\star$, we have that $\mathcal{V}\subset\mathcal{H}\subset \mathcal{V}^\star$ is a Gelfand triple with pivot space $\mathcal{H}$.

Denoting $\partial_t^k f$ the $k$th distributional derivative of
$f\in L^1_{\text{loc}}(\mathbb{R}_+;\mathcal{H})$, we furthermore define the space
$\mathcal{W}_\mu(\mathbb{R}_+;\mathcal{H},\mathcal{V})$ consisting of functions $f\in L^2_\mu(\mathbb{R}_+;\mathcal{V})$
such that $\partial_t f \in L^2_\mu(\mathbb{R}_+;\mathcal{H})$ and
$\partial_t^2 f \in L^2_\mu(\mathbb{R}_+;\mathcal{V}^\star)$ and equip it with the norm
\begin{align*}
	\|f\|_{\mathcal{W}_\mu(\mathbb{R}_+;\mathcal{H},\mathcal{V})}^2 
	\coloneqq
	\|f\|_{L^2_\mu(\mathbb{R}_+;\mathcal{V})}^2
	+
	\|\partial_t f\|_{L^2_\mu(\mathbb{R}_+;\mathcal{H})}^2
	+
	\|\partial_t^2 f\|_{L^2_\mu(\mathbb{R}_+;\mathcal{V}^\star)}^2.
\end{align*}
By embedding, we have that functions $f\in \mathcal{W}_\mu(\mathbb{R}_+;\mathcal{H},\mathcal{V})$
are continuous with values in $\mathcal{H}$, i.e., $f\in \mathscr{C}(\overline{\mathfrak{I}};\mathcal{H})$ for each $T>0$, and $\partial_t f\in \mathscr{C}(\overline{\mathfrak{I}};\mathcal{V}^\star)$, for each $T>0$ and with $\mathfrak{I} \coloneqq (0,T)$.
\subsection{The Abstract Wave Equation}\label{sec:abstract_wave}
We consider the abstract linear, second-order evolution equation
\begin{align}
    \partial_t^2 u(t) + \mathsf{A} u(t) &= f(t), \quad t>0, \label{eq:wave}\\
    u(0) &= 0, \label{eq:ic1}\\
    \partial_t u(0) &= 0. \label{eq:ic2}
\end{align}
This formulation is obtained by replacing inhomogeneous initial conditions with a  problem involving homogeneous initial data and a time-dependent source term $f(t)=q(t)p$ with a given spatial source $p$. The time-dependent source is constructed using a Ricker wavelet
\begin{equation}\label{eq:ricker_wavelet}
     q(t)
     =
     \left(
        1-\frac{\alpha^2}{2}(t-t_0)^2
     \right)
     e^{
        - \frac{\alpha^2}{4}(t-t_0)^2
     },
\end{equation}
which provides a smooth and temporally localized representation of an impulsive excitation. By appropriately choosing its parameters $\alpha>0$ and $t_0>0$, the resulting forcing reproduces the effect of the prescribed initial displacement and velocity while avoiding singular behavior at the initial time. 
This reformulation allows the problem to be treated entirely within the class of second-order evolution equations with homogeneous initial conditions and structured right-hand sides.


To complete the description of \eqref{eq:wave}, let $\cV\subset\cH\subset\cV^\star$ be a Gelfand triple. Furthermore, we assume that $\mathsf{A}:\mathcal{V} \rightarrow \mathcal{V}^\star$ is a self-adjoint, coercive and bounded linear operator, i.e., it holds that 
$\dual{\mathsf{A}v}{w}_{\mathcal{V}^\star \times \mathcal{V}} = \dual{\mathsf{A}w}{v}_{\mathcal{V}^\star \times \mathcal{V}}$, $\forall v,w \in \mathcal{V}$, 
and there exists $ c_{\mathsf{A}}, \, C_{\mathsf{A}}>0$ such that
\begin{equation}\label{eq:A_continuous_coercive}
    c_{\mathsf{A}}
    \norm{v}^2_{\mathcal{V}}
    \leq \dual{\mathsf{A}v}{v}_{\mathcal{V}^\star \times \mathcal{V}},
    \quad
    \text{and}
    \quad
    \norm{\mathsf{A}v}_{\mathcal{V}^\star}
    \leq
    C_{\mathsf{A}}
    \norm{v}_{\mathcal{V}},
    \quad
    \forall 
    v \in \mathcal{V}.
\end{equation}

\begin{example}[Scalar Wave Equation]\label{example:scalar_wave}
We consider the scalar heterogeneous wave equation equipped with homogeneous Dirichlet boundary conditions in a bounded Lipschitz domain $\Omega \subset \IR^d$, $d \in \IN$, with boundary $\partial \Omega$. 
The physical properties of the medium, such as local stiffness, tension, or permeability, are characterized by the coefficient matrix ${\bf A}(\bx)$, $\bx\in\Omega$. 
We assume that ${\bf A}(\bx)$ is symmetric and uniformly positive definite, i.e., there exist constants $\underline{c}_{{\bf A}}, \overline{c}_{\bf A}>0$ such that 
\begin{equation}\label{eq:elliptic_c}
	\esssinf_{\bm{x} \in \Omega} 
	\boldsymbol\xi^\top 
	{\bf A}({\bm x}) 
	\boldsymbol\xi 
	\geq 
	\underline{c}_{\bm{A}}
    \norm{\boldsymbol\xi}^2_2, 
	\quad \forall \boldsymbol\xi \in \IR^{d} \backslash \{\bm{0}\},
\end{equation}	
and
\begin{equation}\label{eq:continuous_c}
	\esssup_{\bm{x}\in \Omega} 
	\norm{
		\bm{A}(\bm{x}) 
	}_2
	\leq
	\overline{c}_{\bm{A}}.
\end{equation}
The associated function spaces are $\mathcal{V} = H^1_0(\Omega)$, and $\mathcal{H} = L^2(\Omega)$.
The operator $\mathsf{A}:\cV\to \cV^\star$, is then given by
\begin{align*}
	\langle \mathsf{A}u,v\rangle_{\cV^\star\times \cV} 
	= 
	\int\limits_{\Omega}
    \left(
        {\bf A}({\bm x})
        \nabla u(\bm{x})
	\right)
    \cdot
	{\nabla v(\bm{x})}
	\text{d}\bm{x},
	\quad
    \forall
	u,v \in H^1_0(\Omega).
\end{align*}
\end{example}
\begin{example}[Elastic Wave Equation]\label{example:elastic_wave}
Modeling the propagation of waves in a linear elastic, isotropic medium finds applications in seismology \cite{fichtner2010full} and structural engineering \cite{Bigoni2020a}. In this setting, assuming a normalized unit density of the material, the wave speeds can be described by the Lamé parameters $\lambda, \mu \in L^\infty(\Omega)$, where, as in Example~\ref{example:scalar_wave}, $\Omega$ is a bounded domain with Lipschitz boundary.
Assuming again homogeneous Dirichlet conditions for the displacement vector $\mathbf{u}$ on $\partial\Omega$, we consider $\cH=L^2(\Omega)^d$, $d\in\{2,3\}$, and $\cV=H^1_0(\Omega)^d$. 
The operator $\mathsf{A}:\cV\to \cV^\star$, is then given by
\begin{align*}
	\langle \mathsf{A} \mathbf{u},\mathbf{v}\rangle_{\cV^\star\times \cV} 
	= 
	\int\limits_{\Omega}
     \lambda (\nabla \cdot \mathbf{u})(\nabla \cdot \mathbf{v}) + 2\mu \varepsilon(\mathbf{u}):\varepsilon(\mathbf{v})\,\mathrm{d}\bx
	\quad
	\mathbf{u},\, \mathbf{v} \in \cV,
\end{align*}
where $:$ denotes the Frobenius inner product for matrices, and $\varepsilon(\mathbf{v})=\frac{1}{2}(\nabla \mathbf{v}+ (\nabla\mathbf{v})^\top)$ is the infinitesimal strain tensor.
Using the assumption
\begin{equation}
    0<\mu_0<\mu({\bm x})< \mu_1<\infty
    \quad
    \text{and}
    \quad
    0<\lambda_0<\lambda({\bm x})< \lambda_1<\infty,
    \quad
    \text{ for a.e. }
    {\bm x} \in \Omega,
\end{equation}
for some constants $\mu_0,\mu_1,\lambda_0,\lambda_1>0$, one can verify that \eqref{eq:A_continuous_coercive} holds using Korn's inequality.

\end{example}

\subsection{Laplace and Fourier transforms}\label{sec:transforms}
Let $\cX$ be a {\it complex} Hilbert space.
The Laplace transform of $f\in L^1_{\text{loc}}(\IR_+;\cX)$
is defined as
\begin{equation} \label{eq:laplace}
	\cL
	\left\{f \right\}(s) 
	= 
	\int\limits_{0}^{\infty}f(t) e^{-st}\text{d}t,
	\quad
	s \in \mathbb{C}.
\end{equation}
We define the abscissa of convergence of $\cL\{f\}$ as
\begin{equation}
    \abs(f)
    \coloneqq
    \inf_{s \in \mathbb{C}}\{\Re(s): \cL\{f\}(s)\text{ exists}\}.
\end{equation}
Therefore, \eqref{eq:laplace} converges for $s\in\IC$ with $\Re(s)>\abs(f)$.
This motivates the definition of the half-planes 
\begin{align}
    \Pi_\lambda 
    \coloneqq 
    \left\{
        s\in\IC: \Re(s)>\lambda
    \right\},
\end{align}
where $\lambda\in\IR$, and we set as well $\Pi_+\coloneqq \Pi_0$.

Throughout, we also use the notation $\widehat{f}(s) = \cL\{f\}(s)$, for $s \in \Pi_{\abs(f)}$, to denote the Laplace transform of $f$.

We also introduce the  bilateral Laplace transform of a function $f\in L^1_{\text{loc}}(\IR;\cX)$,
\begin{align}\label{eq:bilateral_laplace}
    \cB\{f\}(s)
    \coloneqq
    \int_{-\infty}^\infty f(t) e^{-st} \text{d}t, \quad s\in\mathbb{C},
\end{align}
provided the integral exists. 
One can readily verify that
\begin{align}
	\cB\{f\}(\lambda+\imath\omega) 
    =
    \cF\left\{e^{-\lambda t}f\right\}(\omega),
\end{align}
where $\cF$ is the Fourier transformation, defined, for a square-integrable function, as
\begin{equation}
    \cF\{f\}(\omega)
    \coloneqq
    \int_{-\infty}^\infty f(t) e^{-\imath \omega t}\text{d}t.
\end{equation}

\subsection{Hardy spaces of analytic functions}

We introduce Hardy spaces of analytic functions 
following \cite[Chapter 4]{RR97} and \cite[Section 6.4]{hille1996functional}.

\begin{definition}[Hardy Spaces, {\cite[Definition 6.4.1]{hille1996functional}}]
\label{def:hardy_spaces}
Let $\mathcal{X}$ be a {\it complex} Banach space equipped
with the norm $\norm{\cdot}_{\mathcal{X}}$.
For $p\in[1,\infty)$ and $\mu \in \mathbb{R}$, we denote by
$\mathscr{H}^{p}_\mu(\mathcal{X})$ the set of all $\mathcal{X}$-valued 
functions $f:\Pi_\mu \rightarrow \mathcal{X}$ satisfying
the following properties:
\begin{itemize}
    \item[(i)] 
    The function $f:\Pi_\mu \rightarrow \mathcal{X}$ is holomorphic.
    \item[(ii)]
    It holds that
\begin{align}
	\norm{f}_{\mathscr{H}^{p}_\mu(\mathcal{X})}
	\coloneqq
	\sup_{\sigma>\mu}
	\left(
		\,
		\int\limits_{-\infty}^{+\infty}
		\norm{
			f(\sigma+\imath \tau)
		}^p_\mathcal{X}
		\frac{
		\normalfont\text{d}
		\tau}{2\pi}
	\right)^{\frac{1}{p}}
	<
	\infty.
\end{align}
\end{itemize}
Equipped with the norm $\norm{\cdot}_{\mathscr{H}^{p}_\mu(\mathcal{X})}$
the space $\mathscr{H}^{p}_\mu(\mathcal{X})$ is a Banach one.
\end{definition}

\begin{proposition}[{\cite[Theorem 6.4.3]{hille1996functional}}]
\label{prop:properties_hardy}
Let $p\in [1,\infty)$ and $\mu \in \mathbb{R}$. 
\begin{itemize}
\item[(i)]
For each $f \in \mathscr{H}^{p}_\mu(\mathcal{X})$ the function 
\begin{equation}
	T(\sigma,f)
	= 
	\int\limits_{-\infty}^{+\infty}
	\norm{
		f(\sigma+\imath \tau)
	}^p_\mathcal{X}
	\normalfont\text{d}
	\tau
\end{equation}
is a continuous, monotone decreasing function of $\sigma$,
for $\sigma \geq \mu$. 
In particular, 
\begin{equation}
    T(\mu,f) = \norm{f}^p_{\mathscr{H}^{p}_\mu(\mathcal{X})}
    \quad
    \text{and}
    \quad
    \displaystyle\lim_{\sigma \rightarrow \infty} T(\sigma,f)  = 0.
\end{equation}  

\item[(ii)]
For each $f \in \mathscr{H}^{p}_\mu(\mathcal{X})$
\begin{equation}
	\lim_{\sigma \rightarrow \mu}
	\int\limits_{-\infty}^{+\infty}
	\norm{
		f(\sigma+\imath \tau)
		-
		f(\mu+\imath \tau)
	}^p_\mathcal{X}
	\normalfont\text{d}
	\tau
	= 0.
\end{equation}
\end{itemize}
\end{proposition}


The following result is a Hilbert space-valued version of the 
Paley-Wiener representation theorem.

\begin{theorem}[Paley-Wiener Theorem, {\cite[Section 4.8, Theorem E]{RR97}}]
\label{thm:paley_wiener}
Let $\mathcal{X}$ be a {\it Hilbert} space and let $\mu \in \mathbb{R}$.
Then, the map 
$
	\mathcal{L}:
	L^2_\mu(\IR_+;\mathcal{X}) 
	\rightarrow 
	\mathscr{H}^{2}_\mu(\mathcal{X})
$
is an isometric isomorphism, i.e., 
\begin{equation}
    \mathcal{L} 
    \in 
    \mathscr{L}_{\normalfont\text{iso}}
    \left(
        L^2_\mu(\IR_+;\mathcal{X}), \mathscr{H}^{2}_\mu(\mathcal{X})
    \right),
\end{equation}
and for each $f \in L^2_\mu(\IR_+;\mathcal{X})$
\begin{equation}\label{eq:norm_equiv}
	\norm{f}_{L^2_\mu\left(\mathbb{R}_+;\mathcal{X}\right)} 
	= 
	\norm{
		\mathcal{L}\{f\}
	}_{\mathscr{H}^{2}_\mu(\mathcal{X})}.
\end{equation}
\end{theorem}

\section{Model Order Reduction using the Laplace Transform}
\label{sec:laplace_transform_MOR}
In this section, we present the MOR algorithm based on the Laplace transform 
and describe its practical computational implementation.

\subsection{LT-MOR Algorithm}

We consider the problem of accurately approximating the solution of the second-order evolution equation \eqref{eq:wave}--\eqref{eq:ic2} in a subspace $\cV_{R,M}^\rb\subset\cV$ of dimension $R$. The subscript $M\in\IN$ with $M\geq R$ indicates that
$\cV_{R,M}^\rb$ is computed from $M$ suitable snapshots.
Once $\cV_{R,M}^\rb$ has been constructed, we use a Galerkin approximation of \eqref{eq:wave}--\eqref{eq:ic2} to find an approximation $u_{R,M}^\rb:[0,T]\to \cV_{R,M}^\rb$ of the solution $u$ as follows.

\begin{problem}[Reduced semi-discrete problem] \label{pr:sdpr}
We seek $u^{\normalfont\textrm{(rb)}}_{R,M} \in \mathcal{W}_\mu\left(\mathbb{R}_+;\mathcal{H},\cV_{R,M}^\rb\right)$, for some $\mu>0$, such that for all $t>0$ it holds that
\begin{equation}\label{eq:semi_discrete_rom}
\begin{aligned}
	\dotp{
		\partial^2_t
		u_{R,M}^{\normalfont\textrm{(rb)}}(t)
	}{
		v_{R,M}^{\normalfont\textrm{(rb)}}
	}_{\mathcal{H}}
	+ 
    \dual{\sfA u_{R,M}^\rb(t)}{v_{R,M}^\rb}_{\cV^\star\times\cV} = \dotp{f(t)}{v_{R,M}^\rb}_\cH
	\quad
	\forall v_{R,M}^{\normalfont\textrm{(rb)}} \in \cV_{R,M}^\rb,
\end{aligned}
\end{equation}
equipped with vanishing initial conditions, i.e., 
$u_{R,M}^\rb(0)=0$ and $\partial_t u_{R,M}^\rb(0)=0$. 
\end{problem}

We note that, since $\cV_{R,M}^\rb$ is finite dimensional, \eqref{eq:semi_discrete_rom} amounts to a linear ordinary differential equation, which has a unique solution by the Picard-Lindel\"of theorem.

For the construction of a suitable approximation space $\cV_{R,M}^\rb$, we consider an auxiliary problem in the Laplace domain that corresponds to a lossy Helmholtz equation
with a right-hand side.

\begin{problem}[Frequency-controlled auxiliary problem]\label{pbrm:laplace_discrete}
Let $\mu>0$.
For each $s \in \Pi_\mu$,
we seek $\widehat{U}^{(2)}(s) \in \mathcal{V}$ such that\footnote{The superscript is used to highlight that we target the second derivative.}
\begin{equation}
    s^2  \widehat{U}^{(2)}(s) + \sfA \widehat{U}^{(2)}(s)= \cB\{\partial^2_t q\}(s)\, p.
\end{equation}
\end{problem}
We emphasize that the source term in Problem~\ref{pbrm:laplace_discrete} corresponds to the bilateral Laplace transform of the second temporal derivative of the source term $f$.
%
Below, we derive an error bound for \( \widehat{U}^{(2)}(s) - \cL\{\partial_t^2 u\}(s) \), which shows that this difference becomes small for suitable choices of the parameters \( \alpha \) and \( t_0 \) that determine the shape and location of the Ricker wavelet \( q \). Moreover, we show that the rapid decay of \( \cB\{q\}(\mu + \imath \tau) \) for fixed \( \mu > 0 \) and as \( |\tau| \to \infty \) implies the corresponding decay of \( \widehat{U}^{(2)}(\mu + \imath \tau) \) as \( |\tau| \to \infty \).
As a consequence, the smooth function \( \tau \mapsto \widehat{U}^{(2)}(\mu + \imath \tau) \) can be efficiently approximated by sampling at a finite set of points
\begin{equation}
    \mathcal{P}
    \coloneqq\{ s_j \coloneqq  \mu + \imath \tau_j : j = 1, \ldots, M \},
    \quad M \in \IN.
\end{equation}
To achieve further compression, we construct the space $\mathcal{V}^{\text{(rb)}}_{R,M}$ as follows
\begin{equation}\label{eq:fPOD}
	\mathcal{V}^{\text{(rb)}}_{R,M}
	=
	\argmin_{
	\substack{
		\mathcal{V}_R \subset \mathcal{V}
		\\
		\text{dim}(\mathcal{V}_R)\leq R	
	}
	}
    	\sum_{j=1}^{M} 
	\omega_j
	\norm{
		\Re
		\left\{
			\widehat{U}^{(2)}(s_j) 
		\right\}
		- 
		\mathsf{P}_{\mathcal{V}_R}
		\Re
		\left\{
			\widehat{U}^{(2)}(s_j) 
		\right\}
	}^2_{\cV},
\end{equation}
where $\{\omega_1,\ldots,\omega_{M}\}$ are (strictly positive) weights
and $\mathsf{P}_{\mathcal{V}_R}: \mathcal{V} \rightarrow \mathcal{V}_R$
denotes the $\mathcal{V}$-orthogonal projection onto $\mathcal{V}_R$.

\begin{remark}
The reason why we only consider the real part of the snapshots
$\left\{\widehat{U}^{(2)}(s_1) ,\dots, \widehat{U}^{(2)}(s_M)  \right\}$ in \eqref{eq:fPOD} is thoroughly discussed
in \cite[Section 5.1]{henriquez2024fast}.
\end{remark}

\subsection{Practical Implementation}
In order to realize the proposed method, in the following, let $\mathcal{V}_h \subset \mathcal{V}$ be a finite-dimensional Galerkin discretization space of dimension $N_h$. Consider a basis
$\{\varphi_{j,h}\}_{j=1}^{N_h}$ for $\mathcal{V}_h$. The minimization problem stated in
\eqref{eq:fPOD} admits the following algebraic form:
\begin{equation}\label{eq:tPOD_algebraic_frequency}
	\bm{\Phi}^{\textrm{(rb)}}_{R,M}
	=
	\arg\min_{\bm\Phi \in \mathscr{V}_{R}}
	\sum_{j=1}^{M}
	\omega_j
	\norm{
		\Re\{\widehat{\bm{\mathsf{U}}}_h^{(2)}(s_j)\}
		-
		\bm{\Phi}\bm{\Phi}^\top
		\bm{\mathsf{B}}_h
		\Re\{\widehat{\bm{\mathsf{U}}}_h^{(2)}(s_j)\}
	}^2_{\bm{\mathsf{B}}_h},
\end{equation}
where
\begin{equation}\tag*{}
	\norm{\bm{\mathsf{v}}}_{\bm{\mathsf{B}}_h}
	=
	\sqrt{
		\left(\bm{\mathsf{B}}_h\bm{\mathsf{v}},\bm{\mathsf{v}}\right)_{\IC^{N_h}}
	},
	\qquad
	\bm{\mathsf{v}} \in \IC^{N_h},
\end{equation}
and
\begin{equation}
	\mathscr{V}_{R}
	\coloneqq
	\left\{
		\bm{\Phi} \in \IR^{N_h \times R}:
		\bm{\Phi}^\top \bm{\mathsf{B}}_h \bm{\Phi} = \bm{\mathsf{I}}_{R}
	\right\}.
\end{equation}
Here, we identify each $\widehat{U}^{(2)}_h(s)\in \mathcal{V}_h$, i.e., the Galerkin solution to
Problem~\ref{pbrm:laplace_discrete} in $\mathcal{V}_h$, with its vector of degrees of freedom
$\widehat{\bm{\mathsf{U}}}_h^{(2)}(s) \in \IC^{N_h}$ in the basis $\{\varphi_{j,h}\}_{j=1}^{N_h}$.

The connection between the solution to \eqref{eq:fPOD} and \eqref{eq:tPOD_algebraic_frequency}
is as follows. Set
\begin{equation}\label{eq:def_basis_rb}
	\varphi^{\textrm{(rb)}}_{k,M}
	\coloneqq
	\sum_{j=1}^{N_h}
	\left( \bm{\phi}^{\textrm{(rb)}}_{k,M} \right)_j \,\varphi_{j,h}
	\in \mathcal{V},
	\qquad
	k=1,\dots,R,
\end{equation}
where
\begin{equation}\tag*{}
	\bm{\Phi}^{\textrm{(rb)}}_{R,M}
	=
	\left(
		\bm{\phi}^{\textrm{(rb)}}_{1,M},
		\dots,
		\bm{\phi}^{\textrm{(rb)}}_{R,M}
	\right),
\end{equation}
and $\left( \bm{\phi}^{\textrm{(rb)}}_{k,M} \right)_j$ signifies the $j$-th component of
$\bm{\phi}^{\textrm{(rb)}}_k$. Then $\left\{\varphi^{\textrm{(rb)}}_{1,M},\ldots,\varphi^{\textrm{(rb)}}_{R,M}\right\}$
is an orthonormal basis of $\mathcal{V}_{R,M}^{(\textrm{rb})}$ in the $\mathcal{V}$-inner product.
Set
\begin{equation}\label{eq:snapshot_M_LT_RB}
	\bm{\mathsf{S}}
	\coloneqq
	\left(
		\Re\{ \widehat{\bm{\mathsf{U}}}_h(s_1)\},
		\dots,
		\Re\{ \widehat{\bm{\mathsf{U}}}_h(s_{M})\}
	\right)
	\in \IR^{N_h \times M},
\end{equation}
and
\begin{equation}\tag*{}
	\bm{\mathsf{D}}
	\coloneqq
	\text{diag}\left(\omega_1, \ldots, \omega_{M} \right)
	\in \IR^{M \times M}.
\end{equation}

Let $\bm{\mathsf{B}}_h = \bm{\mathsf{R}}_h^\top \bm{\mathsf{R}}_h$ be the Cholesky
decomposition of $\bm{\mathsf{B}}_h$, where $\bm{\mathsf{R}}_h$ is an upper triangular matrix.
Furthermore, consider
\begin{equation}\tag*{}
	\widetilde{\bm{\mathsf{S}}}
	=
	\bm{\mathsf{R}}_h\,\bm{\mathsf{S}}\,\bm{\mathsf{D}}^{\frac{1}{2}}
\end{equation}
together with its singular value decomposition (SVD)
$\widetilde{\bm{\mathsf{S}}} = \widetilde{\bm{\mathsf{U}}}\widetilde{\Sigma}\widetilde{\bm{\mathsf{V}}}^\top$,
where
\begin{equation}
	\widetilde{\bm{\mathsf{U}}}
	=
	\left(
		\widetilde{\boldsymbol{\zeta}}_1,
		\ldots,
		\widetilde{\boldsymbol{\zeta}}_{N_h}
	\right)\in \IR^{N_h \times N_h},
	\quad\text{and}\quad
	\widetilde{\bm{\mathsf{V}}}
	=
	\left(
		\widetilde{\boldsymbol{\psi}}_1,
		\ldots,
		\widetilde{\boldsymbol{\psi}}_{M}
	\right)\in \IR^{M \times M}
\end{equation}
are orthogonal matrices, referred to as the left and right singular vectors of
$\widetilde{\bm{\mathsf{S}}}$, respectively, and
\begin{equation}\tag*{}
	\widetilde{\Sigma}
	=
	\operatorname{diag}\left(\widetilde\sigma_1, \ldots, \widetilde\sigma_r,0,\ldots,0\right)
	\in \IR^{N_h \times M},
	\qquad
	\widetilde\sigma_1 \geq \cdots \geq \widetilde\sigma_r>0,
\end{equation}
where $r = \operatorname{rank}(\widetilde{\bm{\mathsf{S}}})\leq  \min\{N_h,M\}$.

According to the Schmidt--Eckart--Young theorem (see, e.g.,
\cite[Proposition~6.2]{quarteroni2015reduced}), for any $R\leq r$ the solution
$\bm{\Phi}^{\textrm{(rb)}}_R$ to \eqref{eq:tPOD_algebraic_frequency} is given by
\begin{equation}\label{eq:reduced_basis_Phi}
	\bm{\Phi}^{\textrm{(rb)}}_{R,M}
	=
	\left(
		\bm{\mathsf{R}}^{-1}_h\widetilde{\boldsymbol{\zeta}}_1,
		\ldots,
		\bm{\mathsf{R}}^{-1}_h\widetilde{\boldsymbol{\zeta}}_{R}
	\right),
\end{equation}
which consists of the first $R$ left singular vectors of $\widetilde{\bm{\mathsf{S}}}$
multiplied from the left by $\bm{\mathsf{R}}^{-1}_h$. Moreover,
\begin{equation}\label{eq:error_SVD}
\begin{aligned}
	\min_{\bm{\Phi} \in \mathscr{V}_{R}}
	\sum_{j=1}^{M}
	\omega_j
	\norm{
		\Re\left\{\widehat{\bm{\mathsf{U}}}_h^{(2)}(s_j)\right\}
		-
		\bm{\Phi}\bm{\Phi}^\top
		\bm{\mathsf{B}}_h
		\Re\left\{\widehat{\bm{\mathsf{U}}}_h^{(2)}(s_j)\right\}
	}^2_{\bm{\mathsf{B}}_h}
	&=
	\sum_{j=R+1}^r \widetilde{\sigma}_j^2.
\end{aligned}
\end{equation}

\begin{remark}
We emphasize that the presented results do not depend on the dimension of $\cV_h$. Also, it is easy to check that all results hold true if $\cV$ is replaced by $\cV_h$.
\end{remark}

\section{Convergence Analysis}
\label{sec:convergence_analysis}
We present a complete analysis of the LT-MOR method tailored for 
the setting described in Section~\ref{sec:abstract_wave}.
To this end, we discuss properties of the evolution problem \eqref{eq:wave}--\eqref{eq:ic2} with more general initial conditions and source terms than the one in Section~\ref{ssec:auxiliary_L_probl}.
Next, in Section~\ref{ssec:sinc}, we recall important results concerning sinc interpolation
and quadrature.
In Section~\ref{ssec:low_rank}, we quantify the approximation error of $\partial_t^2u$
using finite-dimensional subspaces of dimension $R$, which is based on the combination of three key ingredients: 
\begin{itemize}
    \item[(i)]
    A bound for the consistency error between the Laplace transform of $\partial_t^2 u$ and the solution $\widehat{U}^{(2)}$ to Problem~\ref{pbrm:laplace_discrete}.
    \item[(ii)]
    The behavior of the solution $\widehat{U}^{(2)}(s)$ for large $\Im(s)$.
    \item[(iii)]
    The sinc interpolation error bound recalled ahead in Section~\ref{ssec:sinc}.
\end{itemize}
Finally, in Section~\ref{ssec:exponential_convergence} we present an error bound for approximating $u$ with respect to the $L^2_\mu(\IR_+;\cV)$ norm in the best possible linear space. This bound consists of a term that decays exponentially in $R$ and a consistency error that is related to the behavior of the Ricker wavelet for $t<0$.

\subsection{Auxiliary Laplace domain problem}\label{ssec:auxiliary_L_probl}
We apply the Laplace transform to the evolution problem \eqref{eq:wave}--\eqref{eq:ic2}.
Since the Laplace transform employs complex Hilbert spaces, we understand with slight abuse of notation $\cH$ and $\cV$ as their corresponding complexifications.

For the sake of convenience in our analysis, we consider the following auxiliary problem in the 
Laplace domain with general data.

\begin{problem}[Laplace Domain Auxiliary Problem]\label{pbrm:laplace}
Let $\mu>0$, $\phi_0 \in \mathcal{V}$, $\phi_1 \in \mathcal{H}$, and
$f \in L^2_\mu(\mathbb{R}_+;\mathcal{H})$.
For each $s \in \Pi_\mu$, we seek $\widehat{\phi}(s) \in \cV$ such that
\begin{equation}\label{eq:laplace_weak}
	(s^2 + \mathsf{A})\widehat{\phi}(s)
	=
    \widehat{f}(s) + s \phi_0 + \phi_1,
\end{equation}
where $\widehat{f}(s) = \mathcal{L}\{f\}(s) \in \mathscr{H}^2_\mu(\mathcal{H})$,
for $s \in \Pi_\mu$, (cf. Theorem~\ref{thm:paley_wiener}, Paley-Wiener).
\end{problem}


To analyze Problem~\ref{pbrm:laplace}, we define the following norm
\begin{equation}
	\norm{v}_{\varrho}
	\coloneq
	\sqrt{
		\varrho^2
		\norm{v}^2_{\mathcal{H}}
		+
		\norm{v}^2_{\mathcal{V}}
	},
	\quad
	v \in \mathcal{V},
\end{equation}
for a real parameter $\varrho>0$,
and the corresponding induced norm for the operator $\mathsf{B}:\mathcal{V} \rightarrow \mathcal{V}^\star $
\begin{equation}\label{eq:operator_norm_rho}
	\norm{\mathsf{B}}_{\varrho}
	\coloneqq
	\sup_{w,v \in \mathcal{V} \backslash \{0\}}
	\frac{
		\snorm{\dual{\mathsf{B}w}{v}_{\mathcal{V}^\star \times \mathcal{V}}}
	}{   
		\norm{w}_{\varrho}\norm{v}_{\varrho}
	},
	\quad
	\mathsf{B} \in \mathscr{L}(\mathcal{V},\mathcal{V}^\star).
\end{equation}
We note that in either case $\norm{\cdot}_\varrho$ is equivalent to $\norm{\cdot}_{\mathcal{V}}$ and 
$\norm{\cdot}_{\mathscr{L}(\mathcal{V},\mathcal{V}^\star)}$ respectively.
 
We are interested in bounding the inf-sup constant
\begin{equation}
	\gamma(s)
	\coloneqq
	\inf_{w \in \mathcal{V} \backslash \{0\}}
	\sup_{v\in  \mathcal{V} \backslash \{0\}}
	\frac{
		\snorm{
			\dual{(s^2+\mathsf{A})w}{v}_{\mathcal{V}^\star \times \mathcal{V}}
		}
	}{
		\norm{w}_{\snorm{s}}
		\norm{v}_{\snorm{s}}
	}
\end{equation}
explicitly in terms of $s \in \Pi_\mu$.
As in \cite{MST20}, one can show the following result.

\begin{lemma}\label{lmm:inf_sup}
Let $\mu>0$ be given. For each $s \in \Pi_\mu$, one has
\begin{equation}
	\gamma(s)
	\geq
	\frac{
		{\Re\{s\}}
	}{
		\snorm{s} 	}
    \min\{1, c_{\mathsf{A}}\}
    >0,
\end{equation}
where $c_{\mathsf{A}}$ is the coercivity constant from \eqref{eq:A_continuous_coercive}.
\end{lemma}
\begin{proof}
We follow the steps of the proof given in \cite{MST20}, which in turn uses that of \cite{HBN86},
for convenience of the reader.
Let $w\in\cV$ be given and set $v=\frac{s}{|s|}w$. Observe that
\begin{equation}
\begin{aligned}
	\snorm{\dual{(s^2+\mathsf{A})w}{v}_{\mathcal{V}^\star \times \mathcal{V}}}
	&
	\geq
	\snorm{\Re\left\{\dual{(s^2+\mathsf{A}) w}{v}_{\mathcal{V}^\star \times \mathcal{V}}\right\}}
	\\
	&
	=
	\frac{{\Re\{s\}}}{|s|}\left(|s|^2\norm{w}_\cH^2+\Re\dual{\mathsf{A}w}{w}_{\cV^\star\times\cV}\right),
\end{aligned}
\end{equation}
which proves the claim by \eqref{eq:A_continuous_coercive}.
\end{proof}

\begin{proposition}\label{prop:stability_laplace_domain_problem}
Let $\mu>0$ and assume that $\phi_0,\phi_1\in\cV$ and $\mathsf{A}\phi_0, \mathsf{A}\phi_1\in\cH$.
Then, for each $s \in \Pi_{\mu}$, there exists a unique 
$\widehat{\phi}(s) \in \mathcal{V}$ solution to Problem 
\ref{pbrm:laplace} satisfying
\begin{equation}\label{eq:bound_laplace_1}
	\norm{\widehat{\phi}(s)}_{\mathcal{V}}
	\leq
	\frac{\norm{\phi_0}_{\mathcal{V}}}{\snorm{s}}
	+
	\frac{\norm{\phi_1}_{\mathcal{V}}}{\snorm{s}^2}
	+
	\frac{
		1
	}{
		\Re\{s\} \min\{1, c_{\mathsf{A}}\}
	}
	\left(
		\norm{\widehat{f}(s)}_{\mathcal{H}}
		+
		\frac{\norm{\mathsf{A}\phi_0}_{\mathcal{H}}}{\snorm{s}}
		+
		\frac{\norm{\mathsf{A}\phi_1}_{\mathcal{H}}}{\snorm{s}^2}
	\right)
\end{equation}
and
\begin{equation}\label{eq:bound_laplace_2}
    \norm{s \widehat{\phi}(s) -\phi_0}_{\mathcal{H}}
	\leq
	\frac{\norm{\phi_1}_{\mathcal{H}}}{\snorm{s}}
	+
	\frac{
		1
	}{
		\Re\{s\}  \min\{1, c_{\mathsf{A}}\}
	}
	\left(
		\norm{\widehat{f}(s)}_{\mathcal{H}}
		+
		\frac{\norm{\mathsf{A}\phi_0}_{\mathcal{H}}}{\snorm{s}}
		+
		\frac{\norm{\mathsf{A}\phi_1}_{\mathcal{H}}}{\snorm{s}^2}
	\right)
\end{equation}
together with
\begin{equation}\label{eq:bound_laplace_3}
\begin{aligned}
    \norm{s^2 \widehat{\phi}(s) - s \phi_0 -\phi_1}_{\mathcal{V}^\star}
	\leq
	&
	\norm{\widehat{f}(s)}_{\mathcal{V}^\star}
	+ 
	C_{\mathsf{A}}
	\left(
    		\frac{\norm{\phi_0}_{\mathcal{V}}}{\snorm{s}} 
    		+ 
    		\frac{\norm{\phi_1}_{\mathcal{V}}}{\snorm{s}^2} 
	\right)
	\\
	&
	+
	\frac{C_{\mathsf{A}}}{{\Re\{s\}}\min\{1, c_{\mathsf{A}}\}}
	\left(
		\norm{\widehat{f}(s)}_\mathcal{H} 
		+ 
		\frac{\norm{\mathsf{A} \phi_0}_\mathcal{H}}{\snorm{s}} 
		+ 
		\frac{\norm{\mathsf{A} \phi_1}_\mathcal{H}}{\snorm{s}^2} 
	\right).
\end{aligned}
\end{equation}
Furthermore, the map $\Pi_\mu \ni s \mapsto \widehat{\phi}(s) \in \mathcal{V}$
is holomorphic.
\end{proposition}

\begin{proof}
The operator $s^2 +\mathsf{A}: \mathcal{V} \rightarrow \mathcal{V}^\star$
is continuous according to
\begin{equation}
	\norm{s^2 +\mathsf{A}}_{\snorm{s}}
	\leq
	\max \{1,C_{\mathsf{A}} \},
\end{equation}
where the operator norm $\norm{\cdot}_{\snorm{s}}$ has been defined
in \eqref{eq:operator_norm_rho} with $\rho=\snorm{s}$, and with $C_{\mathsf{A}}$ as in \eqref{eq:A_continuous_coercive}.
It follows, also due to Lemma~\ref{lmm:inf_sup} and \cite[Thm. 2.1.44]{sauter2010boundary}
for each $s\in \Pi_\mu$, that there exists a unique
$\widehat{\phi}(s) \in \mathcal{V}$ solution to Problem~\ref{pbrm:laplace} satisfying
\begin{equation}
\begin{aligned}
	\norm{\widehat{\phi}(s)}_{\snorm{s}}
	&
	\leq
	\frac{\snorm{s}}{{\Re\{s\}}\min\{1, c_{\mathsf{A}}\}}
	\sup_{v \in \mathcal{V} \backslash \{0\}}
	\frac{\snorm{\dual{\widehat{f}(s) + s \phi_0 + \phi_1}{v}_{\mathcal{V}^\star \times \mathcal{V}}}}{\norm{v}_{\snorm{s}}}
	\\
	&
	\leq
	\frac{1}{{\Re\{s\}}\min\{1, c_{\mathsf{A}}\}}
	\left(
		\norm{\widehat{f}(s)}_\mathcal{H} + \snorm{s}\norm{\phi_0}_\mathcal{H} + \norm{\phi_1}_\mathcal{H}
	\right).
\end{aligned}
\end{equation}
Next, recalling that we have assumed $\phi_0,\phi_1 \in \mathcal{V}$, we define for each $s \in \Pi_\mu$
\begin{equation}\label{eq:eqaution_W_s}
	\widehat{w}(s) \coloneqq \widehat{\phi}(s) - \frac{\phi_0}{s} -  \frac{\phi_1}{s^2} \in \mathcal{V},
\end{equation}
which is solution to
\begin{equation}
	(s^2 + \mathsf{A})\widehat{w}(s)
	=
    \widehat{f}(s) - \frac{ \mathsf{A} \phi_0}{s} - \frac{ \mathsf{A} \phi_1}{s^2}
    \in \mathcal{H}.
\end{equation}
Recalling Lemma~\ref{lmm:inf_sup}, we obtain
\begin{equation}\label{eq:bound_W_s}
\begin{aligned}
	\norm{\widehat{w}(s)}_{\snorm{s}}
	&
	\leq
	\frac{\snorm{s}}{{\Re\{s\}}\min\{1, c_{\mathsf{A}}\}}
	\sup_{v \in \mathcal{V} \backslash \{0\}}
	\frac{\snorm{\dual{\widehat{f}(s) - \frac{ \mathsf{A} \phi_0}{s} - \frac{ \mathsf{A} \phi_1}{s^2}}{v}_{\mathcal{V}^\star \times \mathcal{V}}}}{\norm{v}_{\snorm{v}}}
	\\
	&
	\leq
	\frac{1}{{\Re\{s\}}\min\{1, c_{\mathsf{A}}\}}
	\left(
		\norm{ \widehat{f}(s) }_\mathcal{H} + \frac{\norm{\mathsf{A} \phi_0}_\mathcal{H}}{\snorm{s}} + \frac{\norm{\mathsf{A} \phi_1}_\mathcal{H}}{\snorm{s}^2} 
	\right),
\end{aligned}
\end{equation}
which combined with \eqref{eq:eqaution_W_s} yields \eqref{eq:bound_laplace_1}.

Next, we calculate
\begin{equation}
	\norm{s \widehat{\phi}(s) -\phi_0}_{\mathcal{H}}
	\leq
	\snorm{s}
	\norm{\widehat{w}(s)}_{\mathcal{H}}
	+
	\frac{
		\norm{\phi_1}_\mathcal{H}
	}{
		\snorm{s}
	},
\end{equation}
which combined with \eqref{eq:eqaution_W_s} yields \eqref{eq:bound_laplace_2}. Finally, we have that 
\begin{equation}
	\norm{s^2 \widehat{\phi}(s) - s \phi_0 -\phi_1}_{\mathcal{V}^\star}
	=
	\norm{s^2 \widehat{w}(s)}_{\mathcal{V}^\star}
	=
	\sup_{v \in \mathcal{V} \backslash \{0\}}
	\frac{
		\snorm{\dual{s^2 \widehat{w}}{v}_{\mathcal{V}^\star \times \mathcal{V}}}
	}{
		\norm{v}_\mathcal{V}
	},
\end{equation}
thus using \eqref{eq:eqaution_W_s} we obtain
\begin{equation}
\begin{aligned}
	\norm{s^2 \widehat{\phi}(s) - s \phi_0 -\phi_1}_{\mathcal{V}^\star}
	&
	=
	\sup_{v \in \mathcal{V} \backslash \{0\}}
	\frac{
		\snorm{\dual{\widehat{f}(s) - \frac{ \mathsf{A} \phi_0}{s} - \frac{ \mathsf{A} \phi_1}{s^2} - \mathsf{A}\widehat{w}(s) }{v}_{\mathcal{V}^\star \times \mathcal{V}}}
	}{
		\norm{v}_\mathcal{V}
	}
	\\
	&
	\leq
	\norm{\widehat{f}(s)}_{\mathcal{V}^\star}
	+ 
	\frac{\norm{\mathsf{A} \phi_0}_{\mathcal{V}^\star}}{\snorm{s}} 
	+ 
	\frac{\norm{\mathsf{A} \phi_1}_{\mathcal{V}^\star}}{\snorm{s}^2} 
	+
	\norm{\mathsf{A}\widehat{w}(s)}_{\mathcal{V}^\star}
	\\
	&
	\leq
	\norm{ \widehat{f}(s)}_{\mathcal{V}^\star}
	+ 
	C_{\mathsf{A}}
	\left(
    		\frac{\norm{\phi_0}_{\mathcal{V}}}{\snorm{s}} 
    		+ 
    		\frac{\norm{\phi_1}_{\mathcal{V}}}{\snorm{s}^2} 
    		+
    		\norm{\widehat{w}(s)}_{\mathcal{V}}
	\right),
\end{aligned}
\end{equation}
which, together with \eqref{eq:bound_W_s}, yields \eqref{eq:bound_laplace_3}.
\end{proof}
The map $\Pi_\mu \ni s \mapsto s^2 +\mathsf{A} \in \mathscr{L}(\mathcal{V},\mathcal{V}^\star)$ is holomorphic,
where $\mathscr{L}(\mathcal{V},\mathcal{V}^\star)$ corresponds to the Banach space of bounded linear operators
from $\mathcal{V}$ to $\mathcal{V}^\star$ equipped with the standard operator norm 
\begin{equation}
	\norm{\mathsf{B}}_{\mathscr{L}(\mathcal{V},\mathcal{V}^\star).}
	\coloneqq
	\sup_{w,v \in \mathcal{V} \backslash \{0\}}
	\frac{
		\snorm{\dual{\mathsf{B}w}{v}_{\mathcal{V}^\star \times \mathcal{V}}}
	}{   
		\norm{w}_{\mathcal{V}}\norm{v}_{\mathcal{V}}
	},
	\quad
	\mathsf{B} \in \mathscr{L}(\mathcal{V},\mathcal{V}^\star).
\end{equation}	
The map
\begin{equation}
    \mathscr{L}_{\text{iso}}(\mathcal{V},\mathcal{V}^\star) 
    \subset\mathscr{L}(\mathcal{V},\mathcal{V}^\star)\rightarrow \mathscr{L}(\mathcal{V},\mathcal{V}^\star): \mathsf{B} \mapsto \mathsf{B}^{-1}
\end{equation}
is holomorphic as stated in \cite[Proposition 4.20]{henriquez2021shape}, where $\mathscr{L}_{\text{iso}}(\mathcal{V},\mathcal{V}^\star) \subset \mathscr{L}(\mathcal{V},\mathcal{V}^\star)$ corresponds to the {\it open} subset of bounded linear operators with a bounded inverse. 
Furthermore, as a consequence of Theorem~\ref{thm:paley_wiener} (Paley-Wiener), 
$\Pi_\mu \ni s \mapsto \mathcal{L}\{f\}(s) \in \mathcal{H}$ is analytic as well. Therefore, we may conclude that 
$\Pi_\mu \ni s \mapsto \widehat{\phi}(s) \in \mathcal{V}$ is holomorphic as the composition of holomorphic maps. 

As a direct consequence of Proposition~\ref{prop:stability_laplace_domain_problem} and Theorem~\ref{thm:paley_wiener}, we obtain the following result. 

\begin{corollary}\label{cor:hardy_uniqueness_time}
Let $\mu>0$ and assume that $\phi_0,\phi_1\in\cV$ and $\mathsf{A}\phi_0, \mathsf{A}\phi_1\in\cH$. Then, 
\begin{equation}
    \widehat{\phi}(s) \in \mathscr{H}^2_\mu(\mathcal{V}),
    \quad
    s\widehat{\phi}(s) - \phi_0\in \mathscr{H}^2_\mu(\mathcal{H}),
    \quad
    \text{and}
    \quad
    s^2\widehat{\phi}(s) - s\phi_0 -\phi_1\in \mathscr{H}^2_\mu(\mathcal{V}^\star),
\end{equation}
and
\begin{equation}
\begin{aligned}
    \norm{\widehat{\phi}(s)}_{ \mathscr{H}^2_\mu(\mathcal{V})}
    &
    \lesssim
   	\frac{\norm{\phi_0}_{\mathcal{V}}}{\sqrt{\mu}}
	+
	\frac{\norm{\phi_1}_{\mathcal{V}}}{\sqrt{\mu^3}}
	+
	\frac{
		1
	}{
		\mu\min\{1, c_{\mathsf{A}}\}
	}
	\left(
		\norm{\widehat{f}}_{\mathscr{H}^2_\mu(\mathcal{H})}
		+
		\frac{\norm{\mathsf{A}\phi_0}_{\mathcal{H}}}{\sqrt{\mu}}
		+
		\frac{\norm{\mathsf{A}\phi_1}_{\mathcal{H}}}{\sqrt{\mu^3}}
	\right)
    \\
    \norm{s\widehat{\phi}(s) -\phi_0}_{\mathscr{H}^2_\mu(\mathcal{H})}
    &
    \lesssim
   	\frac{\norm{\phi_1}_{\mathcal{V}}}{\sqrt{\mu^3}}
	+
	\frac{
		1
	}{
		\mu \min\{1, c_{\mathsf{A}}\}
	}
	\left(
		\norm{\widehat{f}}_{\mathscr{H}^2_\mu(\mathcal{H})}
		+
		\frac{\norm{\mathsf{A}\phi_0}_{\mathcal{H}}}{\sqrt{\mu}}
		+
		\frac{\norm{\mathsf{A}\phi_1}_{\mathcal{H}}}{\sqrt{\mu^3}}
	\right)
    \\
    \norm{s^2\widehat{\phi}(s) - s\phi_0 -\phi_1}_{\mathscr{H}^2_\mu(\mathcal{V}^\star)}
    &
    \lesssim
   	\norm{\widehat{f}}_{{\mathscr{H}^2_\mu(\mathcal{V}^\star)}}
	+ 
	C_{\mathsf{A}}
	\left(
   		\frac{\norm{\phi_0}_{\mathcal{V}}}{\sqrt{\mu}}
		+
		\frac{\norm{\phi_1}_{\mathcal{V}}}{\sqrt{\mu^3}}
	\right)
	\\
	&
	+
	\frac{C_{\mathsf{A}}}{{\mu}\min\{1, c_{\mathsf{A}}\}}
	\left(
		\norm{\widehat{f}}_{{\mathscr{H}^2_\mu(\mathcal{H})}} 
		+ 
		\frac{\norm{\mathsf{A}\phi_0}_{\mathcal{H}}}{\sqrt{\mu}}
		+
		\frac{\norm{\mathsf{A}\phi_1}_{\mathcal{H}}}{\sqrt{\mu^3}}
	\right),
\end{aligned}
\end{equation}
with a constant independent of $\mu$.

Furthermore, there exists a unique
$\phi \coloneqq \mathcal{L}^{-1}\left\{\widehat{\phi}\right\} \in \mathcal{W}_\mu(\mathbb{R}_+;\mathcal{H},\mathcal{V})$
solution to 
\begin{align}
    \partial^2_t \phi (t)    +  \mathsf{A}  \phi (t)    &=    f(t),    \quad t>0,\\
    \phi(0) &= \phi_0 ,\\
    \partial_t \phi(0) &= \phi_1. 
\end{align}
satisfying
\begin{equation}
\begin{aligned}
	\norm{\phi}_{\mathcal{W}_\mu(\mathbb{R}_+,\mathcal{V},\mathcal{H})}
	\lesssim
    &
	\max\left\{1,C_{\mathsf{A}}\right\}
	\left(
   		\frac{\norm{\phi_0}_{\mathcal{V}}}{\sqrt{\mu}}
		+
		\frac{\norm{\phi_1}_{\mathcal{V}}}{\sqrt{\mu^3}}
	\right)
    +
    \norm{{f}}_{{L^2_\mu(\mathbb{R}_+;\mathcal{V}^\star)}}
    \\
    &
    +
	\frac{C_{\mathsf{A}}}{{\mu}\min\{1, c_{\mathsf{A}}\}}
	\left(
		\norm{{f}}_{{L^2_\mu(\mathbb{R}_+;\mathcal{H})}} 
		+ 
		\frac{\norm{\mathsf{A}\phi_0}_{\mathcal{H}}}{\sqrt{\mu}}
		+
		\frac{\norm{\mathsf{A}\phi_1}_{\mathcal{H}}}{\sqrt{\mu^3}}
	\right).
\end{aligned}
\end{equation}
\end{corollary}

\begin{proof}
Recalling that
\begin{equation}\label{eq:bound_integral_mu}
	\int_{-\infty}^{+\infty}
	\frac{	\text{d}
	\tau}{\snorm{\mu + \imath \tau}^2}
	=
	\frac{\pi}{\mu}
	\quad
	\text{and}
	\quad
	\int_{-\infty}^{+\infty}
	\frac{	\text{d}
	\tau}{\snorm{\mu + \imath \tau}^4}
	=
	\frac{\pi}{2\mu^3}
\end{equation}
together with \eqref{eq:bound_laplace_1}--\eqref{eq:bound_laplace_3} we may obtain
the bounds stated in this result. 

\end{proof}

\subsection{Sinc Quadrature and Interpolation in a Strip}\label{ssec:sinc}
We recall important definitions and results concerning sinc interpolation
and quadrature in a strip as introduced in \cite[Chapter 3]{stenger2012numerical}.

In the following, for a given $\mu>0$, we set $\mathscr{D}_\mu \coloneqq \{z \in \mathbb{C}:\; \snorm{\Im\{z\}} < \mu\}$ and for $\varepsilon>0$ we consider
\begin{equation}
    \mathscr{D}_\mu(\varepsilon)
    \coloneqq
    \left\{
        z \in \mathbb{C}:
        \;
        \snorm{\Re\{z\}}< \frac{1}{\varepsilon}
        \quad
        \text{and}
        \quad
        \snorm{\Im\{z\}}< \mu(1-{\varepsilon})
    \right\}.
\end{equation}

\begin{definition}\label{eq:definition_H_bb}
Let $p\in[1,\infty)$, $\mu>0$, and let $\mathcal{X}$ be a {\it complex}
Banach space equipped with the induced norm $\norm{\cdot}_\mathcal{X}\coloneqq \sqrt{\dotp{\cdot}{\cdot}_\mathcal{X}}$.
We denote by $\mathbb{H}^p(\mathscr{D}_\mu;\mathcal{X})$ the family of all functions
$\varphi: \mathscr{D}_\mu \rightarrow \mathcal{X}$ that are analytic in $\mathscr{D}_\mu$
and such that
\begin{equation}
	\norm{\varphi}_{\mathbb{H}^p(\mathscr{D}_\mu;\mathcal{X})}
	\coloneqq
	\left(
        \lim_{\varepsilon \rightarrow 0^+}
		\int_{\partial \mathscr{D}_\mu(\varepsilon) }
		\norm{\varphi(z)}^p_\mathcal{X}
		\snorm{\text{d}z}
	\right)^p
	<\infty.
\end{equation}
\end{definition}

For a given $\varphi \in \mathbb{H}^p(\mathscr{D}_\mu;\mathcal{X})$, $p\in [1,\infty)$
and $\vartheta>0$, we define for $z \in \mathscr{D}_\mu$
\begin{equation}\label{eq:sinc_interpolation}
	\mathcal{S}(\varphi,\vartheta)(z)
	=
	\sum_{k\in \mathbb{Z}}
	\varphi(k \vartheta)
	\text{Sinc}(k,\vartheta)(z)
	\quad
	\text{and}
	\quad
	\mathcal{S}_N(\varphi,\vartheta)
	=
	\sum_{k=-N}^{N}
	\varphi(k \vartheta)
	\text{Sinc}(k,\vartheta)(z),
\end{equation}
where
\begin{equation}
    \Sinc(k,\vartheta)(z)
    =
    \frac{\sin(\pi(z-k \vartheta)/\vartheta)}{\pi(z-k\vartheta)/\vartheta},
    \quad
    k  \in \mathbb{Z},
    \quad
    z \in \mathbb{C}.
\end{equation}

We recall the following results concerning the approximation 
of sinc interpolation and quadrature in a strip. 

\begin{proposition}[{\cite[Theorem 3.1.3, item (b)]{stenger2012numerical}}]
\label{prop:sinc_interpolation}
Let $\varphi \in \mathbb{H}^2(\mathscr{D}_\mu;\mathcal{X})$ for some $\mu>0$.
Then, it holds for $|y|<\mu$ and $\vartheta>0$
\begin{equation}
	\norm{
	    \varphi(\cdot +\imath y) 
        -
        \mathcal{S}(\varphi,\vartheta)(\cdot +\imath y)
	}_{L^2(\mathbb{R};\mathcal{X})}
    \leq
    \frac{
        \cosh(\pi y/\vartheta)
    }{
        \sinh(\pi \mu/\vartheta)
    }
    \norm{\varphi}_{\mathbb{H}^2(\mathscr{D}_\mu;\mathcal{X})}.
\end{equation}
\end{proposition}

Next, we consider the sinc quadrature in the strip.
For a given function $\varphi: \mathbb{R} \rightarrow \mathcal{X}$
and for $\vartheta>0$, we define
\begin{equation}\label{eq:def_I_T}
    \mathcal{I}(\varphi)
    \coloneqq
    \int_{\mathbb{R}}
    \varphi(x) \,\text{d}x
    \quad
    \text{and}
    \quad
    \mathcal{Q}(\varphi,\vartheta)
    \coloneqq
    \vartheta
    \sum_{k \in \mathbb{Z}} \varphi(k\vartheta)
\end{equation}
together with 
\begin{equation}
    \mathcal{Q}_M(\varphi,\vartheta)
    \coloneqq
    \vartheta
    \sum_{k=-M}^{M} \varphi(k\vartheta).
\end{equation}

\begin{proposition}[{\cite[Theorem 3.2.1]{stenger2012numerical}}]
\label{prop:integral_sinc}
Let $\varphi \in \mathbb{H}^p(\mathscr{D}_\mu;\mathbb{C})$ for some $\mu>0$ and $p\in [1,\infty)$. 
Then, it holds that
\begin{equation}
	\snorm{
	       \mathcal{I}(\varphi)
           -
           \mathcal{Q}(\varphi,\vartheta)
	}
    \leq
    \frac{
        e^{-\pi \mu \vartheta}
    }{
        2\sinh(\pi \mu/\vartheta)
    }
    \norm{\varphi}_{\mathbb{H}^1(\mathscr{D}_\mu;\mathbb{C})}.
\end{equation}
\end{proposition}

\subsection{Low-Rank Approximation}\label{ssec:low_rank}
Let $\mu>0$ be fixed. In view of Proposition~\ref{prop:stability_laplace_domain_problem}, there exists a unique solution $\widehat{U}^{(2)}(s) \in \mathcal{V}$ of Problem~\ref{pbrm:laplace_discrete} for each $s\in\Pi_\mu$.
To analyze the behavior of $\widehat{U}^{(2)}(\mu+\imath\tau)$ as a function of $\tau\in\IR$, which will allow us to use the results on sinc approximation from the previous section, we start with some results concerning the bilateral Laplace transform of $\partial_t^2 q$.
To this end, we note that $q=\partial_t^2 g$ with
\begin{equation}\label{eq:gausian_minus}
    g(t) = -\frac{2}{\alpha^2}\, e^{-\frac{\alpha^2}{4}(t - t_0)^2},
\end{equation}
and obtain the following result. 

\begin{lemma}\label{lem:bilateral_laplace_ricker}
It holds that
\begin{align}
	\cB\{g\}(s)
    = 
    -\frac{4\sqrt{\pi}}{\alpha^3} e^{\left(\frac{s}{\alpha}\right)^2-s t_0}.
\end{align}
\end{lemma}
\begin{proof}
Using transformation rules for the bilateral Laplace transform \cite[Table~14.1]{bracewell2000fourier},
we obtain that
\begin{align}
    \cB\left\{e^{-\frac{\alpha^2}{4}(t-t_0)^2}\right\}(s) 
    = 
    e^{-st_0}
    \cB
    \left\{
        e^{- (\frac{\alpha t}{2})^2}
    \right\}(s) 
    = 
    \frac{2}{\alpha}e^{-st_0}\cB\{e^{-t^2}\}\left(\frac{2s}{\alpha}\right).
\end{align}
Using the substitution $u=t+s/2$, we further have that
\begin{align}
    \cB\{e^{-t^2}\}&=\int_{-\infty}^\infty e^{-(t^2 + st)}\,dt 
    =e^{s^2/4} \int_{-\infty}^\infty e^{-\left(t^2 + st + \left(\frac{s}{2}\right)^2\right)}\,dt \\
&=e^{s^2/4} \int_{-\infty}^\infty e^{-u^2}\,du= \sqrt{\pi}e^{s^2/4}.
\end{align}
Therefore, the result for $\cB\{g\}$ follows. 
\end{proof}

The transform $\cB\{q\}$ again follows using the transformation rules of the bilateral Laplace transform for derivatives. Indeed, observe that
\begin{equation}
	\mathcal{B}
	\left\{
		\partial^2_t  q 
	\right\}(s)
	=
	(-s)^{4}
	\mathcal{B}
	\left\{
		g
	\right\}(s)
	=
	-\frac{4\sqrt{\pi}}{\alpha^3}
	(-s)^{4}
	e^{
		{\left(\frac{s}{\alpha}\right)^2-s t_0}
	}.
\end{equation}
and for each $s \in \mathbb{C}$
\begin{equation}\label{eq:bound_g_n}
	\snorm{
	\mathcal{B}
	\left\{
		\partial^2_t  q
	\right\}(s)
	}
	\leq
	\frac{4\sqrt{\pi}}{\alpha^3}
	e^{-\Re\{s\} t_0+\frac{\Re\{s\}^2}{\alpha^2}}
	\snorm{s}^{4}
	e^{-\frac{\Im\{s\}^2}{\alpha^2}}.
\end{equation}
Combining this bound with the stability bounds derived in Proposition~\ref{prop:stability_laplace_domain_problem}, we obtain the following result that shows exponential decay of $\widehat{U}^{(2)}(\mu+\imath\tau)$ as a function of $\tau$.
\begin{lemma}\label{lmm:regularity_U_n}
Let $\mu >0$ be given. 
\begin{itemize}
\item[(i)]
For any $s \in \Pi_{\mu}$, it holds that
\begin{align}\label{eq:bound_U}
	\norm{
		\widehat{U}^{(2)}(s)
	}_\mathcal{V}
    \lesssim
    \frac{e^{-\Re\{s\} t_0+\frac{\Re\{s\}^2}{\alpha^2}}}{\alpha^3 \Re\{s\}
    \min\{1, c_{\mathsf{A}}\}}
    \snorm{s}^{4}
	e^{-\frac{\Im\{s\}^2}{\alpha^2}}
	\norm{p}_\mathcal{H}.
\end{align}
\item[(ii)] For any $\eta \in (0,{\mu})$, it holds that
$\widehat{U}^{(2)}(\mu + \imath \cdot) \in \mathbb{H}^2(\mathscr{D}_\eta;\mathcal{V})$
and
\begin{align}
	\norm{
		\widehat{U}^{(2)}(\mu + \imath \cdot)
	}_{\mathbb{H}^2(\mathscr{D}_\eta;\mathcal{V})}
    \leq
    \frac{e^{-(\mu-\eta) t_0+\frac{(\mu+\eta)^2}{\alpha^2}}}{\alpha^3(\mu-\eta) \min\{1, c_{\mathsf{A}}\}}
    \left(
		\mu^\frac{9}{2}
		+
		\alpha^\frac{9}{2}
    \right)
    \norm{p}_\mathcal{H}.
\end{align}
\end{itemize}
\end{lemma}

\begin{proof}
It follows from Proposition~\ref{prop:stability_laplace_domain_problem} 
and \eqref{eq:bound_g_n} that for each $s \in \Pi_\mu$
\begin{equation}\label{eq:bound_Laplace_V_s}
\begin{aligned}
	\snorm{s}
	\norm{
		\widehat{U}^{(2)}(s)
	}_\mathcal{H}
	+
	\norm{
		\widehat{U}^{(2)}(s)
	}_\mathcal{V}
	\lesssim
	&
	\frac{
	\snorm{\mathcal{B}
	\left\{
		\partial^2_t  q (t)
	\right\}(s)}
	}{{\Re\{s\}} \min\{1, c_{\mathsf{A}}\}}
	\norm{p}_\mathcal{H}
	\\
	\lesssim
	&
	\frac{e^{-\Re\{s\} t_0+ \frac{\Re\{s\}^2}{\alpha^2}}}{\alpha^3 {\Re\{s\}} \min\{1, c_{\mathsf{A}}\}} 
	\snorm{s}^{4}
	e^{-\frac{\Im\{s\}^2}{\alpha^2}}
	\norm{p}_\mathcal{H},
\end{aligned}
\end{equation}
thus proving \eqref{eq:bound_U} in item (i).
\end{proof}

Next, we proceed to prove that  $\widehat{U}^{(2)}(\mu+\imath \cdot )$
verifies Definition~\ref{eq:definition_H_bb}.
Exactly as in the proof of Proposition~\ref{prop:stability_laplace_domain_problem},
we may prove that the map $\Pi_\mu \ni s \mapsto \widehat{U}^{(2)}(s)$ is analytic and so is
$\mathscr{D}_\eta \ni z  \mapsto \widehat{U}^{(2)}(\mu+\imath z)$, for each $\eta \in (0,\mu)$.
For each $\eta \in (0,\mu)$, we estimate
\begin{equation}
	\norm{
		\widehat{U}^{(2)}(\mu + \imath \cdot)
	}^2_{\mathbb{H}^2(\mathscr{D}_\eta;\mathcal{V})}
    =
    \lim_{\varepsilon \rightarrow 0^+}
	\int_{\partial \mathscr{D}_\eta(\varepsilon) }
	\norm{\widehat{U}^{(2)}(\mu + \imath z)}^2_\mathcal{V}
	\snorm{\text{d}z}
    =
    \text{(I) + (II) + (III) + (IV)},
\end{equation}
where, for $\varepsilon \in (0,1)$, one has
\begin{equation}
\begin{aligned}
    \text{(I)}
    =
    \lim_{\varepsilon \rightarrow 0^+}
    \int_{-\frac{1}{\varepsilon}}^{\frac{1}{\varepsilon}}
    \norm{\widehat{U}^{(2)}(\mu + \imath (\tau  +\imath\eta(1-\varepsilon)))}^2_\mathcal{V}
    \text{d}\tau
    \\
    \text{(II)}
    =
    \lim_{\varepsilon \rightarrow 0^+}
    \int_{-\frac{1}{\varepsilon}}^{\frac{1}{\varepsilon}}
    \norm{\widehat{U}^{(2)}(\mu + \imath (\tau - \imath\eta(1-\varepsilon)))}^2_\mathcal{V}
    \text{d}\tau
\end{aligned}
\end{equation}
and
\begin{equation}
\begin{aligned}
    \text{(III)}
    =
    \lim_{\varepsilon \rightarrow 0^+}
    \int_{-\eta(1-\varepsilon)}^{\eta(1-\varepsilon)}
    \norm{\widehat{U}^{(2)}  
    \left(\mu+\imath\left(+\frac{1}{\varepsilon}+\imath \tau\right)\right)}^2_\mathcal{V}
    \text{d}\tau
    \\
    \text{(IV)}
    =
    \lim_{\varepsilon \rightarrow 0^+}
    \int_{-\eta(1-\varepsilon)}^{\eta(1-\varepsilon)}
    \norm{\widehat{U}^{(2)}
    \left(\mu+\imath\left(-\frac{1}{\varepsilon}+\imath\tau\right)\right)}^2_\mathcal{V}
    \text{d}\tau.
\end{aligned}
\end{equation}
Next, recalling \eqref{eq:bound_Laplace_V_s}
and Lemma~\ref{lem:bilateral_laplace_ricker}, we calculate (I) and (II) as follows
\begin{equation}
\begin{aligned}
	\lim_{\varepsilon \rightarrow 0^+}
    \int_{-\frac{1}{\varepsilon}}^{\frac{1}{\varepsilon}}
    \norm{\widehat{U}^{(2)}(\mu + \imath (\tau  \pm \imath\eta(1-\varepsilon)))}^2_\mathcal{V}
    &
    \text{d}\tau
    \\
    &
    \hspace{-3.5cm}
    \lesssim
    \frac{e^{-2(\mu-\eta) t_0+2\frac{(\mu+\eta)^2}{\alpha^2}}}{\alpha^6(\mu-\eta)^2\min\{1, c^2_{\mathsf{A}}\}}
    \lim_{\varepsilon \rightarrow 0^+}
    \int_{-\frac{1}{\varepsilon}}^{\frac{1}{\varepsilon}}
    (\tau^2+(\mu\pm\eta(1-\varepsilon))^2)^4
	e^{-2\frac{\tau^2}{\alpha^2}}
    \text{d}\tau
	\norm{p}^2_\mathcal{H}
    \\
    &
    \hspace{-3.5cm}
    \lesssim
    \frac{e^{-2(\mu-\eta) t_0+2\frac{(\mu+\eta)^2}{\alpha^2}}}{\alpha^6(\mu-\eta)^2 \min\{1, c^2_{\mathsf{A}}\}}
    \left(
         \int_0^{\mu\pm\eta}
    		(\tau^2+(\mu\pm\eta)^2)^4
      	e^{-2\frac{\tau^2}{\alpha^2}}
         \text{d}\tau
     \right.
     \\
     &
     \left.
         +
         \int_{\mu\pm\eta}^{\infty}
    		(\tau^2+(\mu\pm\eta)^2)^4
      	e^{-2\frac{\tau^2}{\alpha^2}}
         \text{d}\tau
    \right)
    \norm{p}^2_\mathcal{H}
    \\
    &
    \hspace{-3.5cm}
    \lesssim
    \frac{e^{-2(\mu-\eta) t_0+2\frac{(\mu+\eta)^2}{\alpha^2}}}{\alpha^6(\mu-\eta)^2 \min\{1, c^2_{\mathsf{A}}\}}
    \left(
    		(\mu\pm\eta)^9
		+
		\alpha^9
         \int_{\sqrt{2} \frac{\mu\pm\eta}{\alpha}}^{\infty}
    		\tau^8
      	e^{-\tau^2}
         \text{d}\tau
    \right)
    \norm{p}^2_\mathcal{H},
\end{aligned}
\end{equation}
thus yielding
\begin{equation}
\begin{aligned}
	\lim_{\varepsilon \rightarrow 0^+}
    \int_{-\frac{1}{\varepsilon}}^{\frac{1}{\varepsilon}}
    \norm{\widehat{U}^{(2)}(\mu + \imath (\tau  \pm \imath\eta(1-\varepsilon)))}^2_\mathcal{V}
    \leq
    \frac{e^{-2(\mu-\eta) t_0+2\frac{(\mu+\eta)^2}{\alpha^2}}}{\alpha^6(\mu-\eta)^2 \min\{1, c^2_{\mathsf{A}}\}}
    \left(
    		\mu^9
		+
		\alpha^9
    \right)
    \norm{p}^2_\mathcal{H}.
\end{aligned}
\end{equation}

In addition, one can verify straightforwardly that (III) = (IV) = 0,
thus yielding the result stated in item (ii).

\begin{lemma}\label{lmm:error_U}
Let $u \in \mathcal{W}_\mu(\mathbb{R}_+;\mathcal{H},\mathcal{V})$
for some $\mu>0$ be the solution to \eqref{eq:wave}--\eqref{eq:ic2}.
Assume $p \in \mathcal{V}$ and $\mathsf{A}p \in \mathcal{H}$.
\begin{itemize}
	\item[(i)]
	One has that $\partial^2_t u \in \mathcal{W}_{\mu}(\mathbb{R}_+;\mathcal{H},\mathcal{V})$
	and, for each $s \in \Pi_\mu$, it holds that
\begin{equation}\label{eq:bound_L_dt2_u}
\begin{aligned}
	\norm{
		\mathcal{L}\{\partial^2_t u\}(s)
	}_{\mathcal{V}}
	\lesssim
	&
	\frac{
		\snorm{\mathcal{L}\{\partial^2_t q\}(s)}
		\norm{p}_{\mathcal{H}}
	}{
		\Re\{s\} \min\{1, c_{\mathsf{A}}\}
	}
	+
	\left(
      	\norm{p}_{\mathcal{V}}
      	+
      	\frac{
			\norm{\mathsf{A}p}_{\mathcal{H}}
		}{
			\Re\{s\}  \min\{1, c_{\mathsf{A}}\}
		}
	\right)
	\left(
		\frac{\snorm{q(0)}}{\snorm{s}}
		+
		\frac{\snorm{\partial_t q(0)}}{\snorm{s}^2}
	\right).
\end{aligned}
\end{equation}
\item[(ii)] For each $s \in \Pi_\mu$ it holds that
\begin{equation}
\begin{aligned}
	\norm{
		\mathcal{L}\{\partial^2_t u\}(s)
		-
		\widehat{U}^{(2)}(s)
	}_{\mathcal{V}}
	\leq
	&
    \frac{1}{\snorm{s}^2}
    \left(
            \norm{p}_\mathcal{V}
            +
    		\frac{
        \norm{\mathsf{A}p}_{\mathcal{H}}
        }{
        \Re\{s\}  \min\{1, c_{\mathsf{A}}\}
        }
    \right)
	\int_{-\infty}^{0}
	\snorm{\partial^2_t  q (t)}
	e^{-\Re\{s\} t}
	\normalfont\text{d} t
	\\
	&
	+
	\left(
      	\norm{p}_{\mathcal{V}}
      	+
      	\frac{
			\norm{\mathsf{A}p}_{\mathcal{H}}
		}{
			\Re\{s\}  \min\{1, c_{\mathsf{A}}\}
		}
	\right)
	\left(
		\frac{\snorm{q(0)}}{\snorm{s}}
		+
		\frac{\snorm{\partial_t q(0)}}{\snorm{s}^2}
	\right).
\end{aligned}
\end{equation}
\end{itemize}
\end{lemma}

\begin{proof}
Observe that $\partial^2_t u$ is the solution to
\begin{align}
    \partial^4_t u(t)    +  \mathsf{A}  \partial^2_t  u(t)    &=  \partial^2_t f(t), \quad t>0,   \\
    \partial^2_t u (0) &= f(0),\\
    \partial^3_t  u(0) &= \partial_t f(0).
\end{align}
Therefore, in the Laplace domain, for each $s \in \Pi_\mu$, one has that
\begin{equation}
	s^2 \mathcal{L}\{  \partial^2_t  u \}(s)
	+
	\mathsf{A} \mathcal{L}\{  \partial^2_t  u \}(s)
	=
	\mathcal{L}\{  \partial^2_t  q \}(s) p
	+
	s q(0) p
	+
	\partial_t q(0) p.
\end{equation}
Since $f(0)=q(0)p\in \cV$ and $\partial_t f(0)=\partial_t q(0) p\in\cV$, the assumption $\sfA p\in\cH$ allows us to apply Proposition~\ref{prop:stability_laplace_domain_problem} and item(i) follows.

To verify (ii), let $s\in\Pi_\mu$ and observe that
\begin{equation}
\begin{aligned}
	\mathcal{L}\{  \partial^2_t  q \}(s)
	=
	\int_{0}^{\infty}
	\partial^2_t  q (t)
	e^{-s t}
	\text{d} t
	=
	\mathcal{B}
	\left\{
		\partial^2_t  q 
	\right\}(s)
	-
	\int_{-\infty}^{0}
	\partial^2_t  q  (t)
	e^{-s t}
	\text{d} t.
\end{aligned}
\end{equation}
Therefore, the difference $\widehat{w}(s) = \mathcal{L}\{\partial^2_t u\}(s) -\widehat{U}^{(2)}(s) \in \mathcal{V} $ satisfies the problem 
\begin{equation}
	s^2 \widehat{w}(s)
	+
	\mathsf{A} \widehat{w}(s)
	=
    -
    p
	\int_{-\infty}^{0}
	\partial^2_t  q  (t)
	e^{-s t}
	\text{d} t
	+
	s q(0) p
	+
	\partial_t q(0) p.
\end{equation}
Furthermore, since $p\in\cV$ by assumption, one has that
\begin{equation}
	\widehat{h}(s) 
    =
    \widehat{w}(s)
    +
    \frac{p}{s^2}
	\int_{-\infty}^{0}
	\partial^2_t  q  (t)
	e^{-s t}
	\text{d} t,
\end{equation}
satisfies $\widehat{h}(s) \in\cV$, for each $s \in \Pi_\mu$.
We therefore have the bound
\begin{equation}\label{eq:bound_W_s_H}
    \norm{\widehat{w}(s)}_\mathcal{V}
    \leq
    \norm{\widehat{h}(s)}_\mathcal{V}
    +
    \frac{\norm{p}_\mathcal{V}}{\snorm{s}^2}
	\int_{-\infty}^{0}
	\snorm{\partial^2_t  q  (t)}
	e^{-\Re\{s\} t}
	\text{d} t.
\end{equation}
Observe that $\widehat{h}(s)$ satisfies the equation
\begin{equation}\label{eq:Laplace_h_s}
	s^2 \widehat{h}(s)
	+
	\mathsf{A} \widehat{h}(s)
	=
    \frac{\mathsf{A} p}{s^2}
	\int_{-\infty}^{0}
	\partial^2_t  q  (t)
	e^{-s t}
	\text{d} t
	+
	s q(0) p
	+
	\partial_t q(0) p.
\end{equation}
Since $p \in \mathcal{V}$ and $\mathsf{A}p \in \mathcal{H}$, Proposition~\ref{prop:stability_laplace_domain_problem}
applied to~\eqref{eq:Laplace_h_s} implies that, for every $s \in \Pi_\mu$,
\begin{equation}
\begin{aligned}
    \norm{\widehat{h}(s)}_\mathcal{V}
	\leq
	&
    \frac{
        \norm{\mathsf{A}p}_{\mathcal{H}}
    }{
        \Re\{s\}  \min\{1, c_{\mathsf{A}}\}
        \snorm{s}^2
    }
	\int_{-\infty}^{0}
	\snorm{\partial^2_t  q (t)}
	e^{-\Re\{s\} t}
	\text{d} t
	\\
	&
	+
	\left(
      	\norm{p}_{\mathcal{V}}
      	+
      	\frac{
			\norm{\mathsf{A}p}_{\mathcal{H}}
		}{
			\Re\{s\}  \min\{1, c_{\mathsf{A}}\}
		}
	\right)
	\left(
		\frac{\snorm{q(0)}}{\snorm{s}}
		+
		\frac{\snorm{\partial_t q(0)}}{\snorm{s}^2}
	\right).
\end{aligned}
\end{equation}
This bound, together with \eqref{eq:bound_W_s_H}, yields the final result.

\end{proof}
Combining the consistency error bound stated in the previous lemma (item (ii)), exponential decay of $\widehat{U}^{(2)}(\mu+\imath \tau)$ as a function of $\tau\in\IR$ for $\mu>0$ fixed from Lemma~\ref{lmm:regularity_U_n}, and the sinc interpolation error bounds shown in Section~\ref{ssec:sinc} we obtain the following bound for approximation of $\partial_t^2 u$ 
in finite-dimensional subspaces $\mathcal{V}_R$ of dimension $R$.

\begin{lemma}\label{lem:approximation_Hardy_spaces_2}
Let  $u \in \mathcal{W}_\mu(\mathbb{R}_+;\mathcal{V},\mathcal{H})$
for some $\mu>0$ be the solution to \eqref{eq:wave}--\eqref{eq:ic2}.
Furthermore, assume that $p \in \mathcal{V}$ and 
$\mathsf{A} p \in \mathcal{V}$. Then, for any $\eta \in (0,\mu)$, it holds that
\begin{equation}\label{eq:bound_inf_V}
\begin{aligned}
	\inf_{
	\substack{
		\mathcal{V}_R \subset \mathcal{V}
		\\
		\normalfont\text{dim}(\mathcal{V}_R)\leq R	
	}
	}	
	\norm{
		\partial^2_t u
		-
		\mathsf{P}_{\mathcal{V}_R}
		\partial^2_t u
	}_{L^2(\mathfrak{I};\mathcal{V})}
    \lesssim
    &
    e^{\mu T}
    \left(
    \frac{e^{-(\mu-\eta) t_0+\frac{(\mu+\eta)^2}{\alpha^2}}}{\alpha^3(\mu-\eta)}
    \left(
		\mu^4
		+
		\alpha^4
    \right)
        \right.
        \\
        &
        \left.
        +
   \frac{
   	   \mu^{\frac{1}{6}}
   		(\alpha^2+\mu^2)^2 
   }{
   	\alpha^{\frac{11}{3}}
   }
    \frac{
        e^{-\mu t_0 + \frac{\mu^2}{\alpha^2}}
        }{
        \min\{1, c_{\mathsf{A}}\}
    }
    R^\frac{7}{6}
    \right)
	e^{
		-\left(\frac{\pi R \eta}{2\alpha} \right)^{\frac{2}{3}}
	}
	 \norm{p}_\mathcal{H}
	\\
	&
	+
    \frac{e^{\mu T}}{\sqrt{\mu^3}}
    \left(
    \frac{
        \norm{\mathsf{A}p}_{\mathcal{H}}
        }{
        \mu  \min\{1, c_{\mathsf{A}}\}
        }
        +
        \norm{p}_\mathcal{V}
    \right)
	\int_{-\infty}^{0}
	\snorm{\partial^2_t  q (t)}
	e^{-\mu t}
	\normalfont\text{d} t
	\\
	&
	+
	e^{\mu T}
	\left(
      	\norm{p}_{\mathcal{V}}
      	+
      	\frac{
			\norm{\mathsf{A}p}_{\mathcal{H}}
		}{
			\mu \min\{1, c_{\mathsf{A}}\}
		}
	\right)
	\left(
		\frac{\snorm{q(0)}}{\sqrt{\mu}}
		+
		\frac{\snorm{\partial_t q(0)}}{\sqrt{\mu^3}}
	\right).
\end{aligned}
\end{equation}
\end{lemma}

\begin{proof}
In the following, let $\mu>0$ be fixed and $\mathcal{V}_R$ be a finite-dimensional subspace of $\mathcal{V}$, which will be specified below.
As a consequence of Lemma~\ref{lmm:error_U}, item (i), one has that $\partial^2_t u \in \mathcal{W}_{\mu}(\mathbb{R}_+;\mathcal{H},\mathcal{V})$.
Theorem~\ref{thm:paley_wiener} (Paley-Wiener) and Proposition~\ref{prop:properties_hardy}, item (i), yields
\begin{equation}
\begin{aligned}
	\norm{
		\partial^2_t u
		-
		\mathsf{P}_{\mathcal{V}_R}
		\partial^2_t u
	}_{L^2(\mathfrak{I};\mathcal{V})}
    \leq
    &
    e^{\mu T}
	\norm{
		\big(\partial^2_t u 
		-
		\mathsf{P}_{\mathcal{V}_R}
		\partial^2_t u\big) e^{-\mu t}
	}_{L^2(\mathfrak{I};\mathcal{V})}
    \\
    =
    &
    e^{\mu T}
	\norm{
		\partial^2_t u 
		-
		\mathsf{P}_{\mathcal{V}_R}
		\partial^2_t u
	}_{L^2_\mu(\mathbb{R}_+;\mathcal{V})}
	\\
    =
    &
    \frac{
        e^{\mu T}
    }{
        \sqrt{2 \pi}
    }
	\left(
		\int_{-\infty}^{+\infty}
		\norm{
    		\mathcal{L}\left\{\partial^2_t u \right\}
            (\mu+\imath \tau)
    		-
    		\mathsf{P}_{\mathcal{V}_R}
    		\mathcal{L} \left\{\partial^2_t u \right\}
        	(\mu+\imath \tau )
		}^2_\mathcal{V}
		\normalfont\text{d} \tau
	\right)^{\frac{1}{2}},
\end{aligned}
\end{equation}
and for any $w_R \in L^2(\mathbb{R};\mathcal{V}_R)$ 
\begin{equation}
\begin{aligned}
	\norm{
		\partial^2_t u
		-
		\mathsf{P}_{\mathcal{V}_R}
		\partial^2_t u
	}_{L^2(\mathfrak{I};\mathcal{V})}
    \leq
    \frac{
        e^{\mu T}
    }{
        \sqrt{2 \pi}
    }
    \norm{
        \mathcal{L}\left\{\partial^2_t u \right\}
        (\mu+\imath \cdot)
        -
        w_R
    }_{L^2(\mathbb{R};\mathcal{V})}
\end{aligned}
\end{equation}
and 
\begin{equation}\label{eq:decomp}
\begin{aligned}
	\norm{
		\partial^2_t u
		-
		\mathsf{P}_{\mathcal{V}_R}
		\partial^2_t u
	}_{L^2(\mathfrak{I};\mathcal{V})}
    \leq
    &
    \frac{
        e^{\mu T}
    }{
        \sqrt{2 \pi}
    }
    \norm{
        \mathcal{L}\left\{\partial^2_t u \right\}
        (\mu+\imath \cdot)
        -
        \widehat{U}^{(2)}(\mu+\imath\cdot)
    }_{L^2(\mathbb{R};\mathcal{V})}
    \\
    &
    +
    \frac{
        e^{\mu T}
    }{
        \sqrt{2 \pi}
    }
    \norm{
        w_R
        -
        \widehat{U}^{(2)}(\mu+\imath\cdot)
    }_{L^2(\mathbb{R};\mathcal{V})},
\end{aligned}
\end{equation}
where the first term on the right-hand side of \eqref{eq:decomp} is bounded as a
consequence of Lemma~\ref{lmm:error_U}, item (ii), and \eqref{eq:bound_integral_mu},
according to 
\begin{equation}\label{eq:error_consistency}
\begin{aligned}
    \norm{
        \mathcal{L}\left\{\partial^2_t u \right\}
        (\mu+\imath \cdot)
        -
        \widehat{U}^{(2)}(\mu+\imath\cdot)
    }_{L^2(\mathbb{R};\mathcal{V})}
    \\
    	&
	\hspace{-2cm}
    \lesssim
    \frac{1}{\sqrt{\mu^3}}
    \left(
    \frac{
        \norm{\mathsf{A}p}_{\mathcal{H}}
        }{
        \mu  \min\{1, c_{\mathsf{A}}\}
        }
        +
        \norm{p}_\mathcal{V}
    \right)
	\int_{-\infty}^{0}
	\snorm{\partial^2_t  q (t)}
	\exp(-\mu t)
	\normalfont\text{d} t
    \\
    	&
	\hspace{-2cm}
   	+
	\left(
      	\norm{p}_{\mathcal{V}}
      	+
      	\frac{
			\norm{\mathsf{A}p}_{\mathcal{H}}
		}{
			\Re\{s\}  \min\{1, c_{\mathsf{A}}\}
		}
	\right)
	\left(
		\frac{\snorm{q(0)}}{\snorm{s}}
		+
		\frac{\snorm{\partial_t q(0)}}{\snorm{s}^2}
	\right).
\end{aligned}
\end{equation}
For each $\tau \in \mathbb{R}$, we set $\zeta(\tau) = \widehat{U}^{(2)}(\mu+\imath\tau)$.
Observe that $\zeta \in \mathbb{H}^2(\mathscr{D}_{\eta },\mathcal{V})$ for any $\eta \in (0,\mu)$
as a consequence of Lemma~\ref{lmm:regularity_U_n}.
Thus, recalling Proposition \ref{prop:sinc_interpolation} with
$y = 0$ and $\vartheta>0$ to be specified, we have that for any $\eta \in (0,\mu)$ it holds that
\begin{equation}\label{eq:sinc_approx_ineq}
\begin{aligned}
    \norm{
        \zeta
        -
        \mathcal{S}_{\left\lfloor \frac{R}{2} \right\rfloor-1}(\zeta,\vartheta)
    }_{L^2(\mathbb{R};\mathcal{V})}
    \\
   	&
	\hspace{-2cm}
	\leq
    \norm{
        \zeta
        -
        \mathcal{S}(\zeta,\vartheta)
    }_{L^2(\mathbb{R};\mathcal{V})}
    +
    \norm{
        \mathcal{S}(\zeta,\vartheta)
        -
        \mathcal{S}_{\left\lfloor \frac{R}{2} \right\rfloor-1}(\zeta,\vartheta)
    }_{L^2(\mathbb{R};\mathcal{V})}
    \\
   	&
	\hspace{-2cm}
	\leq
    \frac{
        \norm{
             \widehat{U}^{(2)}(\mu+\imath\cdot)
        }_{\mathbb{H}^2(\mathscr{D}_\eta;\mathcal{V})}
    }{
        \snorm{\sinh(\pi \delta/\vartheta )}
    }
    +
    \norm{
        \mathcal{S}(\zeta,\vartheta)
        -
        \mathcal{S}_{\left\lfloor \frac{R}{2} \right\rfloor -1 }(\zeta,\vartheta)
    }_{L^2(\mathbb{R};\mathcal{V})}
    \\
   	&
	\hspace{-2cm}
	\lesssim
    \norm{
         \widehat{U}^{(2)}(\mu+\imath\cdot)
    }_{\mathbb{H}^2(\mathscr{D}_\eta;\mathcal{V})}
    e^{-\pi \delta/\vartheta }
    +
    \norm{
        \mathcal{S}(\zeta,\vartheta)
        -
        \mathcal{S}_{\left\lfloor \frac{R}{2} \right\rfloor-1}(\zeta,\vartheta)
    }_{L^2(\mathbb{R};\mathcal{V})}.
\end{aligned}
\end{equation}
In view of \eqref{eq:sinc_interpolation},
for each $R \in \mathbb{N}$, we set
\begin{equation}
    \mathcal{V}_R 
    \coloneqq
    \spanv
    \left\{\hU(\mu+\imath k\vartheta): 
    k=-\left\lfloor \frac{R}{2}  \right\rfloor+1,\ldots,\left\lfloor \frac{R}{2} \right\rfloor-1
    \right\}
    \subset \mathcal{V}.
\end{equation}  
In particular, one has that
$\mathcal{S}_{\left\lfloor \frac{R}{2} \right\rfloor-1}(\zeta,\vartheta)(\tau) \in \mathcal{V}_R$,
for each $\tau \in \mathbb{R}$.
Observe that $\text{dim}(\mathcal{V}_R) \leq R$.
Using orthogonality of $\{\Sinc(m,\vartheta)\}_{m \in \mathbb{Z}}$, see \cite[p.139]{stenger2012numerical}, 
we obtain
\begin{align}
 \norm{
        \mathcal{S}(\zeta,\vartheta)
        -
        \mathcal{S}_{\left\lfloor \frac{R}{2} \right\rfloor-1}(\zeta,\vartheta)
    }_{L^2(\mathbb{R};\mathcal{V})}^2 
    =\vartheta  \sum_{|m|\geq\left\lfloor \frac{R}{2} \right\rfloor} \norm{\hU(\mu+\imath m\vartheta)}_\cV^2.
\end{align}
Recalling Lemma~\ref{lmm:regularity_U_n}, we obtain
\begin{equation}
\begin{aligned}
    \norm{
        \mathcal{S}(\zeta,\vartheta)
        -
        \mathcal{S}_{\left\lfloor \frac{R}{2} \right\rfloor-1}(\zeta,\vartheta)
    }^2_{L^2(\mathbb{R};\mathcal{V})}
    &
    \lesssim
    \frac{
        e^{-2\mu t_0 +2\frac{\mu^2}{\alpha^2}}
        }{
         \alpha^6  \mu^2\min\{1, c^2_{\mathsf{A}}\}
    }
    \vartheta \norm{p}^2_{\cH}
    \sum_{\snorm{m}\geq\left\lfloor \frac{R}{2} \right\rfloor}
	((m \vartheta)^2+\mu^2)^{4}
	e^{
		-2{\left(\frac{m \vartheta}{\alpha}\right)^2}
	}
	.
\end{aligned}
\end{equation}
Observe that under the assumption $\frac{R \vartheta}{2 \alpha} >1$
\begin{equation}
\begin{aligned}
    \vartheta
    \sum_{\snorm{m}\geq\left\lfloor \frac{R}{2} \right\rfloor}
	((m \vartheta)^2+\mu^2)^{4}
	e^{
		-2{\left(\frac{m \vartheta}{\alpha}\right)^2}
	}
	&
    \lesssim
    \int_{\frac{R}{2} \vartheta }^{\infty}
	(\tau^2+\mu^2)^{4}
	e^{
		-{\left(\sqrt{2}\frac{\tau}{\alpha}\right)^2}
	}
    \text{d}\tau
    \\
    &
    \lesssim
    \alpha
    (\alpha^2+\mu^2)^4
    \int_{\sqrt{2} \frac{R \vartheta}{2 \alpha} }^{\infty}
	\tau^{8}
	e^{
		-\tau^2
	}
    \text{d}\tau.
\end{aligned}
\end{equation}
We calculate and obtain
\begin{equation}\label{eq:integra_x_to_8}
    \int
	\tau^{8}
	e^{
		-\tau^2
	}
    \text{d}\tau
    =
	\frac{1}{32}\left(
  -2\,\tau	e^{
		-\tau^2
	}
  \left(105 + 70 \tau^2 + 28 \tau^4 + 8 \tau^6\right)
  + 105 \sqrt{\pi}\,\erf(\tau)
  \right)
  +C,
\end{equation}
for an arbitrary constant $C>0$.
Therefore, one has that
\begin{equation}
\begin{aligned}
    \norm{
        \mathcal{S}(\zeta,\vartheta)
        -
        \mathcal{S}_{\left\lfloor \frac{R}{2} \right\rfloor-1}(\zeta,\vartheta)
    }_{L^2(\mathbb{R};\mathcal{V})}
    \lesssim
    (\alpha^2+\mu^2)^2
    \frac{
        e^{-\mu t_0 + \frac{\mu^2}{\alpha^2}}
        }{
        \sqrt{\alpha^5} \mu
        \min\{1, c_{\mathsf{A}}\}
    }
    \left(\frac{R\vartheta}{2\alpha}\right)^\frac{7}{2}
    e^{
    -\left(\frac{R\vartheta}{2\alpha}\right)^2}
    \norm{p}_{\mathcal{V}}.
\end{aligned}
\end{equation}

Next, for a given $\eta \in (0,\mu)$, we set $\vartheta = \left( \frac{4\pi \alpha^2 \eta}{R^2} \right)^{\frac{1}{3}}$,
which yields $\frac{R\vartheta}{2\alpha} = \left(\frac{\pi R \eta}{2\alpha} \right)^{\frac{1}{3}}$.
From \eqref{eq:sinc_approx_ineq} we obtain 
\begin{equation}
\begin{aligned}
    \norm{
        \zeta
        -
        \mathcal{S}_{\left\lfloor \frac{R}{2} \right\rfloor-1}(\zeta,\vartheta)
    }_{L^2(\mathbb{R},\mathcal{V})}
    \\
    &
    \hspace{-2cm}
    \lesssim
    \left(
        \norm{
            \widehat{U}^{(2)}(\mu+\imath\cdot)
        }_{\mathbb{H}^2(\mathscr{D}_\eta;\mathcal{V})}
    +
   \frac{
   	   \mu^{\frac{1}{6}}
   		(\alpha^2+\mu^2)^2 
   }{
   	\alpha^{\frac{11}{3}}
   }
    \frac{
        e^{-\mu t_0 + \frac{\mu^2}{\alpha^2}}
        }{
        \min\{1, c_{\mathsf{A}}\}
    }
    R^\frac{7}{6}
    \norm{p}_{\mathcal{V}}
    \right)
	e^{
		-\left(\frac{\pi R \eta}{2\alpha} \right)^{\frac{2}{3}}
	}.
\end{aligned}
\end{equation}
Recalling \eqref{eq:error_consistency}, \eqref{eq:sinc_approx_ineq}, and 
item (i) in Lemma~\ref{lmm:regularity_U_n}, we obtain the final result. 
\end{proof}

It remains to bound the term that includes $\partial_t^2q$ in \eqref{eq:bound_inf_V}.

\begin{lemma}\label{lmm:Smallness_Rickert_wavelet}
Let $\mu>0$ be given and assume
$t_0\geq 2 \frac{\alpha+ \mu}{\alpha^2}$.
It holds that
\begin{equation}
	\int_{-\infty}^{0}
	\snorm{\partial^2_t  q (t)}
	e^{-\mu t}
	\normalfont\text{d} t
	\lesssim
	\alpha
	e^{
        -
        \left(
            \frac{\alpha}{2} t_0
            -
            1
        \right)^2
	}.
\end{equation}
\end{lemma}

\begin{proof}
Firstly, observe that
\begin{equation}
	{\partial^2_t  q (t)}
	=
	{\partial^4_t  g (t)}
	=
	-\frac{2}{\alpha^2}
	\partial^4_t  e^{- \frac{\alpha^2}{4} (t-t_0)^2},
\end{equation}
with $g(t)$ is as in \eqref{eq:gausian_minus}. Therefore, one has that
\begin{equation}
	{\partial^2_t  q (t)}
	=
	-\frac{\alpha^2}{8}
	H_4\left(\alpha\frac{t-t_0}{2}\right)
	e^{- \frac{\alpha^2}{4} (t-t_0)^2},
\end{equation}	
where $H_n(x) \coloneqq (-1)^n e^{x^2} \frac{d^n}{dx^n} e^{-x^2}$ corresponds to the $n$-th Hermite polynomial. Thus, 
\begin{equation}
	\snorm{
		H_4\left(\alpha\frac{t-t_0}{2}\right)
	}
	\lesssim
	e^{\alpha\snorm{t-t_0}}.
\end{equation}
Therefore, assuming $t_0>0$,
\begin{equation}
\begin{aligned}
	\int_{-\infty}^{0}
	\snorm{\partial^2_t  q (t)}
	e^{-\mu t}
	\normalfont\text{d} t
	&
	\lesssim
	\alpha^2
	\int_{-\infty}^{0}
	e^{- \frac{\alpha^2}{4} (t-t_0)^2-\alpha(t-t_0) -\mu t}
	\normalfont\text{d} t
	\\
	&
	=
	\alpha^2 e^{- \mu t_0 }
	\int_{-\infty}^{0}
	e^{- \frac{\alpha^2}{4} (t-t_0)^2-(\alpha+\mu)(t-t_0)}
	\normalfont\text{d} t
	\\
	&
	=
	\alpha^2 
	e^{
		- \mu t_0 
		+
		\frac{(\alpha+\mu)^2}{\alpha^2}
	}
	\int_{-\infty}^{0}
	e^{- \left(\frac{\alpha}{2} (t-t_0)+\frac{(\alpha+\mu)}{\alpha}\right)^2}
	\normalfont\text{d} t
	\\
	&
	=
	\alpha
	e^{
		-\mu t_0 
		+
		\frac{(\alpha+\mu)^2}{\alpha^2}
	}
	\int_{-\infty}^{-\frac{\alpha}{2} t_0+\frac{(\alpha+\mu)}{\alpha}}
	e^{- \tau^2}
	\normalfont\text{d} \tau.
\end{aligned}
\end{equation}
Next, we calculate 
\begin{equation}
\begin{aligned}
	\int_{-\infty}^{-\frac{\alpha}{2} t_0+\frac{(\alpha+\mu)}{\alpha})}
	e^{- \tau^2}
	\normalfont\text{d} \tau
	&
	=
	\frac{\sqrt{\pi}}{2}
	\left(
		\text{erf}
		\left(
            -\frac{\alpha}{2} t_0+\frac{(\alpha+\mu)}{\alpha}
		\right)
		+ 
		1
	\right)
	\\
	&
	=
	\frac{\sqrt{\pi}}{2}
	\left(
		-
		\text{erf}
		\left(
            \frac{\alpha}{2} t_0
            -
            \frac{(\alpha+\mu)}{\alpha}
		\right)
		+ 
		1
	\right)
	\\
	&
	=
	\frac{\sqrt{\pi}}{2}
		\text{erfc}
		\left(
            \frac{\alpha}{2} t_0
            -
            \frac{(\alpha+\mu)}{\alpha}
		\right).
\end{aligned}
\end{equation}
For $x>0$, it holds that $\text{erfc}(x) \leq e^{-x^2}$,
thus yielding
\begin{equation}
\begin{aligned}
	\int_{-\infty}^{0}
	\snorm{\partial^2_t  q (t)}
	e^{-\mu t}
	\normalfont\text{d} t
	&
	\lesssim
	\alpha
	e^{
		-\mu t_0 
		+
		\frac{(\alpha+\mu)^2}{\alpha^2}
	}
	e^{
        -
        \left(
            \frac{\alpha}{2} t_0
            -
            \frac{(\alpha+\mu)}{\alpha}
        \right)^2
	}
	\\
	&
	=
	\alpha
	e^{
		-
        \frac{\alpha^2}{4} t^2_0
        +
        \alpha
        t_0
	}
\end{aligned}
\end{equation}
and concluding the proof.
\end{proof}

\begin{remark}[Dependence of the convergence rate in Lemma~\ref{lem:approximation_Hardy_spaces_2}
on $\alpha$ and $t_0$]\label{rmk:dependence_alpha_t_0}
In view of Lemma~\ref{lmm:Smallness_Rickert_wavelet}, we discuss the structure of the 
result presented in Lemma~\ref{lem:approximation_Hardy_spaces_2}. 
The bound stated in \eqref{eq:bound_inf_V} is composed of two main terms: (i) exponential 
convergence in $R$, i.e., the dimension of sought finite-dimensional space, and (ii) a term decaying exponentially with $\alpha t_0$, which reads as follows: if we select, in particular, $\mu = \alpha$ and consider $\eta =\mu/2$, we have that
\begin{equation}
 \inf_{
	\substack{
		\mathcal{V}_R \subset \mathcal{V}
		\\
		\normalfont\text{dim}(\mathcal{V}_R)\leq R	
	}
	}	
	\norm{
		\partial^2_t u
		-
		\mathsf{P}_{\mathcal{V}_R}
		\partial^2_t u
	}_{L^2(\mathfrak{I};\mathcal{V})}
    \lesssim
    e^{{\alpha} T}
    \left(
    e^{
		-\left(\frac{\pi R }{4} \right)^{\frac{2}{3}}
        - {\alpha}  \frac{t_0}{2}
    }
    +
    e^{
         -
        \frac{\alpha^2}{4} t^2_0
    }
     \right),
\end{equation}
where the hidden constant depends on rational powers of $\alpha$ and $T$, and on $p$ and on $R^{7/6}$.
Observe that we obtain a rate of exponential convergence (in $R$) that is {\it independent}
of the characterizing parameters of the Ricker wavelet, $\alpha$ and $t_0$ as long as the error is larger than the consistency error, which scales like $e^{\alpha T-\frac{\alpha^2}{4} t_0^2}$. 
Taking into account the natural requirement $t_0<T$, the latter decays, however, exponentially fast as $\alpha$ increases,
which results in a narrower wavelet.
\end{remark}

\subsection{Error Analysis: Exponential Convergence}\label{ssec:exponential_convergence}
Let $u \in \mathcal{W}_\mu(\mathbb{R}_+;\mathcal{H},\mathcal{V})$ for some $\mu>0$ be the solution to \eqref{eq:wave} equipped with the vanishing initial conditions \eqref{eq:ic1}--\eqref{eq:ic2}.
Define the operator $\mathsf{T}:L^2_\mu(\IR_+) \rightarrow \mathcal{V} $ as
\begin{equation}
	\mathsf{T} g
	\coloneqq
	\int_{0}^{\infty}
	\partial^2_t u(t)
	g(t)
	e^{-2\mu t}
	\text{d} t,
	\quad 
	\forall g \in L^2_\mu(\IR_+),
\end{equation}
where $\norm{\partial^2_t u}_{L^2_\mu(\IR_+;\mathcal{V})}<\infty$ 
as a consequence of Lemma~\ref{lmm:error_U}, thus rendering a Hilbert-Schmidt kernel. 
The adjoint operator $\mathsf{T}^\star: \mathcal{V} \rightarrow L^2_\mu(\IR_+)$ of $\mathsf{T}$
satisfies
\begin{equation}
	\dotp{
		g
	}{
		\mathsf{T}^\star 
		v
	}_{ L^2_\mu(\IR_+)}
	=
	\dotp{
		\mathsf{T} g
	}{
		 v
	}_{\mathcal{V}},
	\quad 
	\forall g \in L^2_\mu(\IR_+), 
	\quad
	\forall v \in \mathcal{V},
\end{equation}
and is given by
\begin{equation}
	\left(
		\mathsf{T}^\star 
		v
	\right)(t)
	=
	\dotp{
		\partial^2_t {u}(t)
	}{
		v
	}_{\mathcal{V}}
	,
	\quad 
	\forall v \in \mathcal{V}.
\end{equation}
Define
\begin{equation}
	\mathsf{K} \coloneqq \mathsf{T} \mathsf{T}^\star:\mathcal{V} \rightarrow \mathcal{V}
	\quad
	\text{and}
	\quad 
	\mathsf{C} \coloneqq \mathsf{T}^\star \mathsf{T}: L^2_\mu(\IR_+) \rightarrow L^2_\mu(\IR_+).
\end{equation}
Observe that $\mathsf{K}$ can be alternatively represented as an infinite matrix
${\bm{\mathsf{K}}} \in \mathbb{R}^{\mathbb{N} \times \mathbb{N}}$ with entries
\begin{equation}
	\left(
		{\bm{\mathsf{K}}}
	\right)_{i,j}
	=
	\int_{0}^{\infty} 
	\dotp{
		\partial^2_t {u}(\tau)
	}{
		\varphi_{i}
	}_{\mathcal{V}}
	\dotp{
		\varphi_{j}
	}{
		\partial^2_t {u}(\tau)
	}_{\mathcal{V}}
	e^{-2\mu \tau}
	\text{d}
	\tau,
    \quad
    i,j \in \mathbb{N},
\end{equation}
where $\{\varphi_{i}\}_{i \in \mathbb{N}}$ is an orthonormal basis of $\mathcal{V}$.
This operator is symmetric, positive, compact and of trace class, thus admitting an orthonormal eigenbasis $\{\left(\bm{\zeta}_i, \sigma_i\right)\}_{i\in \mathbb{N}} \subset \mathbb{R}^\mathbb{N} \times \mathbb{R}_+$, where the sequence of eigenvalues is
ordered decreasingly.  
Set
\begin{align}
	\zeta_{i}
	=
	\sum_{j=1}^{\infty}
	\left(
	\boldsymbol\zeta_{i}
	\right)_j
	\varphi_{j}
	\in 
	\mathcal{V}.
\end{align}
For each $R \in \mathbb{N}$, we define
\begin{equation}\label{eq:v_rb_optimal}
	\mathcal{V}^{\text{(rb)}}_R
	\coloneqq
	\text{span}
	\left\{	
		\zeta_1
		,
		\dots
		,
		\zeta_R
	\right\}
	\subset
	\mathcal{V}.
\end{equation}
As stated, for example, in \cite[Section 6.4]{quarteroni2015reduced}, it holds 
with $\mathcal{V}^{\text{(rb)}}_R$ as in \eqref{eq:v_rb_optimal} that
\begin{equation}\label{eq:min_time_domain_2}
	\mathcal{V}^{\text{(rb)}}_R
	=
	\argmin_{
	\substack{
		\mathcal{V}_R \subset \mathcal{V}
		\\
		\text{dim}(\mathcal{V}_R)\leq R	
	}
	}	
	\norm{
		\partial^2_t u
		-
		\mathsf{P}_{\mathcal{V}_R}
		\partial^2_t u
	}^2_{L^2_\mu(\IR_+;\mathcal{V})},
\end{equation}
where $\mathsf{P}_{\mathcal{V}_R}: \mathcal{V} \rightarrow \mathcal{V}_R$ denotes the
$\mathcal{V}$-orthogonal projection operator onto the finite dimensional subspace ${\mathcal{V}_R}$.

As discussed in \cite{henriquez2024fast}, the norm equivalence stated in Theorem~\ref{thm:paley_wiener} (Paley-Wiener) implies that
the minimization problem \eqref{eq:min_time_domain_2} can be expressed as follows:
\begin{equation}\label{eq:min_laplace_domain}
	\mathcal{V}^{\text{(rb)}}_R
	=
	\argmin_{
	\substack{
		\mathcal{V}_R \subset \mathcal{V}
		\\
		\text{dim}(\mathcal{V}_R)\leq R	
	}
	}	
	\norm{
		\mathcal{L}\{\partial^2_t u\}
		-
		\mathsf{P}_{\mathcal{V}_R}
		\mathcal{L}\{\partial^2_t u\}
	}^2_{\mathscr{H}^2_\mu(\mathcal{V})},
\end{equation}
where $ \mathcal{L}\{\partial^2_t u\}(s) = s^2 \mathcal{L}\{u\} \in \mathscr{H}^2_\mu(\mathcal{V})$
by virtue of the property of the Laplace transform of the second derivative of a function and the fact that 
\eqref{eq:wave} is equipped with vanishing initial conditions. 

As thoroughly discussed in \cite{henriquez2024fast}, and discussed in this work ahead in
Section~\ref{sec:fully_discrete}, $\mathcal{V}^{\text{(rb)}}_{R,M}$ as in \eqref{eq:fPOD} corresponds to an approximation of $\mathcal{V}^{\text{(rb)}}_R$ as in \eqref{eq:min_laplace_domain}.

Let  $\Pi^{\text{(rb)}}_R:\cV \to\cVrb$ be the so-called {\it elliptic projection}
onto the reduced subspace $\cVrb$, which is defined as follows: for a given
$v \in \mathcal{V}$, its elliptic projection $\Pi^{\text{(rb)}}_R v \in\cVrb$
is defined as
\begin{align}\label{eq:elliptic_projection}
	\dual{
		\mathsf{A}
		\Pi^{\text{(rb)}}_R v
	}{
		\vrb
	}_{\mathcal{V}^\star \times \mathcal{V}^\star}
	=
	\dual{
		\mathsf{A} v
	}{
		\vrb
	}_{\mathcal{V}^\star \times \mathcal{V}^\star},
	\quad\forall \vrb\in\cVrb.
\end{align}
C\'ea's lemma states that
\begin{align}\label{eq:cea_elliptic_projection}
    \norm{u - \Pi^{\text{(rb)}}_R u}_\cV 
    \leq 
    \sqrt{\frac{C_\mathsf{A}}{c_\mathsf{A}}}
    \inf_{\vrb\in\cVrb}\norm{u-\vrb}_\cV,
\end{align}
where $C_\mathsf{A},c_\mathsf{A}>0$ are as in \eqref{eq:A_continuous_coercive}.

\begin{proposition}\label{eq:best_approximation}
Let $u \in \mathcal{W}_\mu(\mathbb{R}_+;\mathcal{H},\mathcal{V})$ for some $\mu>0$ be the solution to \eqref{eq:wave}--\eqref{eq:ic2} 
and let $u^{\normalfont\text{(rb)}}_R$ be the solution to Problem~\ref{pr:sdpr}
with $\cVrb$ as in \eqref{eq:min_laplace_domain}.

Then, it holds that
\begin{equation}
\begin{aligned}
	\norm{
		u
		-
		u^{\normalfont\text{(rb)}}_R
	}_{L^2(\mathfrak{I};\mathcal{V})}
    +
    &
	\norm{
		\partial_t 
		\left(
			u
			-
			u^{\normalfont\text{(rb)}}_R
		\right)
	}_{L^2(\mathfrak{I};\mathcal{H})}
	\\
	&
    \lesssim
    \sqrt{\frac{C_\mathsf{A}}{c_\mathsf{A}}}
   \max\left\{T,T^2, \frac{T}{\min\{1,\sqrt{ c_{\mathsf{A}} }\}}\right\}
    \norm{
        \left(
            \normalfont\text{Id}
            -
            \mathsf{P}^{\normalfont\text{(rb)}}_R
        \right) \partial^2_t u
    }_{L^2(\mathfrak{I};\mathcal{V})},
\end{aligned}
\end{equation}
where $\mathsf{P}^{\normalfont\text{(rb)}}_R: \mathcal{V} 
\rightarrow \mathcal{V}^{\normalfont\text{(rb)}}_R$ denotes
the $\mathcal{V}$-orthogonal projection onto $\mathcal{V}^{\normalfont\text{(rb)}}_R$.
\end{proposition}

\begin{proof}
For a.e. $t \in \IR_{+}$ define 
\begin{equation}\label{eq:def_eta}
	\eta^{\text{(rb)}}_R(t) 
	\coloneqq 
	u^{\normalfont\text{(rb)}}_R(t)
	   -
	\left(\Pi^{\text{(rb)}}_R u\right)(t)
	\in
	\IV^{\text{(rb)}}_R.
\end{equation} 
By subtracting \eqref{eq:wave} and \eqref{eq:semi_discrete_rom},
and recalling that by definition $\cVrb\subset\cV$ we obtain
\begin{equation}\label{eq:subs_rb}
	\dual{
		\partial^2_t
		\left(
			u^{\normalfont\text{(rb)}}_R(t)
			-
			u(t)
		\right)
	}{
		v^{\text{(rb)}}_R
	}_{\mathcal{V}^\star\times \mathcal{V}}	
	+
	\mathsf{a}
	\dotp{
		u^{\normalfont\text{(rb)}}_R(t)
		-
		u(t)
	}{
		v^{\text{(rb)}}_R
	}
	=
	0,
    \quad
    \forall v^{\text{(rb)}}_R\in \IV^{\text{(rb)}}_R.
\end{equation}
Therefore, it follows from \eqref{eq:subs_rb} and the definition
of the elliptic projection, stated in \eqref{eq:elliptic_projection}, that
$\eta^{\text{(rb)}}_R(t)$ satisfies
for a.e.~$t \in \mathfrak{I}$
\begin{equation}\label{eq:residual_eta}
\begin{aligned}
	\dual{	
		\partial^2_t
		\eta^{\text{(rb)}}_R(t)
	}{
		v^{\text{(rb)}}_R
	}_{\mathcal{V}^\star\times \mathcal{V}}	
    +
	\mathsf{a}
	\dotp{
		\eta^{\text{(rb)}}_R(t)
	}{
		v^{\text{(rb)}}_R
	}
	=
   	&
	\dual{
		\left(
			\text{Id}
			-
			\Pi^{\text{(rb)}}_R
		\right)
		\partial^2_t
		 u(t) 
	}{
		v^{\text{(rb)}}_R
	}_{\mathcal{V}^\star\times \mathcal{V}}	,	
    \quad
    \forall v^{\text{(rb)}}_R \in \IV^{\text{(rb)}}_R,
\end{aligned}
\end{equation}
equipped with vanishing initial conditions.
Set $v^{\text{(rb)}}_R = \partial_t \eta^{\text{(rb)}}_R(t) \in  \IV^{\text{(rb)}}_R$, thus
\begin{equation}
\begin{aligned}
    \frac{1}{2}
    \partial_t
    \left(
    	\dotp{	
    		\partial_t
    		\eta^{\text{(rb)}}_R(t)
    	}{
    	   \partial_t 
           \eta^{\text{(rb)}}(t)
    	}_{\mathcal{H}}
        +
        \mathsf{a}
        \dotp{
            \eta^{\text{(rb)}}_R(t)
        }{
            \eta^{\text{(rb)}}_R(t)
        }
    \right)
    =
    \dotp{
		\left(
			\text{Id}
			-
			\Pi^{\text{(rb)}}_R
		\right)
		\partial^2_t
		 u(t) 
	}{
		\partial_t \eta^{\text{(rb)}}_R(t)
	}_{\mathcal{V}^\star\times \mathcal{V}}.
\end{aligned}
\end{equation}
Integrating over $[0,s]$, for any $s \in  \mathfrak{I}$ and $\epsilon>0$, we have
\begin{equation}
\begin{aligned}
   \frac{1}{2} \norm{	
        \partial_t
        \eta^{\text{(rb)}}_R(s)
    }^2_{\mathcal{H}}
	+
	c_{\mathsf{A}}
	\norm{
		\eta^{\text{(rb)}}_R(s)
	}^2_{\mathcal{V}}
	\text{d}\tau
    &
	\leq
    \norm{
    \left(
        \text{Id}
        -
        \Pi^{\text{(rb)}}_R
    \right) \partial^2_t u
    }_{L^2(\mathfrak{I};\mathcal{H})}
	\norm{
		\partial_t \eta^{\text{(rb)}}_R
	}_{L^2(\mathfrak{I};\mathcal{H})}
    \\
    &
    \leq
    \frac{\epsilon}{2}
    \norm{
    \left(
        \text{Id}
        -
       \Pi^{\text{(rb)}}_R
    \right) \partial^2_t u
    }^2_{L^2(\mathfrak{I};\mathcal{H})}
	+
    \frac{1}{2\epsilon}
    \norm{
		\partial_t \eta^{\text{(rb)}}_R
	}^2_{L^2(\mathfrak{I};\mathcal{H})}.
\end{aligned}
\end{equation}
Integrating over $\mathfrak{I}$, we obtain
\begin{equation}
   \frac{1}{2}  \norm{	
        \partial_t
        \eta^{\text{(rb)}}_R
    }^2_{L^2(\mathfrak{I};\mathcal{H})}
	+
	c_{\mathsf{A}}
	\norm{
		\eta^{\text{(rb)}}_R
	}^2_{L^2(\mathfrak{I};\mathcal{V})}
    \leq
    \frac{\epsilon T}{2}
    \norm{
    \left(
        \text{Id}
        -
       \Pi^{\text{(rb)}}_R
    \right) \partial^2_t u
    }^2_{L^2(\mathfrak{I};\mathcal{H})}
	+
    \frac{T}{2\epsilon}
    \norm{
		\partial_t \eta^{\text{(rb)}}_R
	}^2_{L^2(\mathfrak{I};\mathcal{H})}.
\end{equation}
Thus with $\epsilon=2T$, we obtain
\begin{equation}\label{eq:bound_eta_rb}
    \norm{	
        \partial_t
        \eta^{\text{(rb)}}_R
    }_{L^2(\mathfrak{I};\mathcal{H})}
	+
	\norm{
		\eta^{\text{(rb)}}_R
	}_{L^2(\mathfrak{I};\mathcal{V})}
    \leq
    \frac{\sqrt{2}T}{\min\{\frac{\sqrt{2}}{2} ,\sqrt{ c_{\mathsf{A}}}\}}
    \norm{
        \left(
            \text{Id}
            -
            \Pi^{\text{(rb)}}_R
        \right) \partial^2_t u
    }_{L^2(\mathfrak{I};\mathcal{H})}.
\end{equation}
Therefore, the triangle inequality implies that
\begin{equation}
\begin{aligned}
	\norm{
		u
		-
		u^{\normalfont\text{(rb)}}_R
	}_{L^2(\mathfrak{I};\mathcal{V})}
    +
    &
	\norm{
		\partial_t 
		\left(u
		-
		 u^{\normalfont\text{(rb)}}_R
		 \right)
	}_{L^2(\mathfrak{I};\mathcal{H})}
	\\
    \leq
    &
	\norm{
		\eta^{\text{(rb)}}_R
	}_{L^2(\mathfrak{I};\mathcal{V})}
    +
    \norm{	
        \partial_t
        \eta^{\text{(rb)}}_R
    }_{L^2(\mathfrak{I};\mathcal{H})}
    \\
    &
    +
    \norm{
        \left(
            \text{Id}
            -
             \Pi^{\text{(rb)}}_R
        \right)  u
    }_{L^2(\mathfrak{I};\mathcal{V})}
    +
    \norm{
        \left(
            \text{Id}
            -
             \Pi^{\text{(rb)}}_R
        \right)  \partial_t u
    }_{L^2(\mathfrak{I};\mathcal{H})}.
\end{aligned}
\end{equation}
Observe that since $u(0) = \partial_t u(0) = 0$, and using the vector-valued version of
Poincaré's inequality, we obtain
\begin{equation}
    \norm{
    \left(
        \text{Id}
        -
         \Pi^{\text{(rb)}}_R
    \right)
    u}_{L^2(\mathfrak{I};\mathcal{V})}
    \leq
    T
    \norm{
    \left(
        \text{Id}
        -
        \Pi^{\text{(rb)}}_R
    \right)
    \partial_t u}_{L^2(\mathfrak{I};\mathcal{V})}
\end{equation}
and
\begin{equation}
    \norm{
        \left(
            \text{Id}
            -
             \Pi^{\text{(rb)}}_R
        \right)
        \partial_t u
    }_{L^2(\mathfrak{I};\mathcal{V})}
    \leq
    T
    \norm{
    \left(
        \text{Id}
        -
         \Pi^{\text{(rb)}}_R
    \right)
    \partial^2_t u}_{L^2(\mathfrak{I};\mathcal{V})},
\end{equation}
which shows that
\begin{equation}
    \norm{
        \left(
            \text{Id}
            -
            \Pi^{\text{(rb)}}_R
        \right)
        u
    }_{L^2(\mathfrak{I};\mathcal{V})}
    \leq
    T^2
    \norm{
    \left(
        \text{Id}
        -
         \Pi^{\text{(rb)}}_R
    \right)
    \partial^2_t u}_{L^2(\mathfrak{I};\mathcal{V})}.
\end{equation}
Recalling \eqref{eq:bound_eta_rb} and the continuity of the embedding
$\mathcal{V} \subset \mathcal{H}$, we obtain
\begin{equation}
\begin{aligned}
	\norm{
		u
		-
		u^{\normalfont\text{(rb)}}_R
	}_{L^2(\mathfrak{I};\mathcal{V})}
    +
    &
	\norm{
		\partial_t 
		\left(u
		-
		 u^{\normalfont\text{(rb)}}_R
		 \right)
	}_{L^2(\mathfrak{I};\mathcal{H})}
	\\
	&
    \lesssim
    \max
    \left\{
        T,T^2,\frac{\sqrt{2}T}{\min\{\frac{\sqrt{2}}{2} ,\sqrt{ c_{\mathsf{A}}}\}}
    \right\}
    \norm{
        \left(
            \text{Id}
            -
        		\Pi^{\text{(rb)}}_R
        \right) \partial^2_t u
    }_{L^2(\mathfrak{I};\mathcal{V})}.
\end{aligned}
\end{equation}
Together with \eqref{eq:cea_elliptic_projection}, this yields the final result. 
\end{proof}

As a consequence of Proposition~\ref{eq:best_approximation} and 
Lemma~\ref{lem:approximation_Hardy_spaces_2}, we may state the following result. 

\begin{theorem}\label{thm:approximation_Hardy_spaces_3}
Let $u \in \mathcal{W}_\mu(\mathbb{R}_+;\mathcal{H},\mathcal{V})$ for some $\mu>0$ be the solution to \eqref{eq:wave}--\eqref{eq:ic2} 
and let $u^{\normalfont\text{(rb)}}_R$ be the solution to Problem~\ref{pr:sdpr}
with $\cVrb$ as in \eqref{eq:min_laplace_domain}.
Furthermore, assume that $p \in \mathcal{V}$ and 
$\mathsf{A} p \in \mathcal{V}$. 
Then, for any $\eta \in (0,\mu)$ it holds that
\begin{equation}
\begin{aligned}
	\norm{
		u
		-
		u^{\normalfont\text{(rb)}}_R
	}_{L^2(\mathfrak{I};\mathcal{V})}
    +
    &
	\norm{
		\partial_t 
		\left(
			u
			-
			u^{\normalfont\text{(rb)}}_R
		\right)
	}_{L^2(\mathfrak{I};\mathcal{H})}
	\\
    \lesssim
    &
    C(T,\mathsf{A})
    \left(
    \frac{\exp\left(-(\mu-\eta) t_0+\frac{(\mu+\eta)^2}{\alpha^2}\right)}{\alpha^3(\mu-\eta)}
    \left(
		\mu^4
		+
		\alpha^4
    \right)
        \right.
        \\
        &
        \left.
        +
		\alpha \mu^3
    		\exp\left(-\mu t_0 +\frac{\mu^2}{\alpha^2}\right)
		\left( R \eta\right)^\frac{7}{6}
    \right)
	\exp
	\left(
		-\left(\frac{\pi R \eta}{2\alpha} \right)^{\frac{2}{3}}
	\right)
	 \norm{p}_\mathcal{H}
	\\
	+
	&
	C(T,\mathsf{A})
	\left(
    \frac{1}{\sqrt{\mu^3}}
    \left(
    \frac{
        \norm{\mathsf{A}p}_{\mathcal{H}}
        }{
        \mu  \min\{1, c_{\mathsf{A}}\}
        }
        +
        \norm{p}_\mathcal{V}
    \right)
	\int_{-\infty}^{0}
	\snorm{\partial^2_t  q (t)}
	\exp(-\mu t)
	\normalfont\text{d} t
	\right.
	\\
	&
	\left.
	+
	\snorm{q(0)}
	\frac{\norm{p}_{\mathcal{V}}}{\sqrt{\mu}}
	+
	\snorm{\partial_t q(0)}
	\frac{\norm{p}_{\mathcal{V}}}{\sqrt{\mu^3}}
	+
	\frac{
        \norm{\mathsf{A}p}_{\mathcal{H}}
	}{
		\mu
	}
	\left(
		\frac{\snorm{q(0)}}{\sqrt{\mu}}
		+
		\frac{\snorm{\partial_t q(0)}}{\sqrt{\mu^3}}
	\right)
	\right).
\end{aligned}
\end{equation}
where
\begin{equation}\label{eq:constant_C_T_A}
	C(T,\mathsf{A})
	\coloneqq
	\exp(\mu T)
    \sqrt{\frac{C_\mathsf{A}}{c_\mathsf{A}}}
    \max
    \left\{
        T,T^2,\frac{\sqrt{2}T}{\min\{\frac{\sqrt{2}}{2} ,\sqrt{ c_{\mathsf{A}}}\}}
    \right\}.
\end{equation}
\end{theorem}

\section{Fully Discrete Error Analysis}\label{sec:fully_discrete}
For a given $\mu>0$ and $\mathcal{V}_R$ any finite-dimensional subspace of dimension $R$
of $\mathcal{V}$, we define
\begin{subequations}
\begin{align}
	\varepsilon(\mathcal{V}_R)
	&
	\coloneqq
	\norm{
		\partial^2_t u
		-
		\mathsf{P}_{\mathcal{V}_R}
		\partial^2_t u
	}^2_{L^2_\mu(\mathbb{R}_+;\mathcal{V})} 
     \label{eq:error_1}
    \\
	\widetilde{\varepsilon}(\mathcal{V}_R)
	&
	\coloneqq
		\int_{-\infty}^{+\infty}
	\norm{
		\widehat{U}^{(2)}(\mu+\imath \tau) 
		-
		\mathsf{P}_{\mathcal{V}_R}
		\widehat{U}^{(2)}(\mu+\imath \tau)
	}^2_{\mathcal{V}}
	\text{d}\tau
     \label{eq:error_2}
	\quad
	\text{and}
	\\
	\varepsilon^{(M)}(\mathcal{V}_R)
	&
	\coloneqq
    	\sum_{j=-M}^{M} 
	\omega_j
	\norm{
		\widehat{U}^{(2)}(s_j) 
		- 
		\mathsf{P}_{\mathcal{V}_R}
		\widehat{U}^{(2)}(s_j) 
	}^2_{\cV}
    \label{eq:error_3},
\end{align}
\end{subequations}
where $u$ is the solution to \eqref{eq:wave}--\eqref{eq:ic2} 
and, for each $s \in \Pi_\mu$, $\widehat{U}^{(2)}(s)$ corresponds to the solution to
Problem~\ref{pbrm:laplace_discrete},and $P_{\cV_R}:\cV\to\cV_R$ is the $\cV$-orthogonal projection onto $\cV_R$.

\begin{lemma}\label{eq:error_discrete_norm}
Let $u \in \mathcal{W}_\mu(\mathbb{R}_+;\mathcal{H},\mathcal{V})$ for some $\mu>0$ be the solution to \eqref{eq:wave}--\eqref{eq:ic2}
and, for each $s \in \Pi_\mu$, let $\widehat{U}^{(2)}(s)$ correspond to the solution to
Problem~\ref{pbrm:laplace_discrete}.
Furthermore, assume that $p \in \mathcal{V}$ and 
$\mathsf{A} p \in \mathcal{V}$.  
For each $\eta \in (0,\mu)$, consider the weights and points
\begin{equation}
	\omega_k =
	\left( \frac{\pi \alpha^2 \eta}{M^2} \right)^{\frac{1}{3}}
	\quad
	\text{and}
	\quad
	s_ k = \mu +\imath k \left( \frac{\pi \alpha^2 \eta}{M^2} \right)^{\frac{1}{3}},
	\quad
	k=-M,\dots,M,
\end{equation}
respectively. Then, for any $\eta \in (0,\mu)$, it holds that
\begin{equation}
\begin{aligned}
	\snorm{
		\varepsilon(\mathcal{V}_R )
		-
		\varepsilon^{(M)}(\mathcal{V}_R)
	}
	\lesssim
	&
	\left(
       	\norm{
    			\widehat{U}^{(2)}(\mu+\imath \cdot) 
    		}^2_{\mathbb{H}^2(\mathscr{D}_\eta;\mathcal{V})}
	\right.
	\\
	&
	\left.
        +
		\frac{e^{-2\mu t_0+2\frac{\mu^2}{\alpha^2}}}{\alpha^5 \mu^2 \min\{1, c^2_{\mathsf{A}}\}}
    		(\alpha^2+\mu^2)^4
    		\left(\frac{\pi M \eta}{\alpha} \right)^\frac{7}{3}
            \norm{p}^2_{\mathcal{H}}
	\right)
	e^{-2 \left(\frac{\pi M \eta}{\alpha} \right)^{\frac{2}{3}}}
	\\
	&
	+
	E,
\end{aligned}
\end{equation}
where $	\varepsilon(\mathcal{V}_R )$ and $\varepsilon^{(M)}(\mathcal{V}_R)$
are as in \eqref{eq:error_1} and \eqref{eq:error_3}, respectively, and
\begin{equation}\label{eq:discrete_norm_Error_E}
\begin{aligned}
	E
	\coloneqq
	&
	\left(
    \frac{e^{-\mu t_0+\frac{\mu^2}{\alpha^2}}}{\alpha^3 \mu \min\{1, c_{\mathsf{A}}\}}
    \left(
        \mu^9
        +
        \alpha^9
    \right)^{\frac{1}{2}}
	\norm{p}_\mathcal{H}
	+
	\frac{
		\norm{\partial^2_t q}_{L^2_\mu(\mathbb{R}_+)}
		\norm{p}_{\mathcal{H}}
	}{
		\mu \min\{1, c_{\mathsf{A}}\}
	}
	\right.
	\\
	&
	+
	\left.
	\left(
	\norm{p}_{\mathcal{V}}
	+
	\frac{
        \norm{\mathsf{A}p}_{\mathcal{H}}
	}{
		\mu  \min\{1, c_{\mathsf{A}}\}
	}
	\right)
	\left(
		\frac{\snorm{q(0)}}{\sqrt{\mu}}
		+
		\frac{\snorm{\partial_t q(0)}}{\sqrt{\mu^3}}
	\right)
	\right)
	\left(
	\left(
	\norm{p}_{\mathcal{V}}
	+
	\frac{
        \norm{\mathsf{A}p}_{\mathcal{H}}
	}{
		\mu  \min\{1, c_{\mathsf{A}}\}
	}
	\right)
	\left(
		\frac{\snorm{q(0)}}{\sqrt{\mu}}
		+
		\frac{\snorm{\partial_t q(0)}}{\sqrt{\mu^3}}
	\right)
	\right.
	\\
	&
	+
	\left.
    \frac{1}{\sqrt{\mu^3}}
    \left(
    \frac{
        \norm{\mathsf{A}p}_{\mathcal{H}}
        }{
        \mu \min\{1, c_{\mathsf{A}}\}
        }
        +
        \norm{p}_\mathcal{V}
    \right)
	\int_{-\infty}^{0}
	\snorm{\partial^2_t  q (t)}
	e^{-\mu t}
	\normalfont\text{d} t
	\right).
\end{aligned}
\end{equation}
\end{lemma}

\begin{proof}
For any $\mu >0$, recalling Theorem~\ref{thm:paley_wiener} (Paley-Wiener)
and Proposition~\ref{prop:properties_hardy}, item (i),
we can write $\varepsilon(\cV_R)$ in Laplace domain as follows
\begin{equation}
\begin{aligned}
	\varepsilon(\mathcal{V}_R)
	&
	=
	\norm{
		\mathcal{L}\left\{\partial^2_t u\right\}
		-
		\mathsf{P}_{\mathcal{V}_R}
		\mathcal{L}\left\{\partial^2_t u\right\}
	}^2_{\mathscr{H}^2_\mu(\mathcal{V})}
	\\
	&
	=
	\frac{1}{2\pi }
		\int_{-\infty}^{+\infty}
		\norm{
    		\mathcal{L}\left\{\partial^2_t u \right\} (\mu+\imath \tau)
    		-
    		\mathsf{P}_{\mathcal{V}_R}
    		\mathcal{L} \left\{\partial^2_t u \right\}(\mu+\imath \tau)
		}^2_\mathcal{V}
		\normalfont\text{d} \tau.
\end{aligned}
\end{equation}
Setting
	$w^\pm(s)
	\coloneqq
	\mathcal{L}\left\{\partial^2_t u \right\}(s) \pm \widehat{U}^{(2)}(s)
	\in \mathcal{V}$,
	$s \in \Pi_\mu$,
using the parallelogram law and the Cauchy-Schwarz inequality,
we further obtain that 
\begin{equation}
\begin{aligned}
		\norm{
    		\mathcal{L}\left\{\partial^2_t u \right\}(s)
    		-
    		\mathsf{P}_{\mathcal{V}_R}
    		\mathcal{L} \left\{\partial^2_t u \right\}(s)
		}^2_\mathcal{V}
		&
		-
		\norm{
    			\widehat{U}^{(2)}(s)
    			-
    			\mathsf{P}_{\mathcal{V}_R}
    			\widehat{U}^{(2)}(s)
		}^2_\mathcal{V}
	\\
	&
	\leq
	\norm{
    	w^+(s)
    	-
    	\mathsf{P}_{\mathcal{V}_R}
    	w^+(s)
	}_\mathcal{V}
	\norm{
    	w^-(s)
    	-
    	\mathsf{P}_{\mathcal{V}_R}
    	w^-(s)
	}_\mathcal{V}
	\\
	&
	\leq
	\left(
	\norm{
    	\widehat{U}^{(2)}(s)
	}_\mathcal{V}
	+
	\norm{
    	\mathcal{L}\left\{\partial^2_t u \right\}(s)
	}_\mathcal{V}
	\right)
	\norm{
    	w^-(s)
	}_\mathcal{V}.
\end{aligned}
\end{equation}
Recalling the definition of $\widetilde{\varepsilon}^{(M)}(\mathcal{V}_R)$ given in \eqref{eq:error_2}, we therefore have that
\begin{equation}
\begin{aligned}
	\snorm{
		\varepsilon(\mathcal{V}_R )
		-
		\widetilde{\varepsilon}(\mathcal{V}_R )
	}
	\lesssim
	&
	\left(
        \left(
          	\int_{-\infty}^{+\infty}
          	\norm{
             	\widehat{U}^{(2)}(\mu+\imath \omega)
          	}^2_\mathcal{V} 
    		\normalfont\text{d} \omega
        \right)^{\frac{1}{2}}
   	\right.
	\\
	&
	\left.
      	+
        \left(
      	    \int_{-\infty}^{+\infty}
          	\norm{
             	\mathcal{L}\left\{\partial^2_t u \right\}(\mu+\imath \omega)
          	}^2_\mathcal{V}
    		\normalfont\text{d} \omega
        \right)^{\frac{1}{2}}
	\right)
    \left(
		\int_{-\infty}^{+\infty}
		\norm{
			w^-(\mu+\imath \omega)
		}^2_\mathcal{V}
		\normalfont\text{d} \omega
    \right)^{\frac{1}{2}}.
\end{aligned}
\end{equation}
%
Next, we bound the terms in the previous inequality.
It follows from Lemma~\ref{lmm:regularity_U_n}, item (i),
and \eqref{eq:integra_x_to_8} that
\begin{equation}
\begin{aligned}
    \left(
        \int_{-\infty}^{+\infty}
        \norm{
            \widehat{U}^{(2)}(\mu+\imath \omega)
        }^2_\mathcal{V} 
        \normalfont\text{d} \omega
    \right)^{\frac{1}{2}}
    \\
    &
    \hspace{-3cm}
    \lesssim
    \frac{e^{-\mu t_0+\frac{\mu^2}{\alpha^2}}}{\alpha^3 \mu \min\{1, c_{\mathsf{A}}\}}
    \left(
        \int_{-\infty}^{+\infty}
        (\mu^2+\omega^2)^{4}
        e^{-2\frac{\omega^2}{\alpha^2}}
        \normalfont\text{d} \omega
    \right)^{\frac{1}{2}}
	\norm{p}_\mathcal{H}
    \\
    &
    \hspace{-3cm}
    \lesssim
    \frac{e^{-\mu t_0+\frac{\mu^2}{\alpha^2}}}{\alpha^3 \mu \min\{1, c_{\mathsf{A}}\}}
    \left(
        \int_{0}^{\mu}
        (\mu^2+\omega^2)^{4}
        e^{-2\frac{\omega^2}{\alpha^2}}
        \normalfont\text{d} \omega
        +
        \int_{\mu}^{+\infty}
        (\mu^2+\omega^2)^{4}
        e^{-2\frac{\omega^2}{\alpha^2}}
        \normalfont\text{d} \omega
    \right)^{\frac{1}{2}}
	\norm{p}_\mathcal{H}
    \\
    &
    \hspace{-3cm}
    \lesssim
    \frac{e^{-\mu t_0+\frac{\mu^2}{\alpha^2}}}{\alpha^3 \mu \min\{1, c_{\mathsf{A}}\}}
    \left(
        \mu^9
        +
        \alpha^9
        \int_{\sqrt{2}\frac{\mu}{\alpha}}^{+\infty}
        \tau^8
        e^{-\tau^2}
        \normalfont\text{d} \tau
    \right)^{\frac{1}{2}}
	\norm{p}_\mathcal{H}
    \\
    &
    \hspace{-3cm}
    \lesssim
    \frac{e^{-\mu t_0+\frac{\mu^2}{\alpha^2}}}{\alpha^3 \mu \min\{1, c_{\mathsf{A}}\}}
    \left(
        \mu^9
        +
        \alpha^9
    \right)^{\frac{1}{2}}
	\norm{p}_\mathcal{H}
    \eqqcolon E_1.
\end{aligned}
\end{equation}
From Lemma~\ref{lmm:error_U}, item(i), one has that
\begin{equation}
\begin{aligned}
    \left(
        \int_{-\infty}^{+\infty}
        \norm{
            \mathcal{L}\left\{\partial^2_t u \right\}(\mu+\imath \omega)
        }^2_\mathcal{V}
        \normalfont\text{d} \omega
    \right)^{\frac{1}{2}}
    &
    \lesssim
	\frac{
		\norm{\partial^2_t q}_{L^2_\mu(\mathbb{R}_+)}
		\norm{p}_{\mathcal{H}}
	}{
		\mu \min\{1, c_{\mathsf{A}}\}
	}
	\\
	&
	+
	\left(
	\norm{p}_{\mathcal{V}}
	+
	\frac{
        \norm{\mathsf{A}p}_{\mathcal{H}}
	}{
		\mu  \min\{1, c_{\mathsf{A}}\}
	}
	\right)
	\left(
		\frac{\snorm{q(0)}}{\sqrt{\mu}}
		+
		\frac{\snorm{\partial_t q(0)}}{\sqrt{\mu^3}}
	\right) \eqqcolon E_2,
\end{aligned}
\end{equation}
and from item (ii)
\begin{equation}
\begin{aligned}
    \left(
		\int_{-\infty}^{+\infty}
		\norm{
			w^-(\mu+\imath \omega)
		}^2_\mathcal{V}
		\normalfont\text{d} \omega
    \right)^{\frac{1}{2}}
    \lesssim
	&
    \frac{1}{\sqrt{\mu^3}}
    \left(
    \frac{
        \norm{\mathsf{A}p}_{\mathcal{H}}
        }{
        \mu  \min\{1, c_{\mathsf{A}}\}
        }
        +
        \norm{p}_\mathcal{V}
    \right)
	\int_{-\infty}^{0}
	\snorm{\partial^2_t  q (t)}
	e^{-\mu t}
	\normalfont\text{d} t
	\\
	&
	+
	\left(
	\norm{p}_{\mathcal{V}}
	+
	\frac{
        \norm{\mathsf{A}p}_{\mathcal{H}}
	}{
		\mu  \min\{1, c_{\mathsf{A}}\}
	}
	\right)
	\left(
		\frac{\snorm{q(0)}}{\sqrt{\mu}}
		+
		\frac{\snorm{\partial_t q(0)}}{\sqrt{\mu^3}}
	\right)\eqqcolon E_3,
\end{aligned}
\end{equation}
which allows us to bound $\snorm{
		\varepsilon(\mathcal{V}_R )
		-
		\widetilde{\varepsilon}(\mathcal{V}_R )
	}$ in terms of $E=(E_1+E_2)E_3$ defined in \eqref{eq:discrete_norm_Error_E}.
It remains to bound $\snorm{
		\widetilde{\varepsilon}(\mathcal{V}_R )
        -
        \varepsilon^{(M)}(\cV_R)}$
        and to apply the triangle inequality.
To that end, we observe that the map 
\begin{equation}
	\tau
	\mapsto
    \norm{
		\widehat{U}^{(2)}(\mu+\imath \tau) 
		-
		\mathsf{P}_{\mathcal{V}_R}
		\widehat{U}^{(2)}(\mu+\imath \tau)
	}^2_{\mathcal{V}},
    \quad\tau \in \mathbb{R},
\end{equation}
admits an analytic extension to $\mathscr{D}_\eta$, for each $\eta \in (0,\mu)$, given by
\begin{equation}
	g(z)
	=
	\dotp{
		\widehat{U}^{(2)}(\mu+\imath z) 
		-
		\mathsf{P}_{\mathcal{V}_R}
		\widehat{U}^{(2)}(\mu+\imath z)
	}{
		\overline{\widehat{U}^{(2)}(\mu-\imath z)
		-
		\mathsf{P}_{\mathcal{V}_R}
		\widehat{U}^{(2)}(\mu-\imath z)}
	}_{\mathcal{V}},
	\quad
	z \in \mathscr{D}_\eta,
\end{equation}
where $\dotp{\cdot}{\cdot}_\mathcal{V}$ corresponds to the
$\mathcal{V}$-inner product, which is a sesquilinear form, and
we have used that $\overline{\widehat{U}^{(2)}(s)} = {\widehat{U}^{(2)}(\overline{s})} $,
for each $s \in \Pi_\mu$. The map $z \mapsto g(z)$ is indeed 
analytic, because it is the composition of the map
$\mathcal{V} \times \mathcal{V} \ni (v,w) \mapsto (v,\overline{w}) \in \mathbb{C}$, which is linear in each component,
thus analytic, and the map $z \mapsto \widehat{U}^{(2)}(\mu\pm\imath z) $,
which is analytic as a consequence of Lemma~\ref{lmm:regularity_U_n}, item (ii). 
Furthermore, for each $\eta\in (0,\mu)$
\begin{equation}
\begin{aligned}
	\norm{g}_{\mathbb{H}^1(\mathscr{D}_\eta;\mathbb{C})}
	&
	\leq
	\norm{
			\widehat{U}^{(2)}(\mu+\imath \cdot) 
		-
		\mathsf{P}_{\mathcal{V}_R}
		\widehat{U}^{(2)}(\mu+\imath \cdot)
	}_{\mathbb{H}^2(\mathscr{D}_\eta;\mathcal{V})}
	\norm{
			\widehat{U}^{(2)}(\mu-\imath \cdot) 
		-
		\mathsf{P}_{\mathcal{V}_R}
		\widehat{U}^{(2)}(\mu-\imath \cdot)
	}_{\mathbb{H}^2(\mathscr{D}_\eta;\mathcal{V})}
	\\
	&
	\leq
	\norm{
			\widehat{U}^{(2)}(\mu+\imath \cdot) 
	}_{\mathbb{H}^2(\mathscr{D}_\eta;\mathcal{V})}
	\norm{
			\widehat{U}^{(2)}(\mu-\imath \cdot) 
	}_{\mathbb{H}^2(\mathscr{D}_\eta;\mathcal{V})}
	\\
	&
	\leq
	\norm{
			\widehat{U}^{(2)}(\mu+\imath \cdot) 
	}^2_{\mathbb{H}^2(\mathscr{D}_\eta;\mathcal{V})}.
\end{aligned}
\end{equation}
Recalling the definition of the integral operator $\cI$, its trapezoidal rule approximation $\cT$, and the corresponding truncation $\cT_M$ from \eqref{eq:def_I_T}, it follows from Proposition~\ref{prop:integral_sinc}
that for each $\vartheta>0$ 
\begin{equation}
\begin{aligned}
	\snorm{
		\widetilde{\varepsilon}(\mathcal{V}_R)
		-
		{\varepsilon}^{(M)}(\mathcal{V}_R)
	}
	&
	\leq
	\snorm{
	       \mathcal{I}(g)
          -
           \mathcal{Q}(g,\vartheta)
	}
	+
	\snorm{
	       \mathcal{Q}(g,\vartheta)
     	-
        \mathcal{Q}_M(g,\vartheta)
	}
	\\
	&
	\leq
    \frac{
        e^{-\pi \eta/ \vartheta}
    }{
        2\sinh(\pi \eta/\vartheta)
    }
    \norm{g}_{\mathbb{H}^1(\mathscr{D}_\eta;\mathbb{C})}
   +
	\snorm{
	       \mathcal{Q}(g,\vartheta)
     	-
        \mathcal{Q}_M(g,\vartheta)
	}
	\\
	&
	\lesssim
    e^{- 2 \pi \eta / \vartheta}
   	\norm{
			\widehat{U}^{(2)}(\mu+\imath \cdot) 
	}^2_{\mathbb{H}^2(\mathscr{D}_\eta;\mathcal{V})}
    +
   	\snorm{
	       \mathcal{Q}(g,\vartheta)
     	-
        \mathcal{Q}_M(g,\vartheta)
	}.
\end{aligned}
\end{equation}
According to Lemma~\ref{lmm:regularity_U_n} item (i), one has similar as above that
\begin{equation}
\begin{aligned}
	\snorm{
	       \mathcal{Q}(g,\vartheta)
     	-
        \mathcal{Q}_M(g,\vartheta)
	}
	&
	\leq
	\vartheta
	\sum_{\snorm{k} > M}
	\norm{
			\widehat{U}^{(2)}(\mu+\imath k\vartheta ) 
	}^2_{\mathcal{V}}
	\\
	&
	\leq
	\vartheta
	\frac{e^{-2\mu t_0+2\frac{\mu^2}{\alpha^2}}}{\alpha^6 \mu^2 \min\{1, c^2_{\mathsf{A}}\} }
	\sum_{\snorm{k} > M}
    (\mu^2 + k^2 \vartheta^2)^4
	e^{-2\frac{ k^2 \vartheta^2}{\alpha^2}}
	\norm{p}^2_\mathcal{H}
	\\
	&
	\lesssim
	\frac{e^{-2\mu t_0+2\frac{\mu^2}{\alpha^2}}}{\alpha^6 \mu^2 \min\{1, c^2_{\mathsf{A}}\}}
	\int_{M\vartheta} 
	\left(\mu^2+\omega^2\right)^4
	e^{-2\frac{ \omega^2}{\alpha^2}}
	\text{d} \omega
    \norm{p}^2_\mathcal{H}
	\\
	&
	\lesssim
	\frac{e^{-2\mu t_0+2\frac{\mu^2}{\alpha^2}}}{\alpha^5 \mu^2 \min\{1, c^2_{\mathsf{A}}\}}
	\int_{\sqrt{2}\frac{M\vartheta}{\alpha}} 
	\left(\mu^2+\alpha^2\tau^2\right)^4
	e^{-\tau^2}
	\text{d} \tau
    \norm{p}^2_\mathcal{H}.
\end{aligned}
\end{equation}
Assuming $\frac{M\vartheta}{\alpha}>1$, we hence obtain upon integration the bound
\begin{equation}
\begin{aligned}
	\snorm{
	       \mathcal{Q}(g,\vartheta)
     	-
        \mathcal{Q}_M(g,\vartheta)
	}
	&
	&
	\lesssim
	\frac{e^{-2\mu t_0+2\frac{\mu^2}{\alpha^2}}}{\alpha^5 \mu^2 \min\{1, c^2_{\mathsf{A}}\}}
    (\alpha^2+\mu^2)^4
    \left(\frac{M\vartheta}{\alpha}\right)^7
    e^{
    -\left(\sqrt{2}\frac{M\vartheta}{\alpha}\right)^2}
    \norm{p}^2_{\mathcal{H}}.
\end{aligned}
\end{equation}
In particular, by selecting $\vartheta = \left( \frac{ \pi \alpha^2 \eta}{M^2} \right)^{\frac{1}{3}}$,
which also gives $\frac{M\vartheta}{\alpha} = \left(\frac{\pi M \eta}{\alpha} \right)^{\frac{1}{3}}$, we obtain
\begin{equation}
\begin{aligned}
	\snorm{
		\widetilde{\varepsilon}(\mathcal{V}_R)
		-
		{\varepsilon}^{(M)}(\mathcal{V}_R)
	}
	\lesssim
	&
	\left(
       	\norm{
    			\widehat{U}^{(2)}(\mu+\imath \cdot) 
    		}^2_{\mathbb{H}^2(\mathscr{D}_\eta;\mathcal{V})}
	\right.
	\\
	&
	\left.
        +
		\frac{e^{-2\mu t_0+2\frac{\mu^2}{\alpha^2}}}{\alpha^5 \mu^2 \min\{1, c^2_{\mathsf{A}}\}}
    		(\alpha^2+\mu^2)^4
    		\left(\frac{\pi M \eta}{\alpha} \right)^\frac{7}{3}
	            \norm{p}^2_{\mathcal{H}}\right)
	e^{-2\left(\frac{\pi M \eta}{\alpha} \right)^{\frac{2}{3}}}.
\end{aligned}
\end{equation}
This bound, together with Lemma~\ref{lmm:regularity_U_n} item (i), yields the final result. 
\end{proof}

\begin{theorem}[Fully Discrete Error Estimate]\label{thm:fully_discrete_error}
Consider the setting of Lemma~\ref{eq:error_discrete_norm}. 
Let $u^{\normalfont\text{(rb)}}_{R,M}$ be the solution to Problem~\ref{pr:sdpr} in $\cV^\rb_{R,M}$,
with $\cV^\rb_{R,M}$ as in \eqref{eq:fPOD}.
Then, for any $\eta \in (0,\mu)$, it holds that
\begin{equation}\label{eq:error_bound_fully_discrete}
\begin{aligned}
	\norm{
		u
		-
		u^{\normalfont\text{(rb)}}_{R,M}
	}_{L^2(\mathfrak{I};\mathcal{V})}
    +
    &
	\norm{
		\partial_t 
		\left(
			u
			-
			u^{\normalfont\text{(rb)}}_{R,M}
		\right)
	}_{L^2(\mathfrak{I};\mathcal{H})}
	\\
	&
    \hspace{-2cm}
    \lesssim
    C(T,\mathsf{A})
    \left(
       	\norm{
    		\widehat{U}^{(2)}(\mu+\imath \cdot) 
    	}_{\mathbb{H}^2(\mathscr{D}_\eta;\mathcal{V})}
    \right.
    \\
    &
    \left.
        +
		\frac{e^{-\mu t_0+\frac{\mu^2}{\alpha^2}}}{\alpha^5 \mu^2 \min\{1, c_{\mathsf{A}}\}}
    		(\alpha^2+\mu^2)^2
    		\left(\frac{\pi M \eta}{\alpha} \right)^\frac{7}{6}
        \norm{p}^2_{\mathcal{V}} 
	\right)
	e^{-\left(\frac{\pi M \eta}{\alpha} \right)^{\frac{2}{3}}}
	\\
	&
    \hspace{-2cm}
    + 
    \sqrt{
        \varepsilon
	\left(
		\IV^{\normalfont\text{(rb)}}_{R}
	\right)
    }
    +
    e^{\mu T}E,
\end{aligned}
\end{equation}
where $C(T,\mathsf{A})>0$ and $E$ are as in \eqref{eq:constant_C_T_A}
and \eqref{eq:discrete_norm_Error_E}, respectively, and 
\begin{equation}
\begin{aligned}
    \sqrt{
        \varepsilon
	\left(
		\IV^{\normalfont\text{(rb)}}_{R}
	\right)
    }
    \lesssim
    &
    e^{\mu T}
    \left(
    \frac{e^{-(\mu-\eta) t_0+\frac{(\mu+\eta)^2}{\alpha^2}}}{\alpha^3(\mu-\eta)}
    \left(
		\mu^4
		+
		\alpha^4
    \right)
        \right.
        \\
        &
        \left.
        +
   \frac{
   	   \mu^{\frac{1}{6}}
   		(\alpha^2+\mu^2)^2 
   }{
   	\alpha^{\frac{11}{3}}
   }
    \frac{
        e^{-\mu t_0 + \frac{\mu^2}{\alpha^2}}
        }{
        \min\{1, c_{\mathsf{A}}\}
    }
    R^\frac{7}{6}
    \right)
	e^{
		-\left(\frac{\pi R \eta}{2\alpha} \right)^{\frac{2}{3}}
	}
	 \norm{p}_\mathcal{H}
	\\
	&
	+
    \frac{e^{\mu T}}{\sqrt{\mu^3}}
    \left(
    \frac{
        \norm{\mathsf{A}p}_{\mathcal{H}}
        }{
        \mu  \min\{1, c_{\mathsf{A}}\}
        }
        +
        \norm{p}_\mathcal{V}
    \right)
	\int_{-\infty}^{0}
	\snorm{\partial^2_t  q (t)}
	e^{-\mu t}
	\normalfont\text{d} t
	\\
	&
	+
	e^{\mu T}
	\left(
      	\norm{p}_{\mathcal{V}}
      	+
      	\frac{
			\norm{\mathsf{A}p}_{\mathcal{H}}
		}{
			\mu \min\{1, c_{\mathsf{A}}\}
		}
	\right)
	\left(
		\frac{\snorm{q(0)}}{\sqrt{\mu}}
		+
		\frac{\snorm{\partial_t q(0)}}{\sqrt{\mu^3}}
	\right).
\end{aligned}
\end{equation}
\end{theorem}

\begin{proof}
Exactly as in the proof of Proposition~\ref{eq:best_approximation},
and using that $\|v\|_{L^2(\mathfrak{I};\cV)}\leq e^{\mu T}\|v\|_{L^2_\mu(\IR_+;\cV)}$,
we have that
\begin{equation}
\begin{aligned}
	\norm{
		u
		-
		u^{\normalfont\text{(rb)}}_{R,M}
	}^2_{L^2(\mathfrak{I};\mathcal{V})}
    +
	\norm{
		\partial_t 
		\left(
			u
			-
			u^{\normalfont\text{(rb)}}_{R,M}
		\right)
	}^2_{L^2(\mathfrak{I};\mathcal{H})}
    \lesssim
    C(T,\mathsf{A})
    \varepsilon(\mathcal{V}^{\normalfont\text{(rb)}}_{R,M}).
\end{aligned}
\end{equation}
It follows from Lemma~\ref{eq:error_discrete_norm} that
\begin{equation}
\begin{aligned}
	\varepsilon
	\left(
		\IV^{\normalfont\text{(rb)}}_{R,M}
	\right)
	\leq
	&
	\snorm{
      	\varepsilon
      	\left(
      		\IV^{\normalfont\text{(rb)}}_{R,M}
      	\right)
		-
		\varepsilon^{(M)}
		\left(
			\IV^{\normalfont\text{(rb)}}_{R,M}
		\right)
	}
	+
	\varepsilon^{(M)}
	\left(
		\IV^{\normalfont\text{(rb)}}_{R,M}
	\right)
	\\
	\lesssim
	&
    \left(
       	\norm{
    			\widehat{U}^{(2)}(\mu+\imath \cdot) 
    		}^2_{\mathbb{H}^2(\mathscr{D}_\eta;\mathcal{V})}
        +
		\frac{e^{-2\mu t_0+2\frac{\mu^2}{\alpha^2}}}{\alpha^5 \mu^2 \min\{1, c^2_{\mathsf{A}}\}}
    		(\alpha^2+\mu^2)^4
    		\left(\frac{M \eta}{\alpha} \right)^\frac{7}{3}
	\norm{p}^2_{\mathcal{H}}\right)
	e^{-2 \left(\frac{\pi M \eta}{\alpha} \right)^{\frac{2}{3}}}
    \\
    &
    +
    E
	+
	\varepsilon^{(M)}
	\left(
		\IV^{\normalfont\text{(rb)}}_{R,M}
	\right),
\end{aligned}
\end{equation}
where $E$ is as in \eqref{eq:discrete_norm_Error_E}.
Recalling that $\IV^{\normalfont\text{(rb)}}_{R,M}$ minimizes
$\varepsilon^{(M)}$ defined in \eqref{eq:error_3}, cf.~\eqref{eq:fPOD},
we have that
\begin{equation}
\begin{aligned}
	\varepsilon^{(M)}\left(\IV^{\normalfont\text{(rb)}}_{R,M}\right)
	\leq
    &
	\varepsilon^{(M)}\left(\IV^{\normalfont\text{(rb)}}_{R}\right)
	\\
	\leq
    &
	\snorm{
		\varepsilon^{(M)}\left(\IV^{\normalfont\text{(rb)}}_{R}\right)
		-
		\varepsilon\left(\IV^{\normalfont\text{(rb)}}_{R}\right)
	}
	+
	\varepsilon\left(\IV^{\normalfont\text{(rb)}}_{R}\right)
	\\
	\lesssim
	&
    \left(
       	\norm{
    			\widehat{U}^{(2)}(\mu+\imath \cdot) 
    		}^2_{\mathbb{H}^2(\mathscr{D}_\eta;\mathcal{V})}
        +
		\frac{e^{-2\mu t_0+2\frac{\mu^2}{\alpha^2}}}{\alpha^5 \mu^2 \min\{1, c^2_{\mathsf{A}}\}}
    		(\alpha^2+\mu^2)^4
    		\left(\frac{\pi M \eta}{\alpha} \right)^\frac{7}{3}
            \norm{p}^2_{\mathcal{H}}
	\right)
	e^{-2 \left(\frac{\pi M \eta}{\alpha} \right)^{\frac{2}{3}}}
    \\
    &
	+
	\varepsilon
	\left(
		\IV^{\normalfont\text{(rb)}}_{R}
	\right) + E,
\end{aligned}
\end{equation}
where we have used Lemma~\ref{eq:error_discrete_norm} again.
Finally, we bound
$\varepsilon\left(\IV^{\normalfont\text{(rb)}}_{R}\right)$ using Lemma~\ref{lem:approximation_Hardy_spaces_2} and obtain the final result.
\end{proof}

\begin{remark}[Dependence of the convergence rate in Theorem~\ref{thm:fully_discrete_error}
on $\alpha$ and $t_0$]
In view of Lemma~\ref{lmm:Smallness_Rickert_wavelet}, and as discussed previously in Remark~\ref{rmk:dependence_alpha_t_0}, we examine the structure of the error estimate in Theorem~\ref{thm:fully_discrete_error}. 
In this case, the bound stated in \eqref{eq:error_bound_fully_discrete} consists of three main contributions: 
(i) exponential convergence with respect to $R$, i.e., the dimension of the sought finite-dimensional space, 
(ii) exponential convergence with respect to $M$, i.e., the number of samples in the Laplace domain, and 
(iii) a consistency term decaying exponentially with $\alpha t_0$. 

More precisely, selecting $\mu = \alpha$ and $\eta = \mu/2$, we obtain
\begin{equation}
	\norm{
		u
		-
		u^{\normalfont\text{(rb)}}_{R,M}
	}_{L^2(\mathfrak{I};\mathcal{V})}
    +
	\norm{
		\partial_t 
		\left(
			u
			-
			u^{\normalfont\text{(rb)}}_{R,M}
		\right)
	}_{L^2(\mathfrak{I};\mathcal{H})}
    \lesssim
    e^{\alpha T}
    \left(
    e^{-\left(\frac{\pi R }{4} \right)^{\frac{2}{3}} - \frac{\alpha t_0}{2}}
    +
    e^{-\left(\frac{\pi M }{2} \right)^{\frac{2}{3}} - \frac{\alpha t_0}{2}}
    +
    e^{-\frac{\alpha^2}{4} t_0^2}
    \right),
\end{equation}
where the hidden constant depends polynomially on $\alpha$, $T$, and $p$.

As in Remark~\ref{rmk:dependence_alpha_t_0}, this estimate shows that the exponential convergence rate with respect to $R$ is independent of the parameters characterizing the Ricker wavelet, $\alpha$ and $t_0$, provided that the error remains larger than the consistency term, which scales like $e^{\alpha T - \frac{\alpha^2}{4} t_0^2}$. In addition, the fully discrete scheme exhibits exponential convergence with respect to the number of Laplace-domain samples $M$.
\end{remark}

\section{Numerical Results}
\label{sec:numerical_results}
We present numerical results supporting our theoretical assertions
and demonstrate the computational benefits of the LT-MOR method for the scalar wave equation
(see Example~\ref{example:scalar_wave}).
In particular, we assess the performance of the LT-MOR method
in three key aspects: 
\begin{itemize}
	\item[(i)]
	Accuracy with respect to the high-fidelity solution.
	We set $\mathfrak{I} = (0,T)$ and 
	consider the following metric
	\begin{equation}\label{eq:rel_error}
	\begin{aligned}
		\text{Rel\_Error}^{\normalfont\text{(rb)}}_{R,M}(\mathfrak{I};\mathcal{X})
		&
		=
		\frac{
			\norm{
				u_h - u^{\normalfont\text{(rb)}}_{R,M}
			}_{L^2(\mathfrak{I};\mathcal{X})}
		}{
			\norm{
				u_h
			}_{L^2(\mathfrak{I};X)}
		}
		\\
		&
		\approx
		\frac{
			\left(
			\displaystyle\sum_{j=0}^{N_t}
			\norm{
				u_h(t_j) - u^{\normalfont\text{(rb)}}_{R,M}(t_j)
			}^2_{\cX}
			\right)^{\frac{1}{2}}
		}{
			\left(
			\displaystyle\sum_{j=0}^{N_t}
			\norm{
				u_h(t_j)
			}^2_{\cX}
			\right)^{\frac{1}{2}}
		},
	\end{aligned}
	\end{equation}
	where $\cX \in \{L^2(\Omega),H^1_0(\Omega)\}$, i.e.~we
	compute (an approximation of) the $L^2(\mathfrak{I};\cX)$-relative error for a number of reduced spaces of dimension $R \in \left\{1,\dots,R_{\text{max}}\right\}$, and $t_j = \frac{j}{N_t} T$, $j=1,\dots,N_t$.
	\item[(ii)]
	Accuracy with respect to the number of snapshots in the offline phase, i.e., the 
	number of samples in the Laplace domain.
    \item[(iii)]
	Speed-up with respect to the high-fidelity solver.
\end{itemize}

The FE implementation is conducted in the {\tt MATLAB}
library {\tt Gypsilab} \cite{alouges2018fem}.
We set $\Omega = (-\frac{1}{2},\frac{1}{2})^2$ and
consider
\begin{equation}\label{eq:initial_cond_gauss}
	p (\bm{x} )
	=
	\frac{1}{\sqrt{2\pi}\zeta}
	\exp\left(-\frac{\norm{\bm{x}-\bm{x}_0}^2}{2\zeta^2} \right),
	\quad \bm{x} \in \Omega.
\end{equation}
In addition, we use the following set-up:
\begin{itemize}
	\item[(i)] {\bf FE Discretization.}
	We consider a FE discretization using $H^1$-conforming $\mathcal{P}^1$
	elements on a quasi-uniform mesh $\mathcal{T}_h$ of $1.5\times 10^4$ triangles,
	with a total number of degrees of freedom equal to $14641$ and mesh size $h = 1.16 \times 10^{-2}$.
	\item[(ii)] {\bf Construction of the Reduced Space.}
	The space $\mathcal{V}^{\text{(rb)}}_{R,M} \subset \mathcal{V}_h$ is computed
	as in Section~\ref{sec:laplace_transform_MOR}.
	For the computation of the snapshots, we use the 
    sampling points given in Lemma~\ref{eq:error_discrete_norm},
    with a particular instance of $\mu>0$ to be specified. 
	We remark that, in view of the insights of \cite[Section 5.4]{henriquez2024fast},
	we only effectively compute $M+1$ samples instead of $2M+1$.
	\item[(iii)] {\bf Parameter Setting.}
	In \eqref{eq:initial_cond_gauss}, we set $\bm{x}_0 = (0.25,-0.15)^\top$,
	$\zeta =0.05$ and for the parameters of the Ricker wavelet we consider $t_0 = 2.5$ and $\alpha \in  \left\{\pi, \frac{3}{2} \pi,2 \pi,\frac{5}{2} \pi\right\}$.
	\item[(iv)] {\bf Time-stepping Scheme.}
	For both the computation of the high-fidelity solution and the reduced basis solution, 
	i.e., the numerical approximation of \eqref{eq:wave}--\eqref{eq:ic2} and Problem~\ref{pr:sdpr},
	respectively, we consider the implicit Newmark-beta time-stepping scheme.
	We set the final time to $T=10$, and the total number of time steps
    to $N_t = 2 \times 10^4$.
\end{itemize}

\subsection{Singular Values of the Snapshot Matrix}
Figure~\ref{fig:sing_val} portrays the decay of the singular values of the snapshot matrix for
the initial condition in \eqref{eq:initial_cond_gauss} with $\zeta = 0.05$ and $t_0 = 2.5$,
and for $\alpha \in \left\{\pi, \frac{3}{2} \pi, 2 \pi, \frac{5}{2} \pi\right\}$.
\begin{figure}[!ht]
	\centering
	\begin{subfigure}[t]{0.48\textwidth}
		\centering
		\includegraphics[width=\textwidth]
		{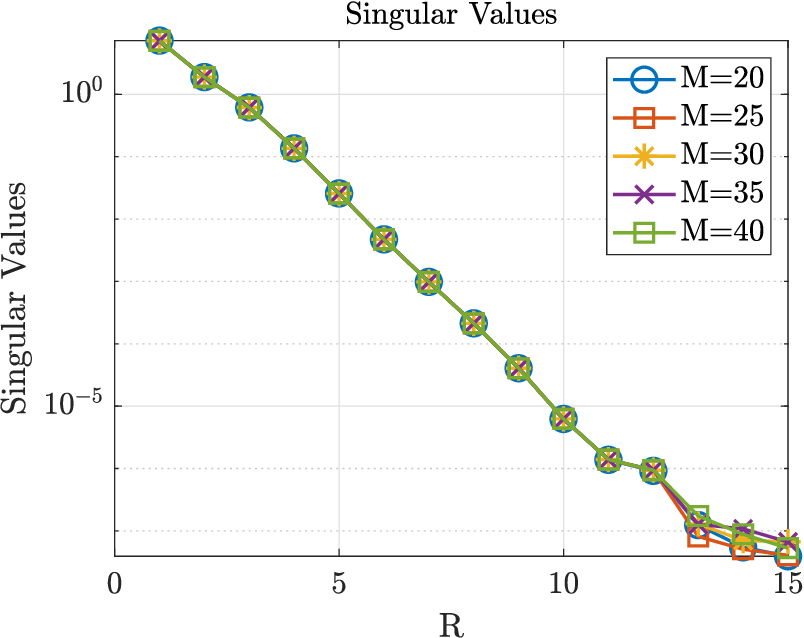}
		\subcaption{$\alpha = \pi$}
		\label{fig:sing_val_alpha_25}
	\end{subfigure}
	\hfill
	\begin{subfigure}[t]{0.48\textwidth}
		\centering
		\includegraphics[width=\textwidth]
		{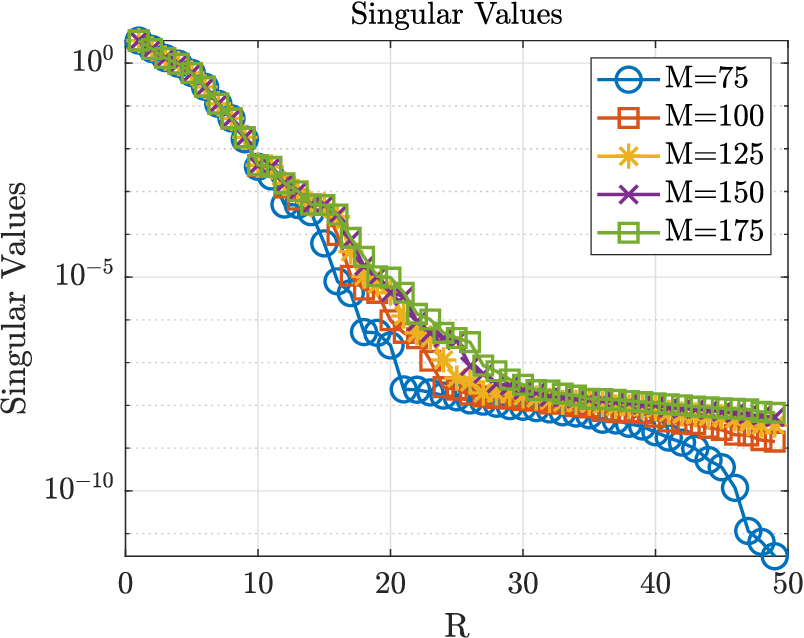}
		\subcaption{$\alpha = \frac{3}{2} \pi$}
		\label{fig:sing_val_alpha_5}
	\end{subfigure}

	\vspace{0.5em}

	\begin{subfigure}[t]{0.48\textwidth}
		\centering
		\includegraphics[width=\textwidth]
		{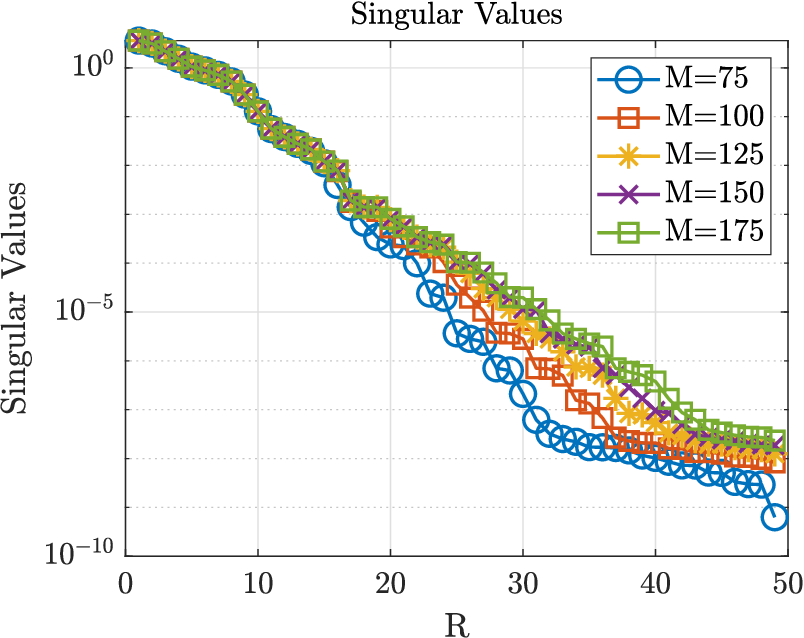}
		\subcaption{$\alpha = 2 \pi$}
		\label{fig:sing_val_alpha_75}
	\end{subfigure}
	\hfill
	\begin{subfigure}[t]{0.48\textwidth}
		\centering
		\includegraphics[width=\textwidth]
		{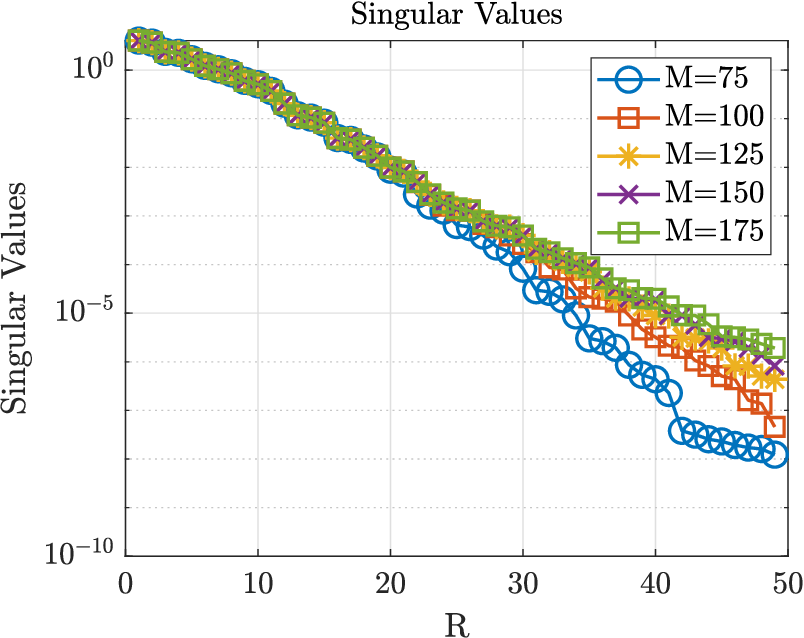}
		\subcaption{$\alpha = \frac{5}{2} \pi$}
		\label{fig:sing_val_alpha_10}
	\end{subfigure}

	\caption{
	Singular values of the snapshot matrix for the initial conditions
	stated in \eqref{eq:initial_cond_gauss} with $\zeta = 0.05$ and for $t_0 = 2.5$
    and for $\alpha \in  \left\{\pi, \frac{3}{2} \pi,2 \pi,\frac{5}{2} \pi\right\}$.
	}
	\label{fig:sing_val}
\end{figure}


\subsection{Convergence of the Relative Error}
Figure~\ref{fig:error_rel_U1} through Figure~\ref{fig:error_rel_U4} portray the
convergence of the relative error as defined in \eqref{eq:rel_error}
between the high-fidelity solution and the reduced one as the dimension of the reduced space
increases for $t_0 = 2.5$ and $\alpha \in  \left\{\pi, \frac{3}{2} \pi,2 \pi,\frac{5}{2} \pi\right\}$.
More precisely, Figure~\ref{fig:plot_relative_error_U1_L2} and Figure~\ref{fig:plot_relative_error_U1_H1}
present the aforementioned error measure with $\cX = L^2(\Omega)$ and $\cX=H^1_0(\Omega)$
in \eqref{eq:rel_error}, respectively, for $t_0 = 2.5$ and $\alpha = \pi$, and for $M \in \{20,25,30,35,40,45\}$.
For the choice of $\mu$, we select $\mu = \alpha$ for $\alpha = \pi$
and $\mu = \frac{\alpha}{8}$ for $\alpha \in  \left\{\frac{3}{2} \pi,2 \pi,\frac{5}{2} \pi\right\}$.
Again, we remark that under the considerations presented in \cite[{Section 5.4}]{henriquez2024fast}, effectively only half, i.e., $M+1$, snapshots are computed instead of $2M+1$. Figures~\ref{fig:plot_relative_error_U1_L2}--\ref{fig:plot_relative_error_U1_H1}, Figures~\ref{fig:plot_relative_error_U2_L2}--\ref{fig:plot_relative_error_U2_H1}, Figures~\ref{fig:plot_relative_error_U3_L2}--\ref{fig:plot_relative_error_U3_H1}
present the aforementioned error measure with $\cX = L^2(\Omega)$ and $\cX=H^1_0(\Omega)$, 
respectively, for $t_0 = 2.5$ and $\alpha \in  \left\{\ \frac{3}{2} \pi,2 \pi,\frac{5}{2} \pi\right\}$ and for $M \in  \{75,100,125,150,175\}$.
\begin{figure}[!ht]
	\centering
	\begin{subfigure}{0.48\linewidth}
		\includegraphics[width=\textwidth]
		{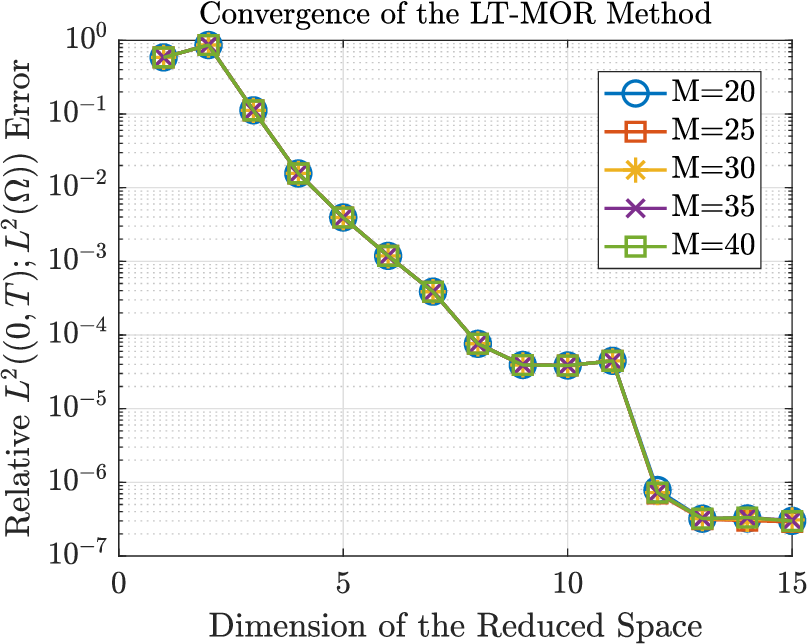}
		\subcaption{
		$\text{Rel\_Error}^{\normalfont\text{(rb)}}_{R,M}(\mathfrak{I};L^2(\Omega))$ for $\alpha =\pi$.}
		\label{fig:plot_relative_error_U1_L2}
	\end{subfigure}
	\begin{subfigure}{0.48\linewidth}
		\includegraphics[width=\textwidth]
		{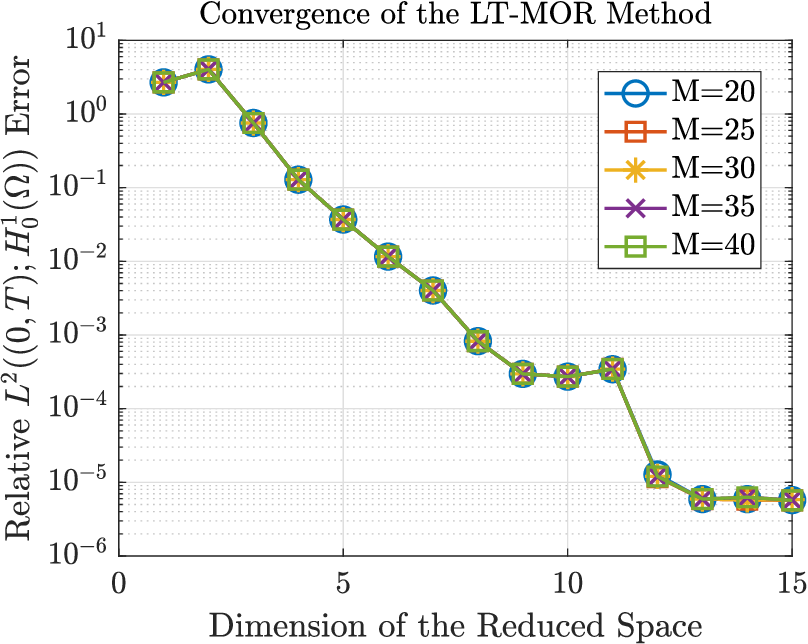}
		\subcaption{$\text{Rel\_Error}^{\normalfont\text{(rb)}}_{R,M}(\mathfrak{I};H^1_0(\Omega))$ for $\alpha =\pi$.}
		\label{fig:plot_relative_error_U1_H1}
	\end{subfigure}
	\caption{\label{fig:error_rel_U1} 
	Convergence of the relative error between the high-fidelity solution 
	and the reduced solution for $t_0 = 2.5$ and $\alpha = \pi$.
	In Figure~\ref{fig:plot_relative_error_U1_L2} the relative error is computed in the $L^2(\Omega)$-norm
	and in Figure~\ref{fig:plot_relative_error_U1_H1} in the $H^1_0(\Omega)$-norm.
	}
\end{figure}

\begin{figure}[!ht]
	\centering
	\begin{subfigure}{0.48\linewidth}
		\includegraphics[width=\textwidth]
		{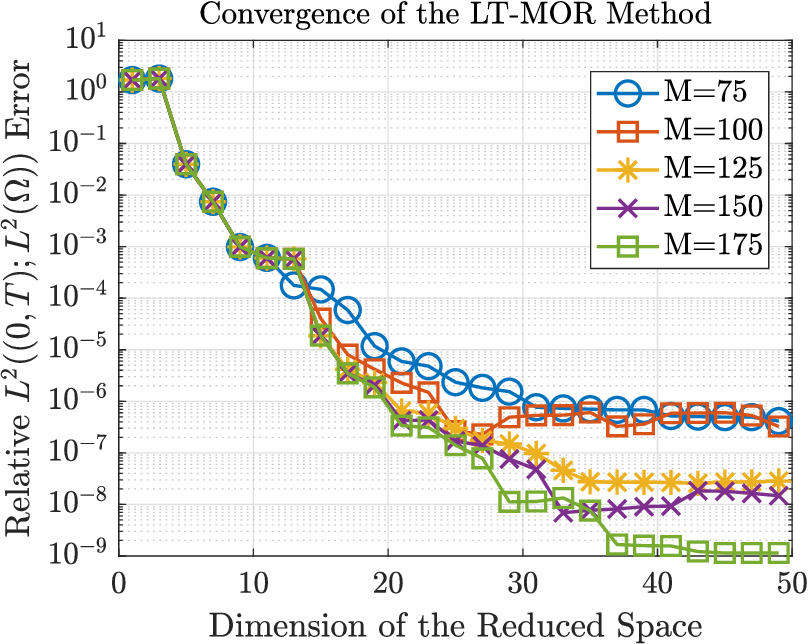}
		\subcaption{
		$\text{Rel\_Error}^{\normalfont\text{(rb)}}_{R,M}(\mathfrak{I};L^2(\Omega))$ for $\alpha = \frac{3}{2}\pi$.}
		\label{fig:plot_relative_error_U2_L2}
	\end{subfigure}
	\begin{subfigure}{0.48\linewidth}
		\includegraphics[width=\textwidth]
		{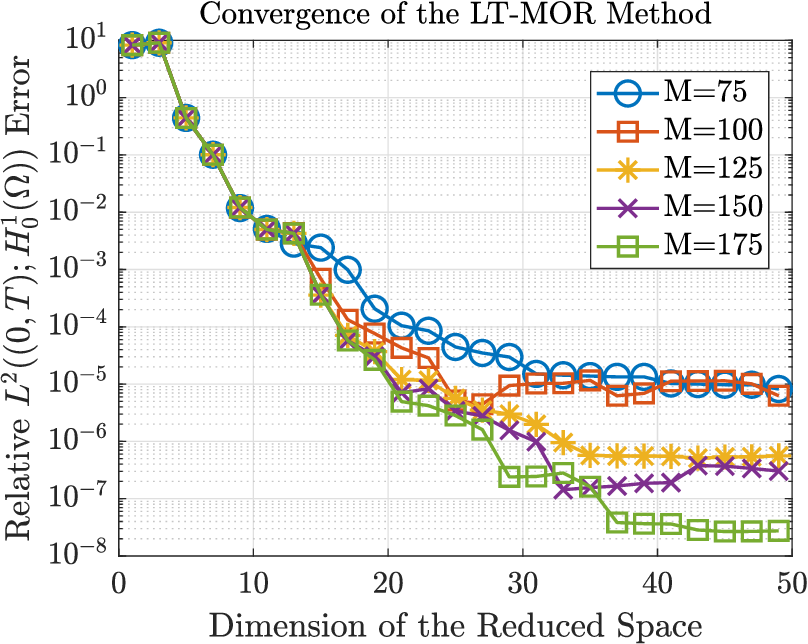}
		\subcaption{$\text{Rel\_Error}^{\normalfont\text{(rb)}}_{R,M}(\mathfrak{I};H^1_0(\Omega))$ for $\alpha =  \frac{3}{2} \pi$.}
		\label{fig:plot_relative_error_U2_H1}
	\end{subfigure}
	\caption{\label{fig:error_rel_U2} 
	Convergence of the relative error between the high-fidelity solution 
	and the reduced solution for $t_0 = 2.5$ and $\alpha = \frac{3}{2} \pi$.
	In Figure~\ref{fig:plot_relative_error_U2_L2} the relative error is computed in the $L^2(\Omega)$-norm
	and in Figure~\ref{fig:plot_relative_error_U2_H1} in the $H^1_0(\Omega)$-norm.
	}
\end{figure}

\begin{figure}[!ht]
	\centering
	\begin{subfigure}{0.48\linewidth}
		\includegraphics[width=\textwidth]
		{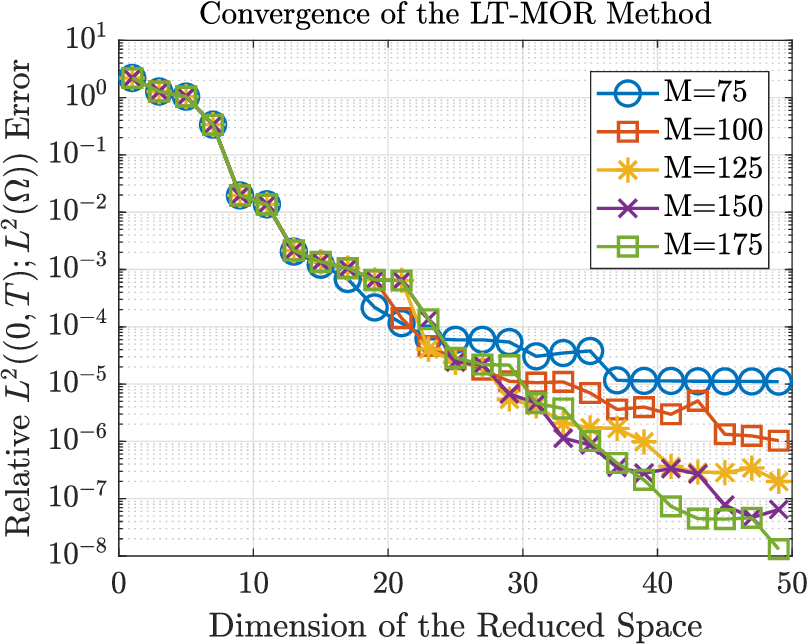}
		\subcaption{
		$\text{Rel\_Error}^{\normalfont\text{(rb)}}_{R,M}(\mathfrak{I};L^2(\Omega))$ for $\alpha = 2\pi$.}
		\label{fig:plot_relative_error_U3_L2}
	\end{subfigure}
	\begin{subfigure}{0.48\linewidth}
		\includegraphics[width=\textwidth]
		{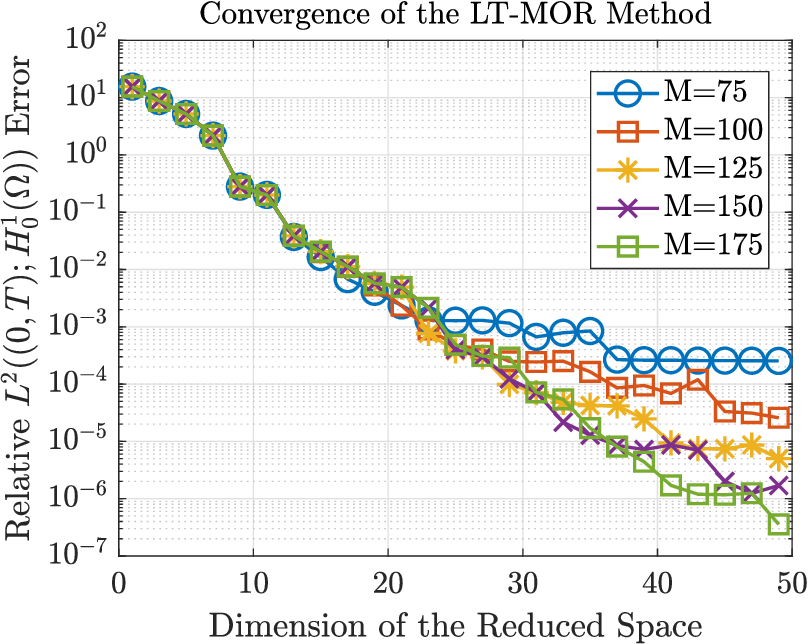}
		\subcaption{$\text{Rel\_Error}^{\normalfont\text{(rb)}}_{R,M}(\mathfrak{I};H^1_0(\Omega))$ for $\alpha =  2 \pi$.}
		\label{fig:plot_relative_error_U3_H1}
	\end{subfigure}
	\caption{\label{fig:error_rel_U3} 
	Convergence of the relative error between the high-fidelity solution 
	and the reduced solution for $t_0 = 2.5$ and $\alpha = 2 \pi$.
	In Figure~\ref{fig:plot_relative_error_U3_L2} the relative error is computed in the $L^2(\Omega)$-norm
	and in Figure~\ref{fig:plot_relative_error_U3_H1} in the $H^1_0(\Omega)$-norm.
	}
\end{figure}

\begin{figure}[!ht]
	\centering
	\begin{subfigure}{0.48\linewidth}
		\includegraphics[width=\textwidth]
		{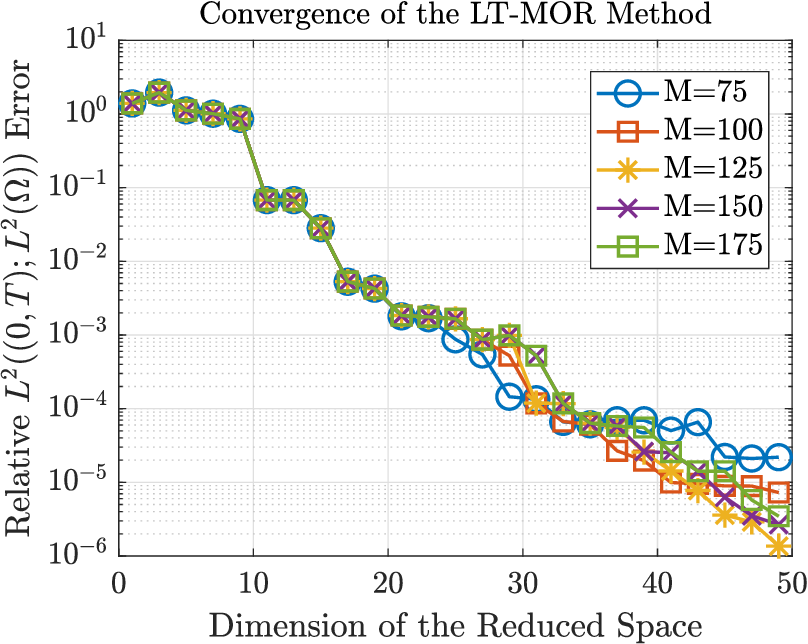}
		\subcaption{
		$\text{Rel\_Error}^{\normalfont\text{(rb)}}_{R,M}(\mathfrak{I};L^2(\Omega))$ for $\alpha = \frac{5}{2} \pi$.}
		\label{fig:plot_relative_error_U4_L2}
	\end{subfigure}
	\begin{subfigure}{0.48\linewidth}
		\includegraphics[width=\textwidth]
		{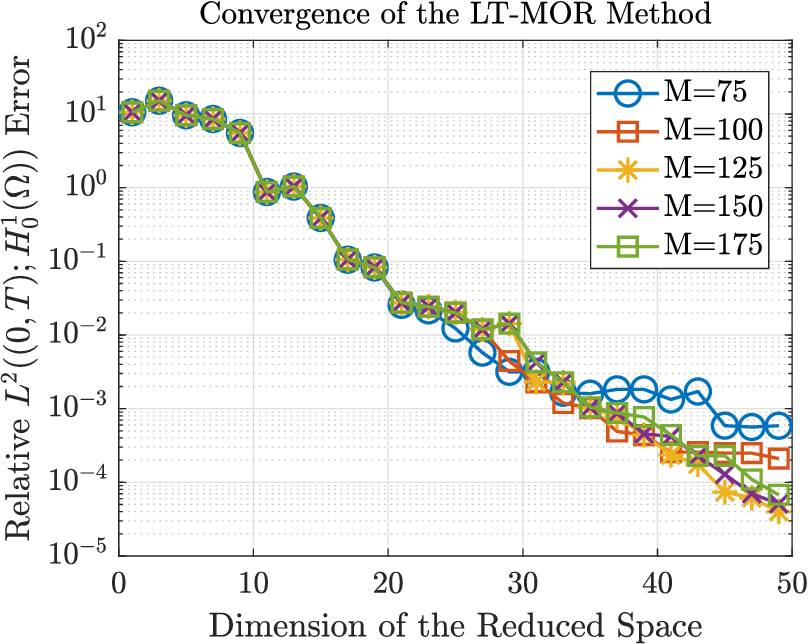}
		\subcaption{$\text{Rel\_Error}^{\normalfont\text{(rb)}}_{R,M}(\mathfrak{I};H^1_0(\Omega))$ for $\alpha =  \frac{5}{2} \pi$.}
		\label{fig:plot_relative_error_U4_H1}
	\end{subfigure}
	\caption{\label{fig:error_rel_U4} 
	Convergence of the relative error between the high-fidelity solution 
	and the reduced solution for $t_0 = 2.5$ and $\alpha = \frac{5}{2} \pi$.
	In Figure~\ref{fig:plot_relative_error_U4_L2} the relative error is computed in the $L^2(\Omega)$-norm and in Figure~\ref{fig:plot_relative_error_U4_H1} in the $H^1_0(\Omega)$-norm.
	}
\end{figure}

\subsection{Speed-up}

Figure~\ref{fig:speed_square}, more precisely Figures~\ref{fig:plot_speed_RB_600}--\ref{fig:plot_speed_RB_1600},
show the execution times of the LT-MOR method for
$M \in \{6,8,10,12,14,16\} \times 10^2$.
In each plot, the total time is broken down into the following contributions:
(1) assembling the FE discretization ({\sf Assemble FEM}), (2) computing the snapshots
(or high-fidelity solutions) in the Laplace domain ({\sf LD-HF}), (3) building the reduced basis ({\sf Build RB}),
(4) reconstructing the reduced solution in the high-fidelity space ({\sf Reconstruct HF}), and
(5) computing the reduced solution in the time domain ({\sf Solve TD-RB}).

\begin{figure}[!ht]
	\centering
	\begin{subfigure}{0.49\linewidth}
		\includegraphics[width=\textwidth]
		{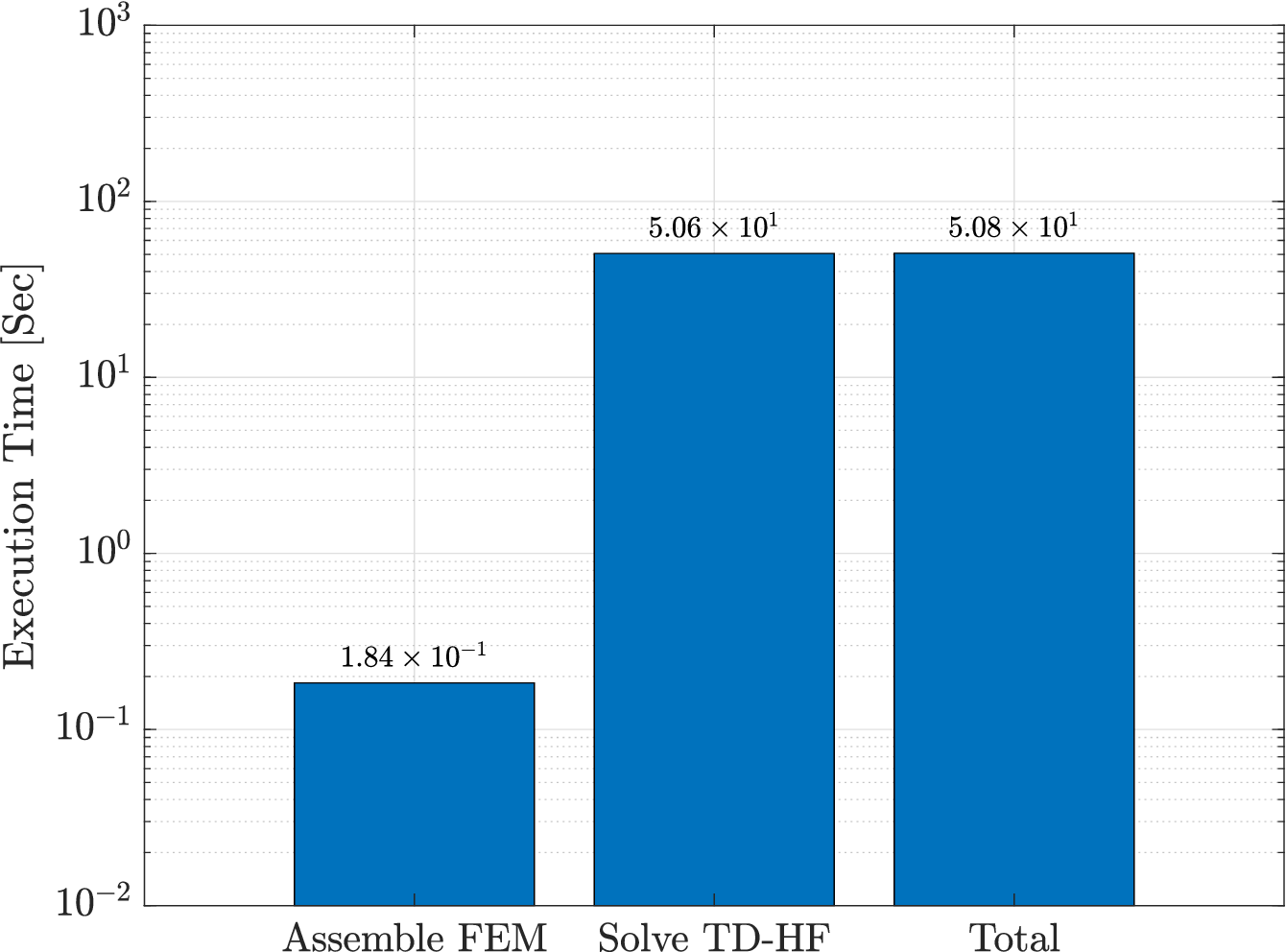}
        \subcaption{High-fidelity Solution.}
		\label{fig:plot_speed_RB_600}
	\end{subfigure}
	\begin{subfigure}{0.49\linewidth}
		\includegraphics[width=\textwidth]
		{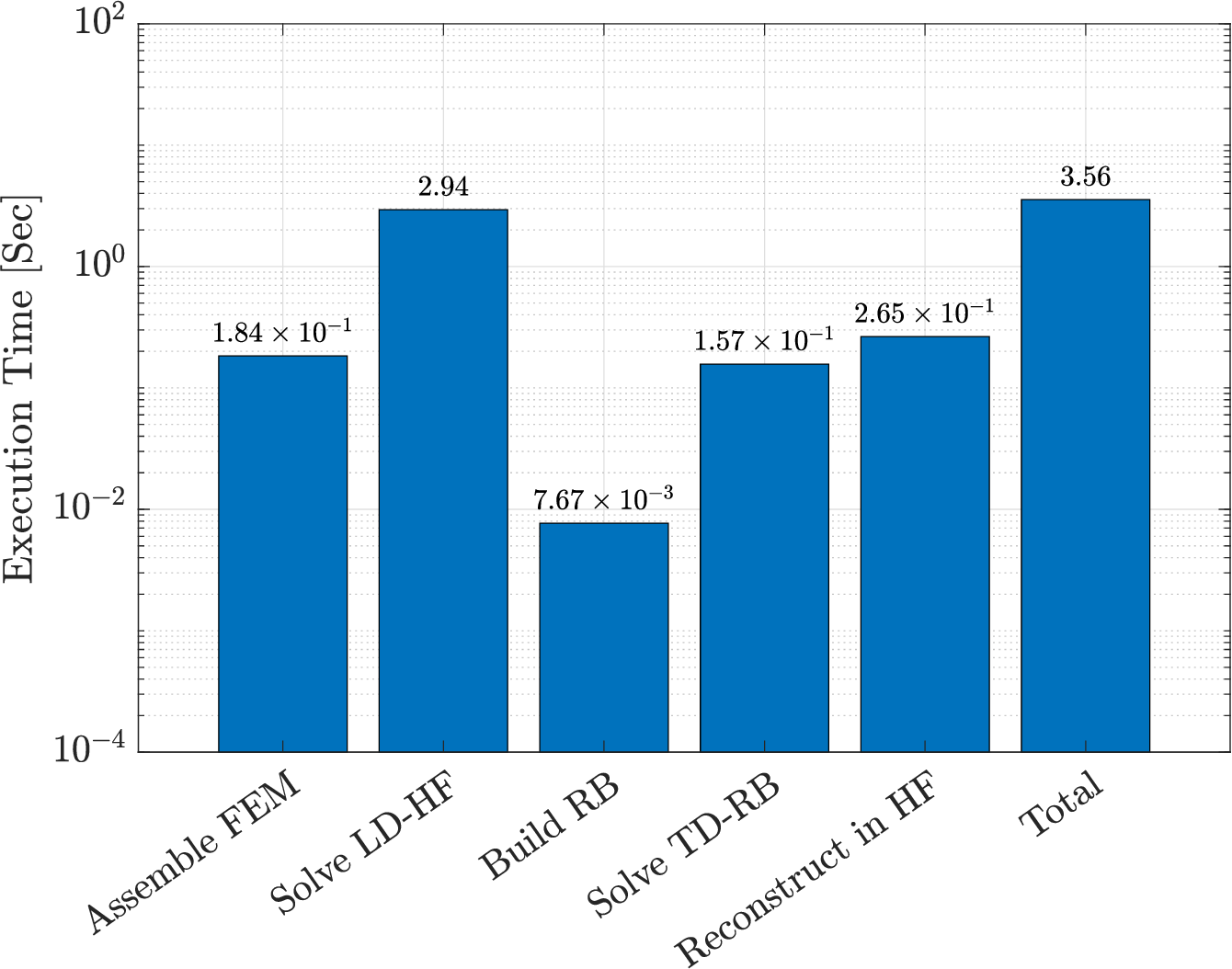}
		\subcaption{LT-MOR with $M=75$ and $R=50$.}
		\label{fig:plot_speed_RB_800}
	\end{subfigure}
	\centering
	\begin{subfigure}{0.49\linewidth}
		\includegraphics[width=\textwidth]
		{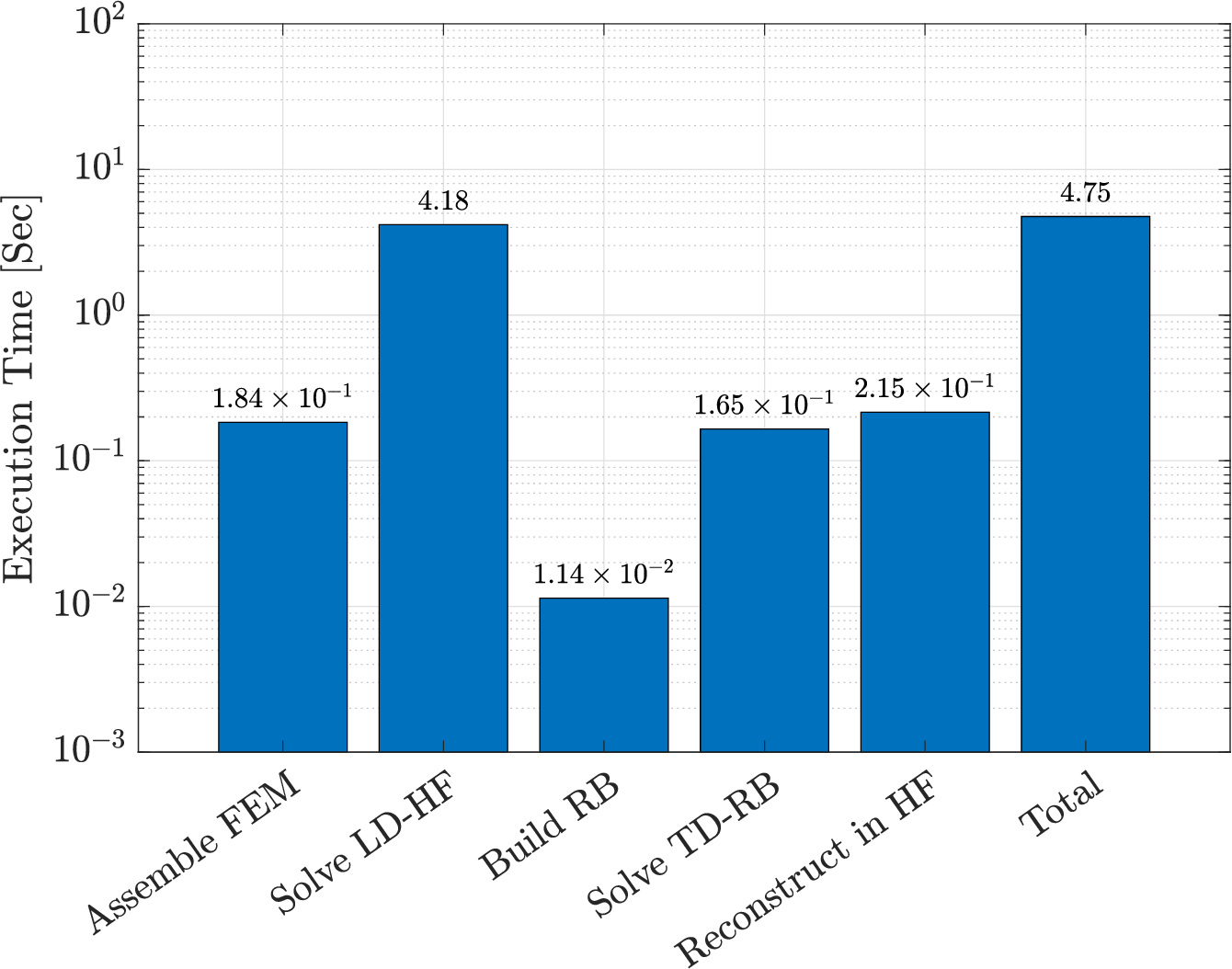}
		\subcaption{LT-MOR with $M=100$ and $R=50$.}
		\label{fig:plot_speed_RB_1000}
	\end{subfigure}
	\begin{subfigure}{0.49\linewidth}
		\includegraphics[width=\textwidth]
		{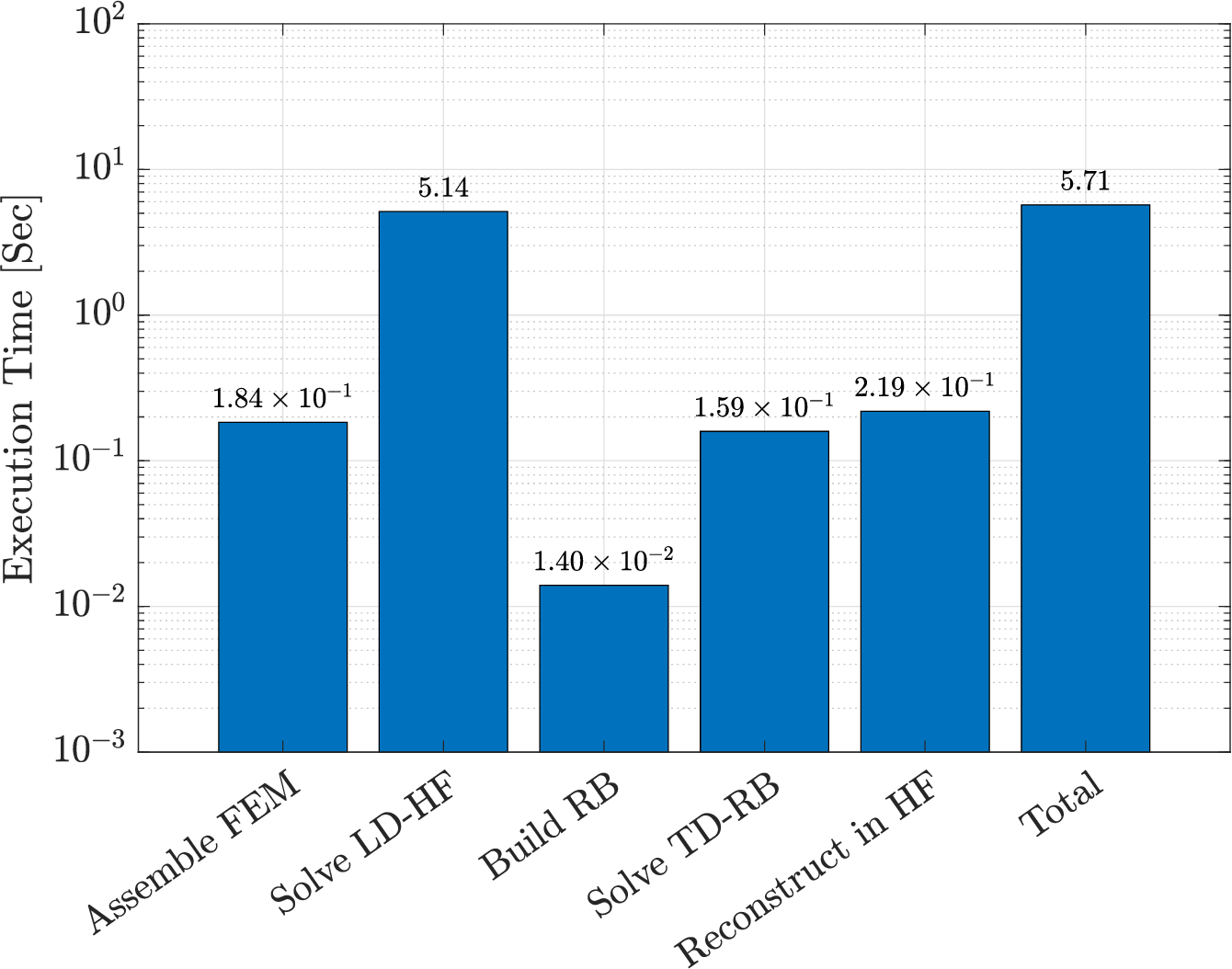}
		\subcaption{LT-MOR with $M=125$ and $R=50$.}
		\label{fig:plot_speed_RB_1200}
	\end{subfigure}
	\centering
	\begin{subfigure}{0.49\linewidth}
		\includegraphics[width=\textwidth]
		{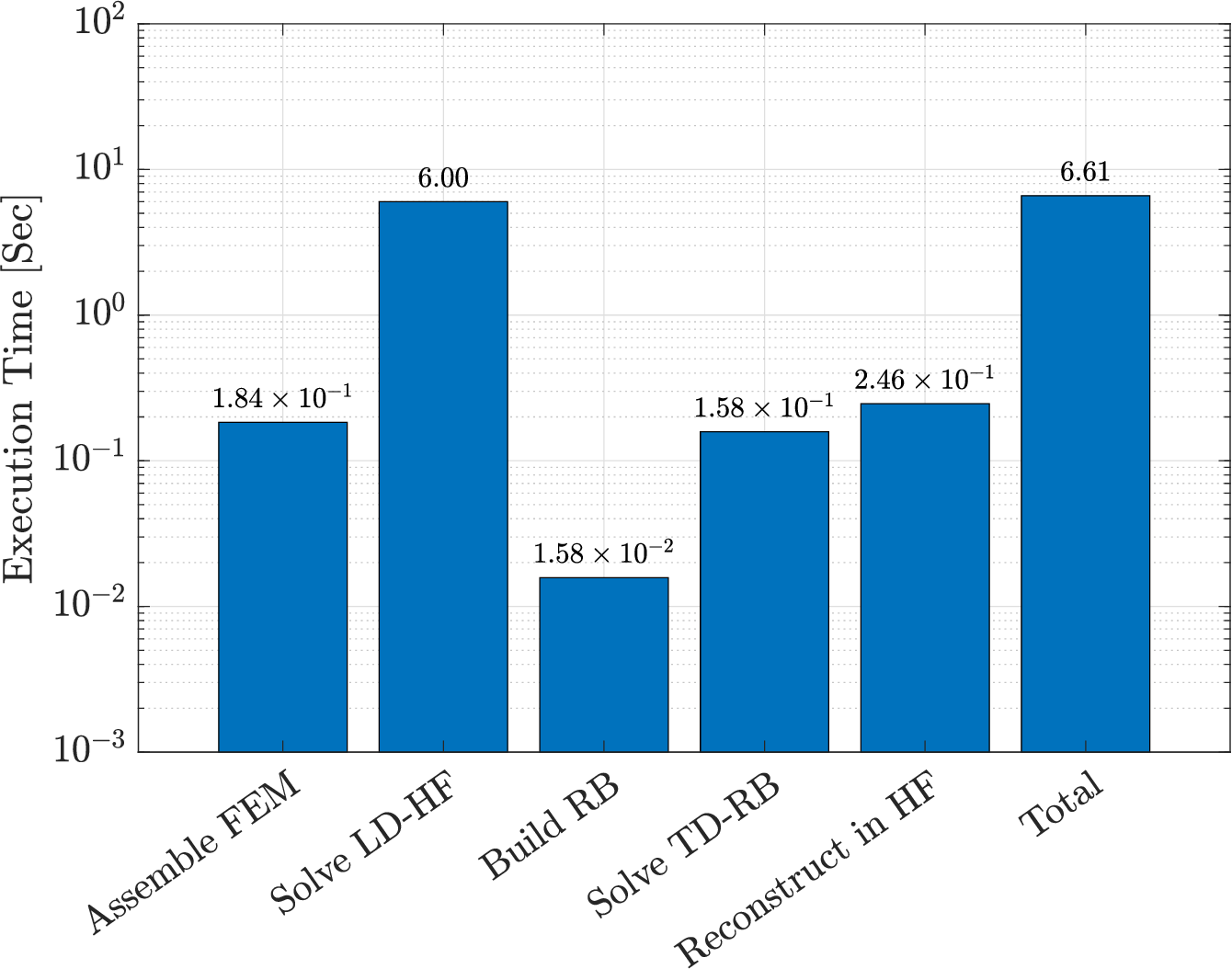}
		\subcaption{LT-MOR with $M=150$ and $R=50$.}
		\label{fig:plot_speed_RB_1400}
	\end{subfigure}
	\begin{subfigure}{0.49\linewidth}
		\includegraphics[width=\textwidth]
		{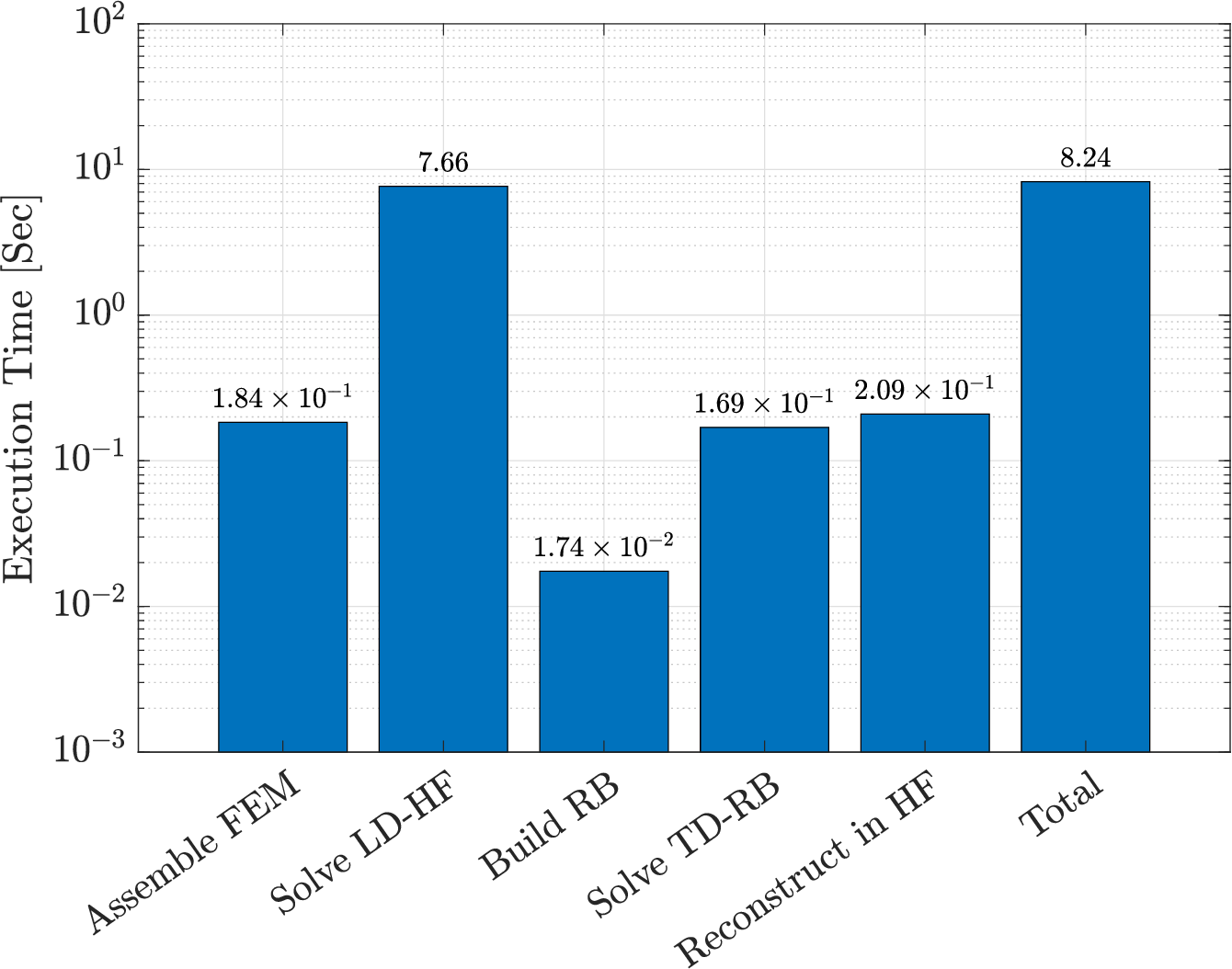}
		\subcaption{LT-MOR with $M=175$ and $R=50$.}
		\label{fig:plot_speed_RB_1600}
	\end{subfigure}
	\caption{\label{fig:speed_square} 
	Execution times of the LT-MOR method
	for the computation of the reduced basis solution 
	for $M \in \{75,100,125,150,175\}$, $R=50$, $t_0 = 2.5$,
    and $\alpha = \frac{5}{2} \pi$
	split into the following five main contributions:
	(1) Assembling the FE discretization ({\sf Assemble FEM}), (2) computing the snapshots 
	or high-fidelity solutions in the Laplace domain ({\sf LD-HF}),
	(3) building the reduced basis ({\sf Build RB}),
	and (4) computing the reduced solution in the time domain ({\sf Solve TD-RB}).
	}
\end{figure}

\subsection{Visualization of the Reduced Basis}
Let $\mathcal{V}^{\text{(rb)}}_R$ be as in \eqref{eq:fPOD} for some $R\in\IN$.
Then, $\left\{\varphi^{\textrm{(rb)}}_1,\dots,\varphi^{\textrm{(rb)}}_R\right\}$
constitutes an orthonormal basis of $\IV_R^{(\textrm{rb})}$, with $\varphi^{\textrm{(rb)}}_j$ as
in \eqref{eq:def_basis_rb}. Indeed, provided that the solution $\bm{\Phi}^{\textrm{(rb)}}_R$ to
\eqref{eq:fPOD} has been computed, one can plot the corresponding representation in $\mathcal{V}_h$
by using the expression in \eqref{eq:def_basis_rb}.
In Figures~\ref{fig:reduced_basis_1}--\ref{fig:reduced_basis_2} we plot the basis
$\left\{\varphi^{\textrm{(rb)}}_1,\dots,\varphi^{\textrm{(rb)}}_R\right\}$ in the space $\mathcal{V}_h$ for $R=1,\dots,8$.
\begin{figure}[!ht]
	\centering
	\begin{subfigure}{0.4\linewidth}
		\includegraphics[width=\textwidth]
		{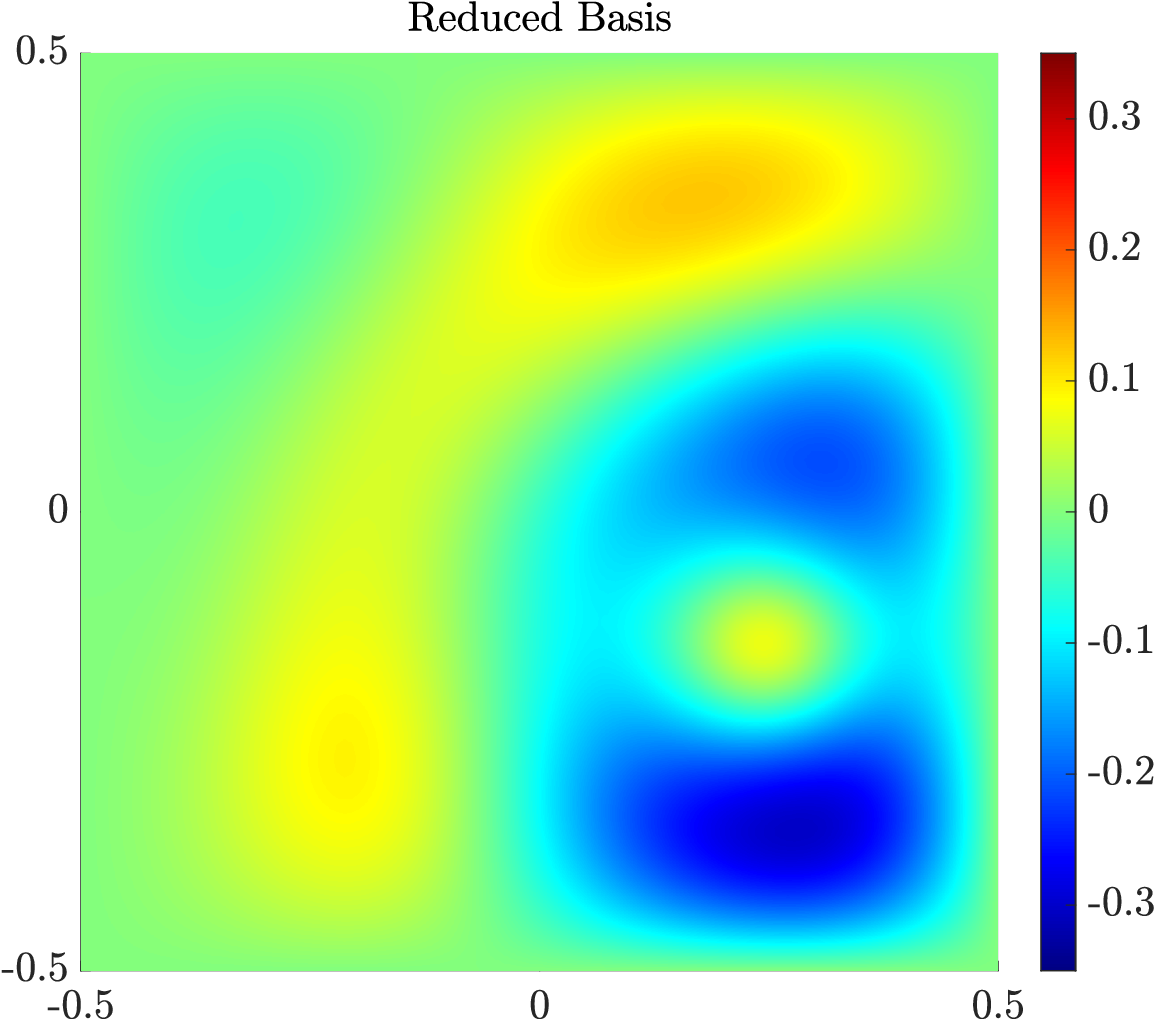}
		\subcaption{$\varphi^{\textrm{(rb)}}_1$}
		\label{fig:plot_mode_1}
	\end{subfigure}
	\begin{subfigure}{0.4\linewidth}
		\includegraphics[width=\textwidth]
		{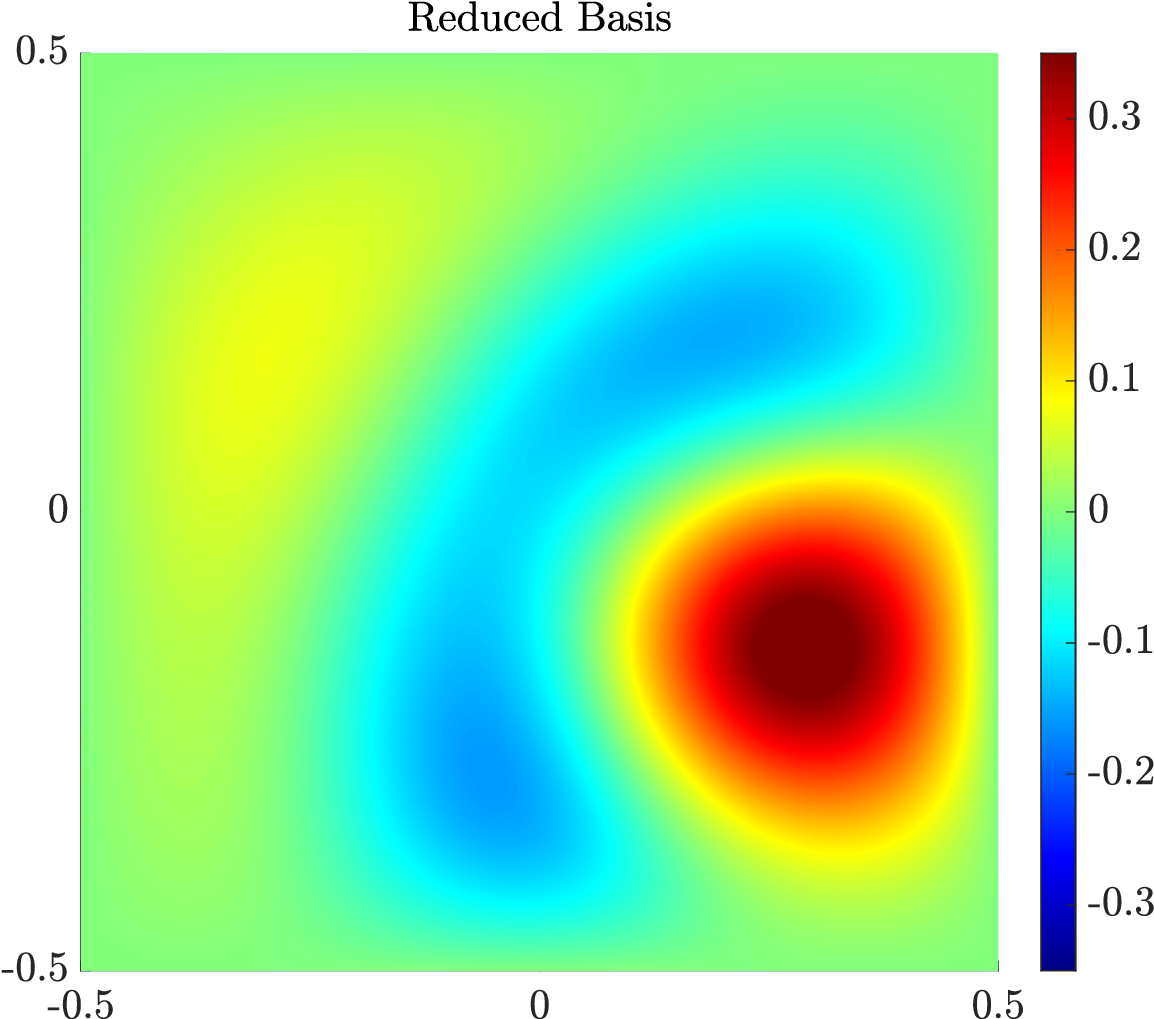}
		\subcaption{$\varphi^{\textrm{(rb)}}_2$}
		\label{fig:plot_mode_2}
	\end{subfigure}
	\centering
	\begin{subfigure}{0.4\linewidth}
		\includegraphics[width=\textwidth]
		{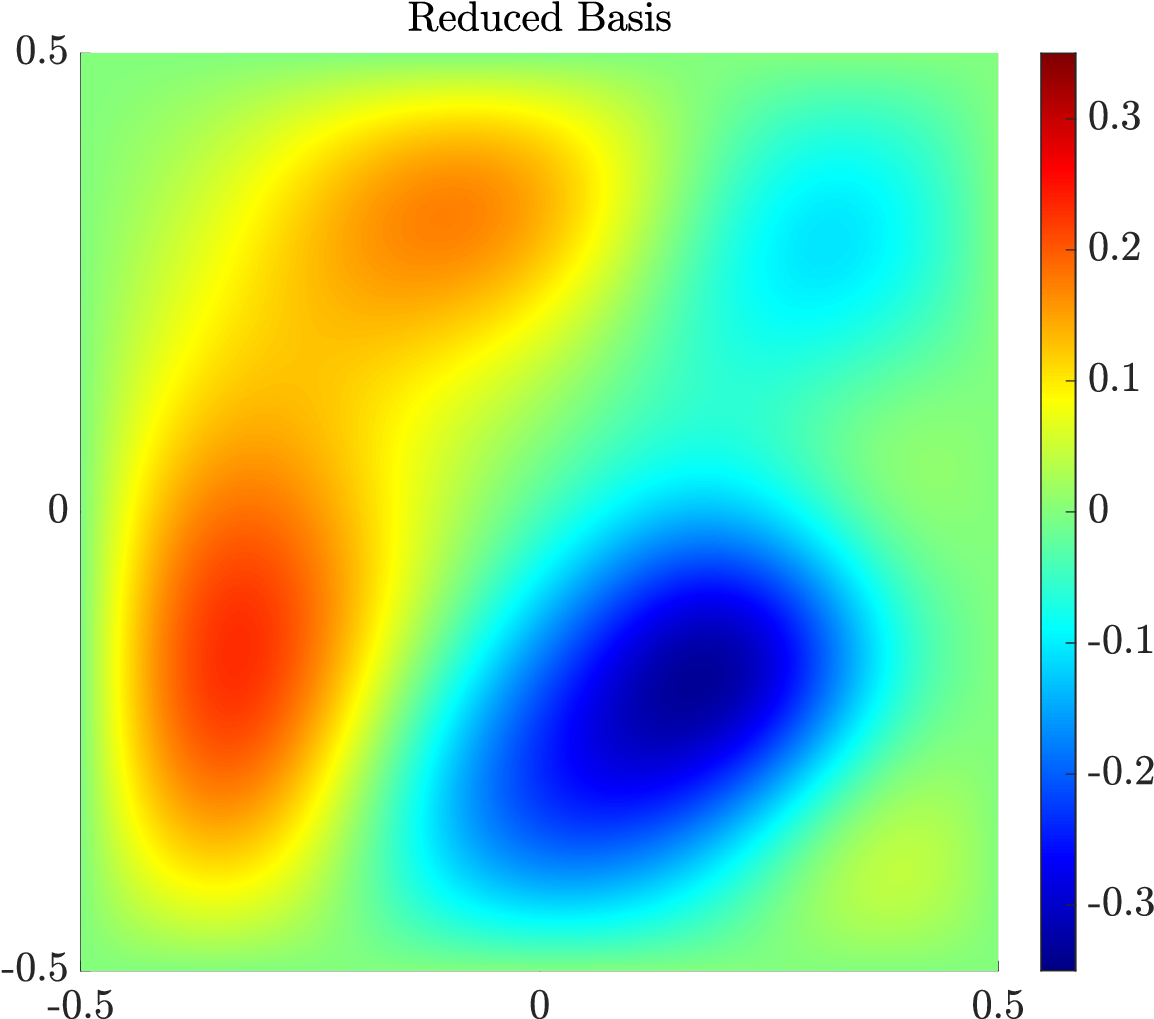}
		\subcaption{$\varphi^{\textrm{(rb)}}_3$}
		\label{fig:plot_mode_3}
	\end{subfigure}
	\begin{subfigure}{0.4\linewidth}
		\includegraphics[width=\textwidth]
		{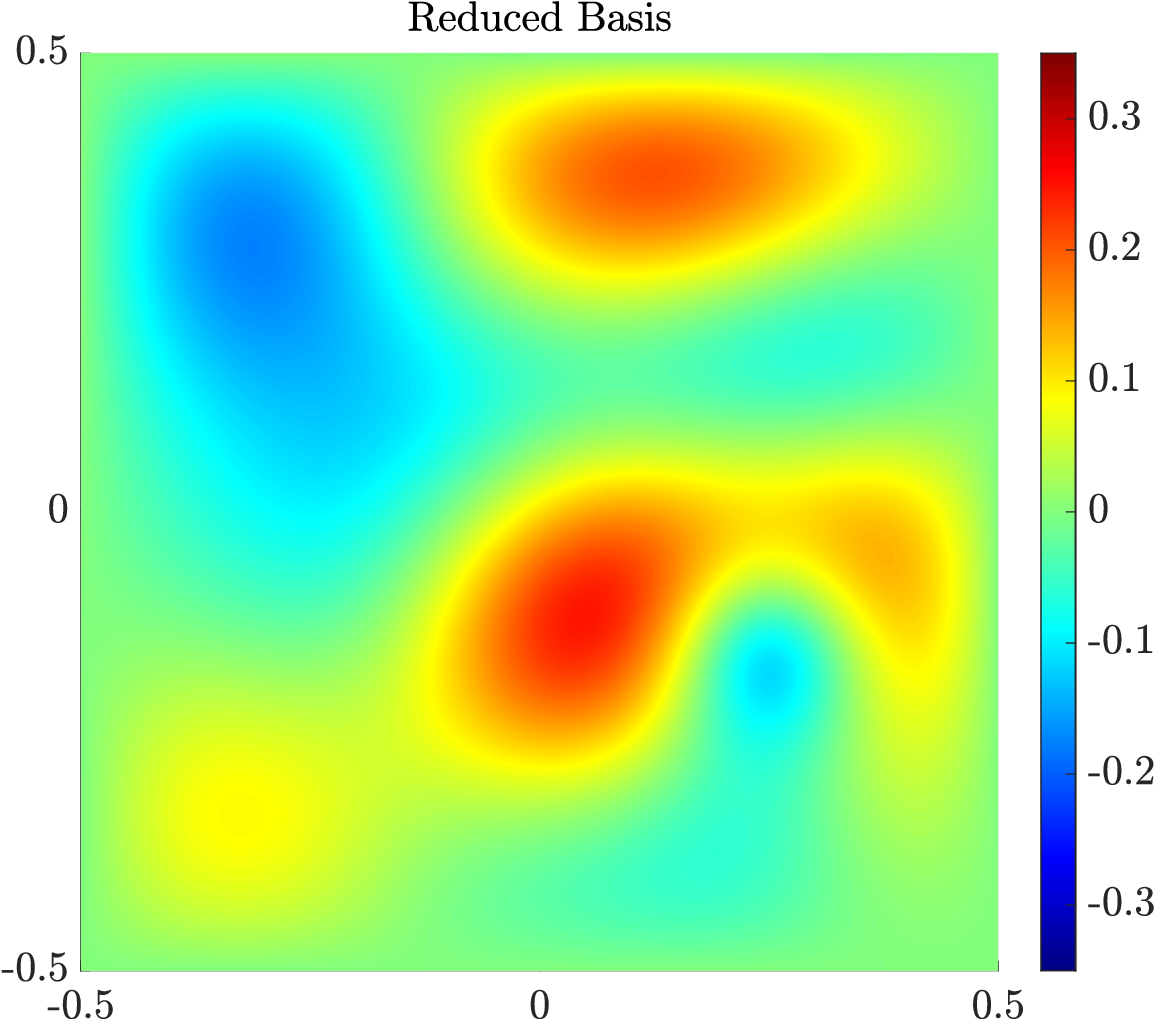}
		\subcaption{$\varphi^{\textrm{(rb)}}_4$}
		\label{fig:plot_mode_4}
	\end{subfigure}
	\caption{\label{fig:reduced_basis_1}
	Visualization of the first to fourth elements of the reduced space $\mathcal{V}^{\text{(rb)}}_{R,M}$
	for $t_0 = 2.5$, $\alpha = \frac{5}{2} \pi$, $M = 175$, and $R = 50$.
	}
\end{figure}

\begin{figure}[!ht]
	\centering
	\begin{subfigure}{0.4\linewidth}
		\includegraphics[width=\textwidth]
		{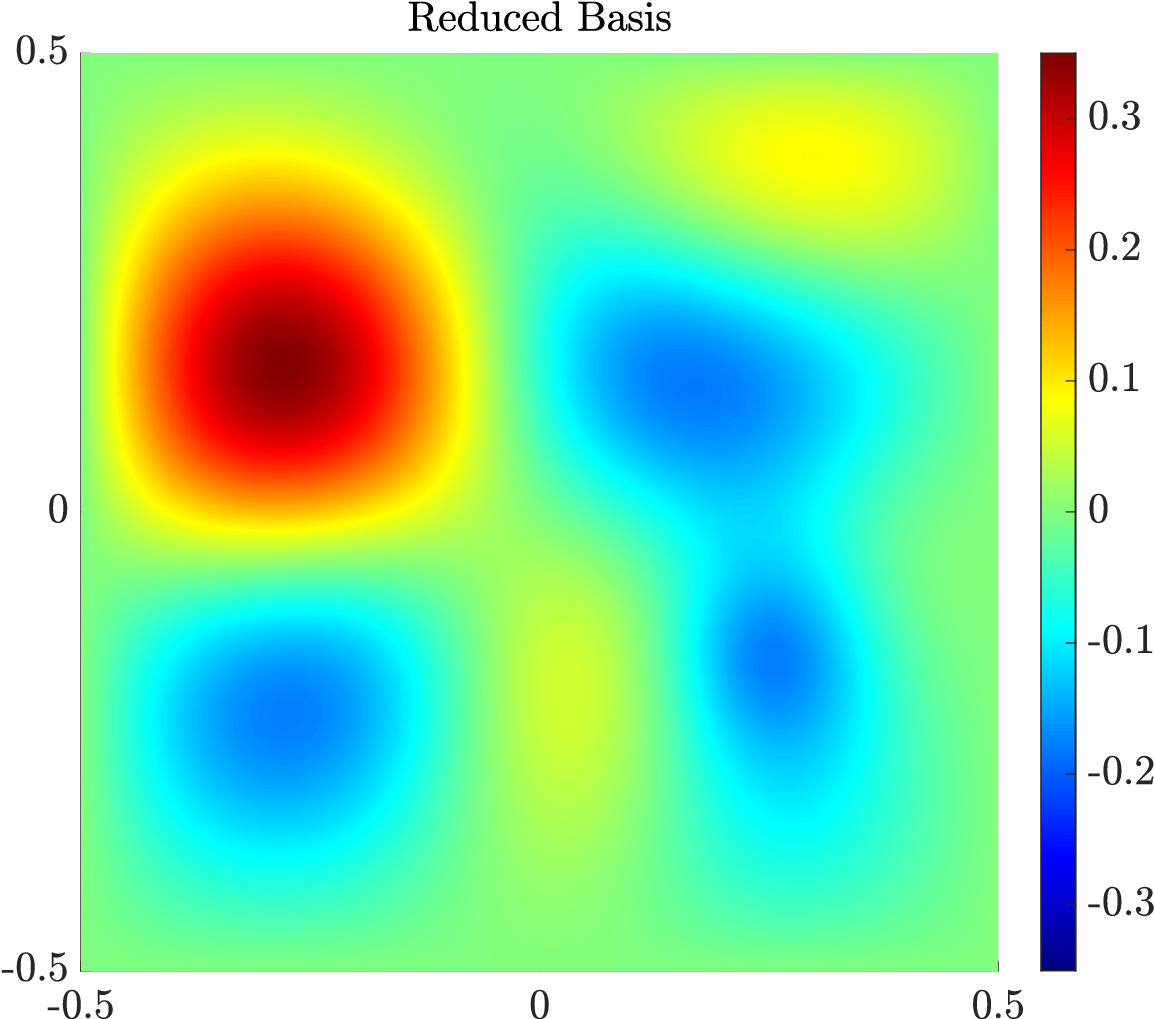}
		\subcaption{$\varphi^{\textrm{(rb)}}_5$}
		\label{fig:plot_mode_5}
	\end{subfigure}
	\begin{subfigure}{0.4\linewidth}
		\includegraphics[width=\textwidth]
		{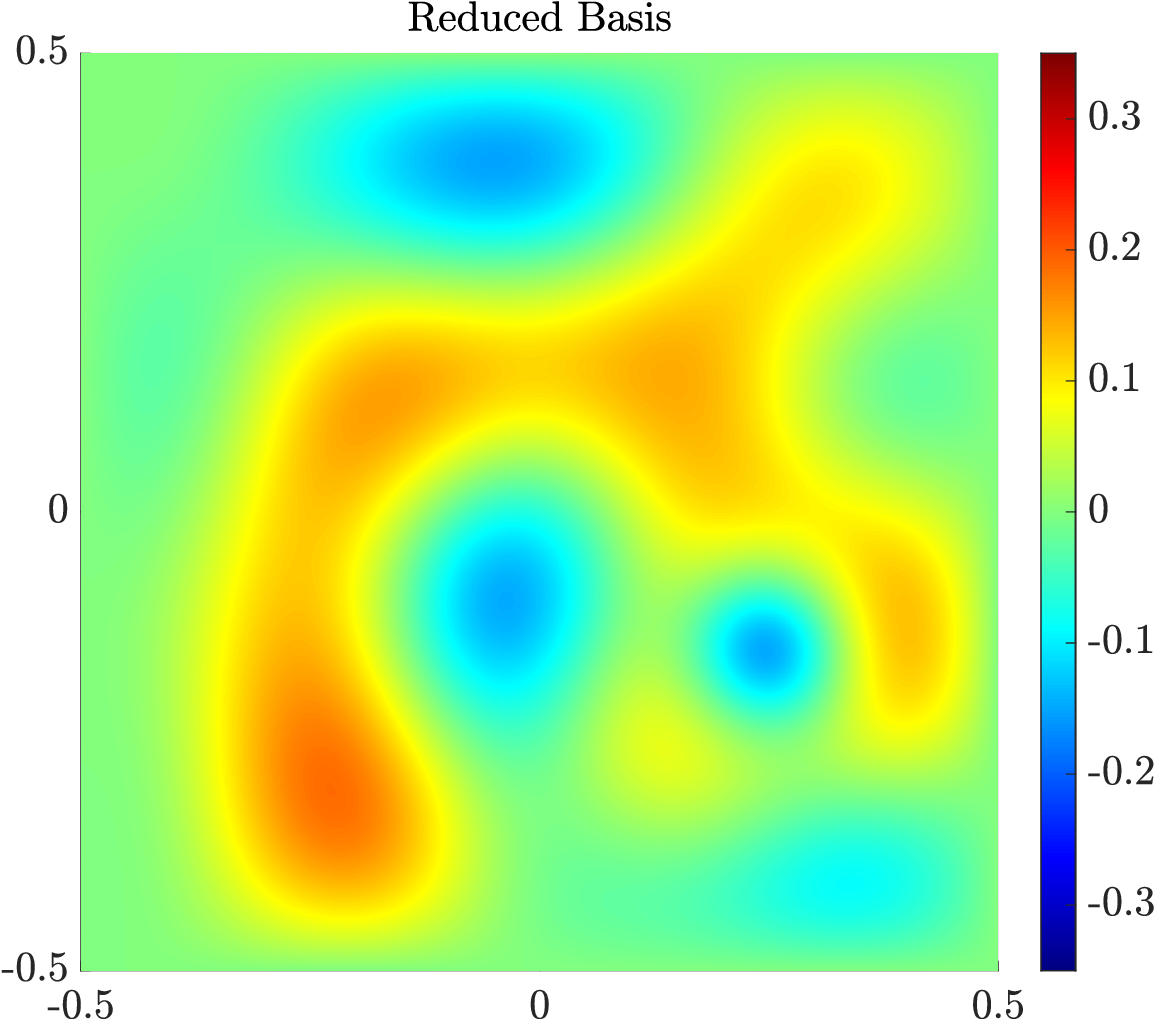}
		\subcaption{$\varphi^{\textrm{(rb)}}_6$}
		\label{fig:plot_mode_6}
	\end{subfigure}
	\centering
	\begin{subfigure}{0.4\linewidth}
		\includegraphics[width=\textwidth]
		{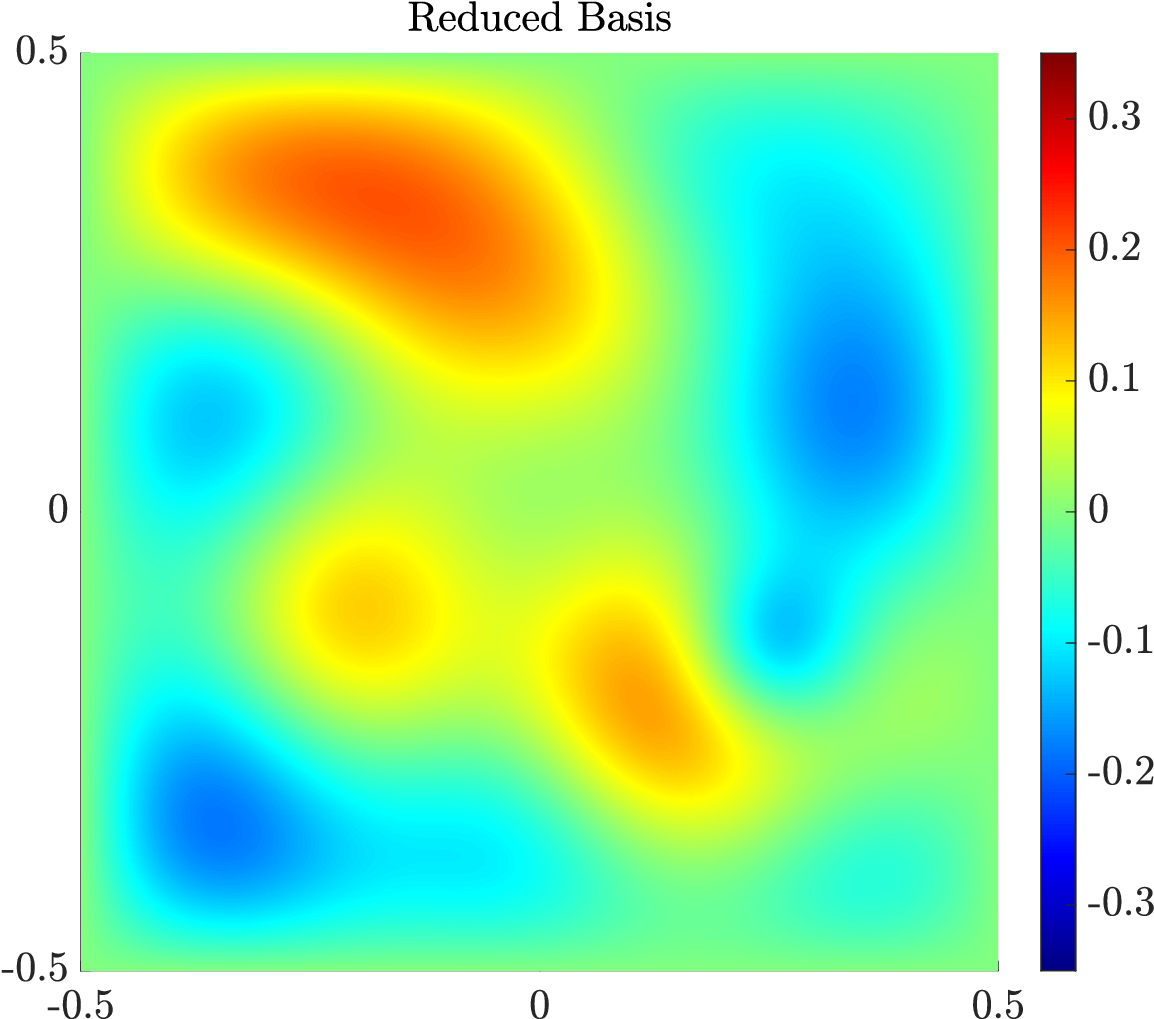}
		\subcaption{$\varphi^{\textrm{(rb)}}_7$}
		\label{fig:plot_mode_7}
	\end{subfigure}
	\begin{subfigure}{0.4\linewidth}
		\includegraphics[width=\textwidth]
		{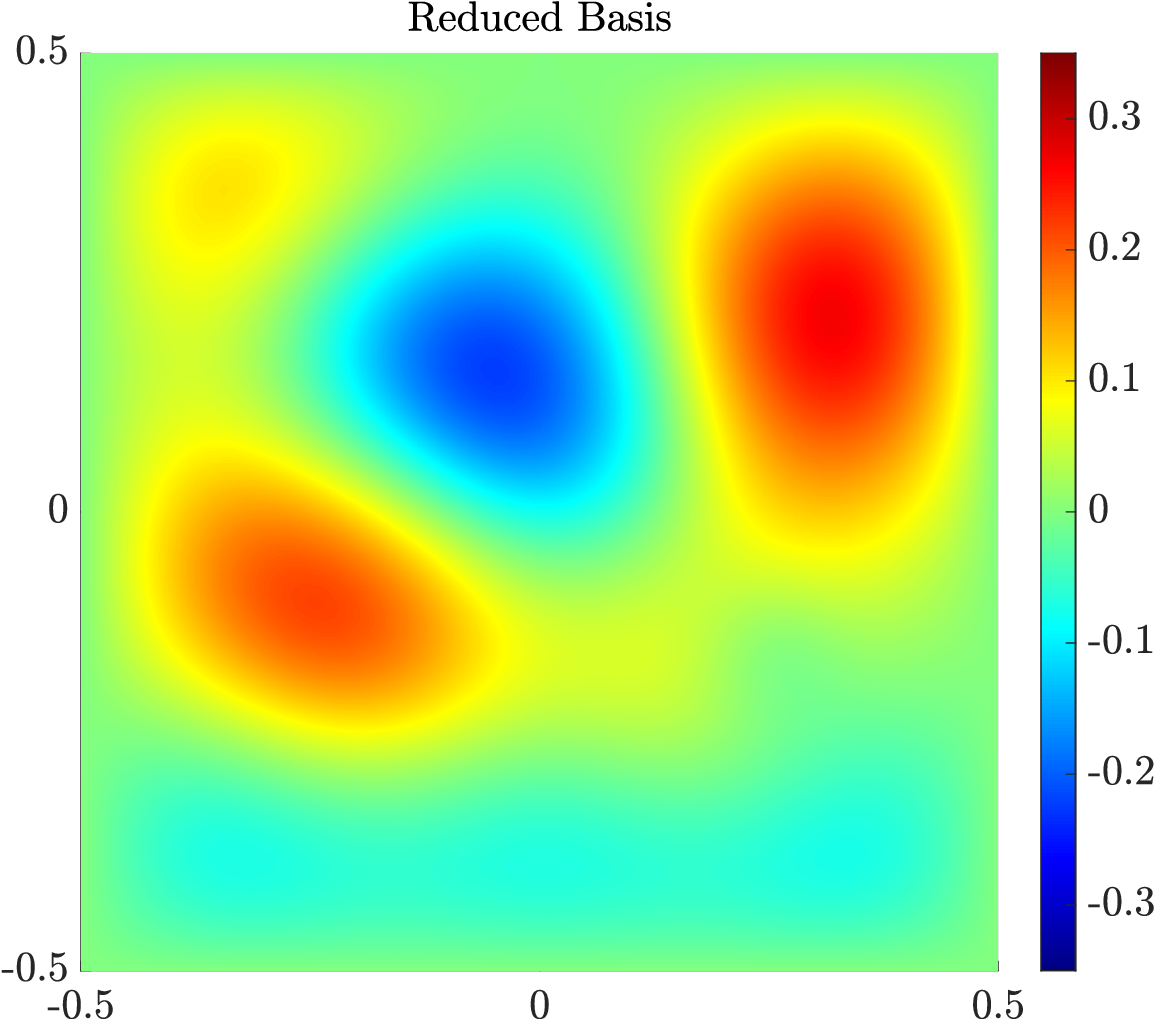}
		\subcaption{$\varphi^{\textrm{(rb)}}_8$}
		\label{fig:plot_mode_8}
	\end{subfigure}
	\caption{\label{fig:reduced_basis_2}
	Visualization of the fifth to eighth elements of the reduced space $\mathcal{V}^{\text{(rb)}}_{R,M}$
	for $t_0 = 2.5$, $\alpha = \frac{5}{2} \pi$, $M = 175$, and $R = 50$.
	}
\end{figure}

\subsection{Visualization of the Solution}
Figure~\ref{fig:plot_Sol_1} shows the LT-MOR solution and the corresponding
error with respect to the high-fidelity solution at different times, computed with
$M = 175$ and $R=50$. Specifically, the figure reports the solution/error pairs
at $t\in\{2.5,5,7.5,10\}$.
\begin{figure}[!ht]
	\centering
	\begin{subfigure}{0.4\linewidth}
		\includegraphics[width=\textwidth]
		{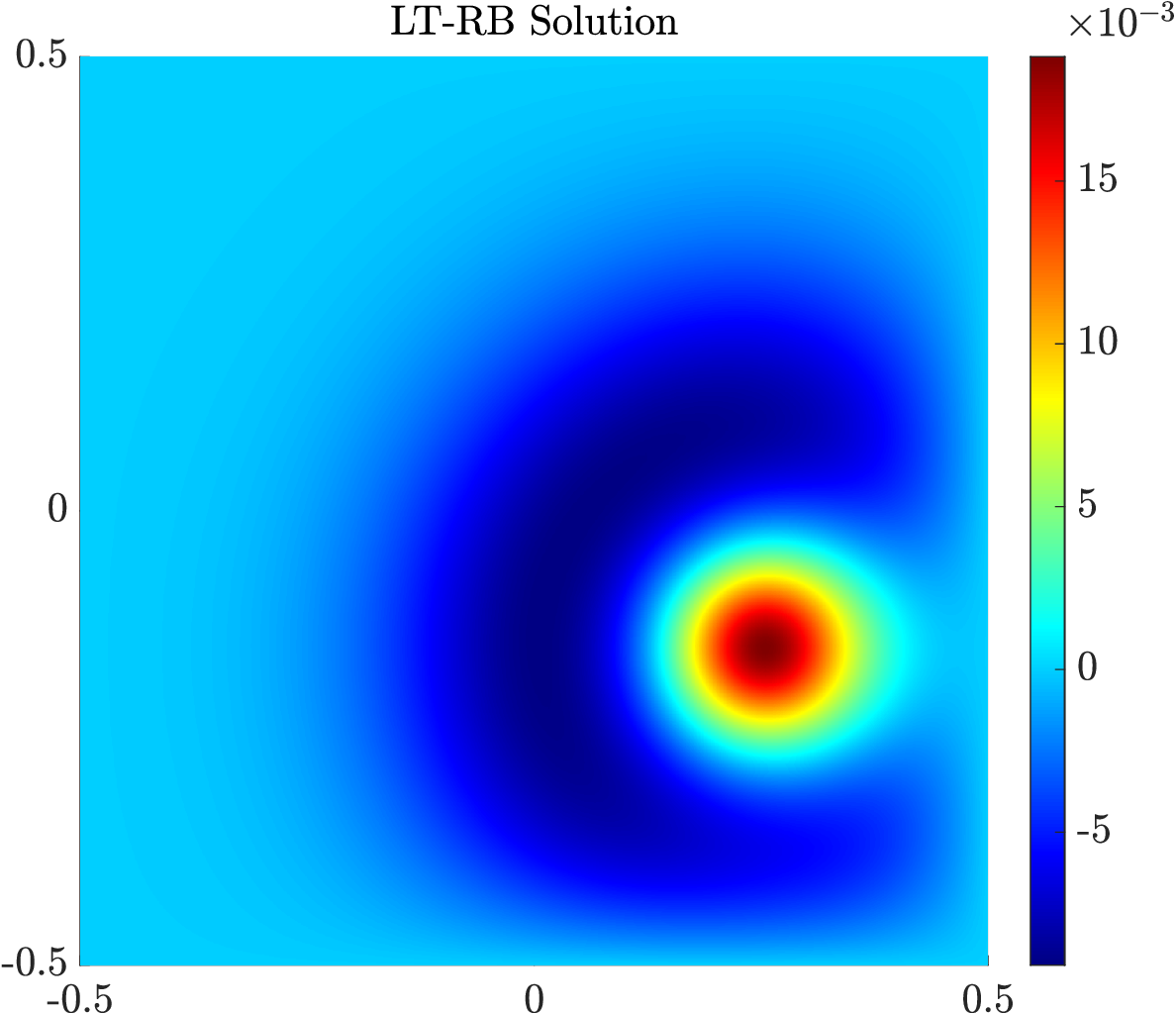}
		\subcaption{LT-MOR Solution at $t=2.5$.}
		\label{fig:plot_sol_1}
	\end{subfigure}
	\begin{subfigure}{0.4\linewidth}
		\includegraphics[width=\textwidth]
		{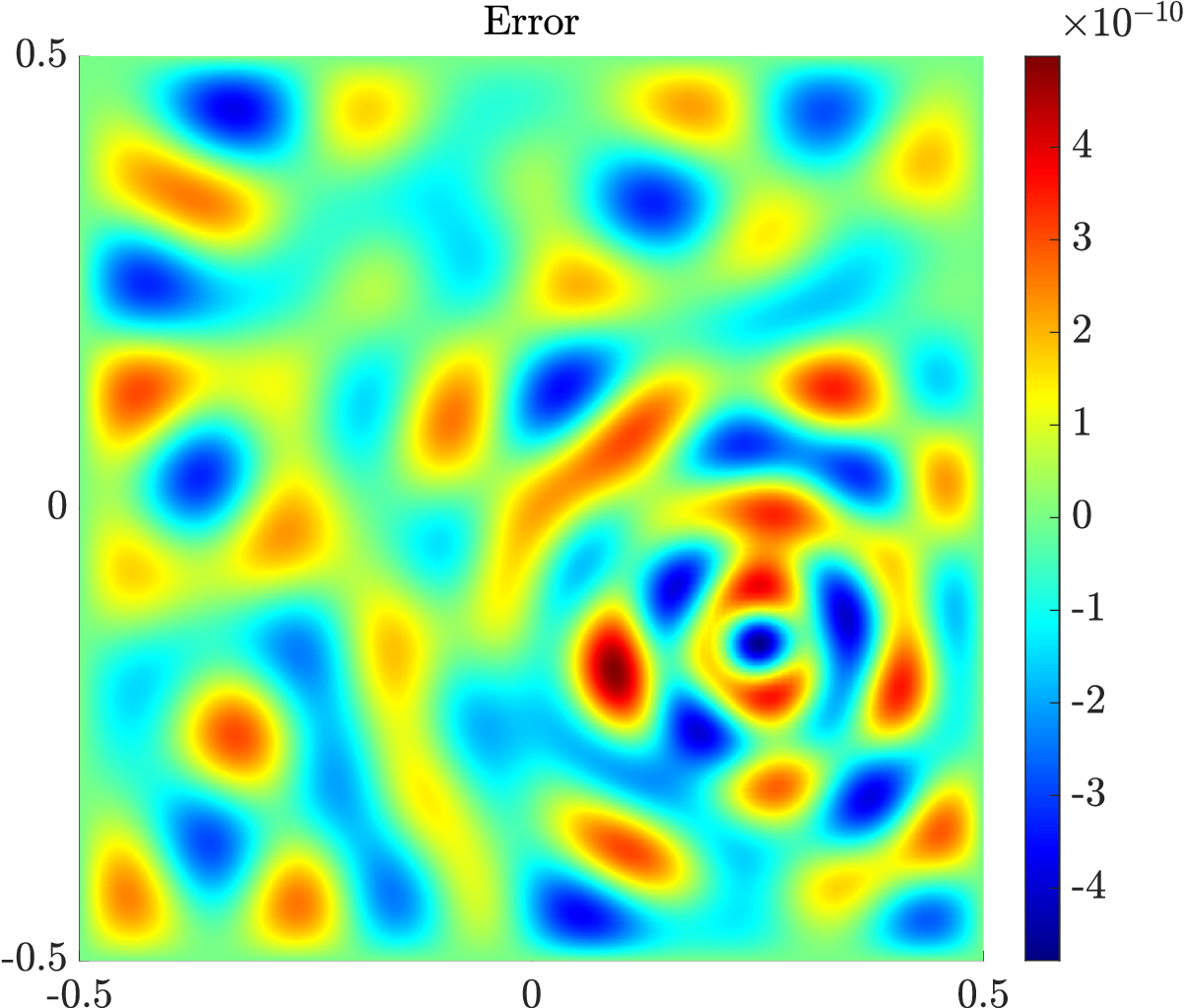}
		\subcaption{Error at $t=2.5$.}
		\label{fig:plot_sol_2}
	\end{subfigure}
	\centering
	\begin{subfigure}{0.4\linewidth}
		\includegraphics[width=\textwidth]
		{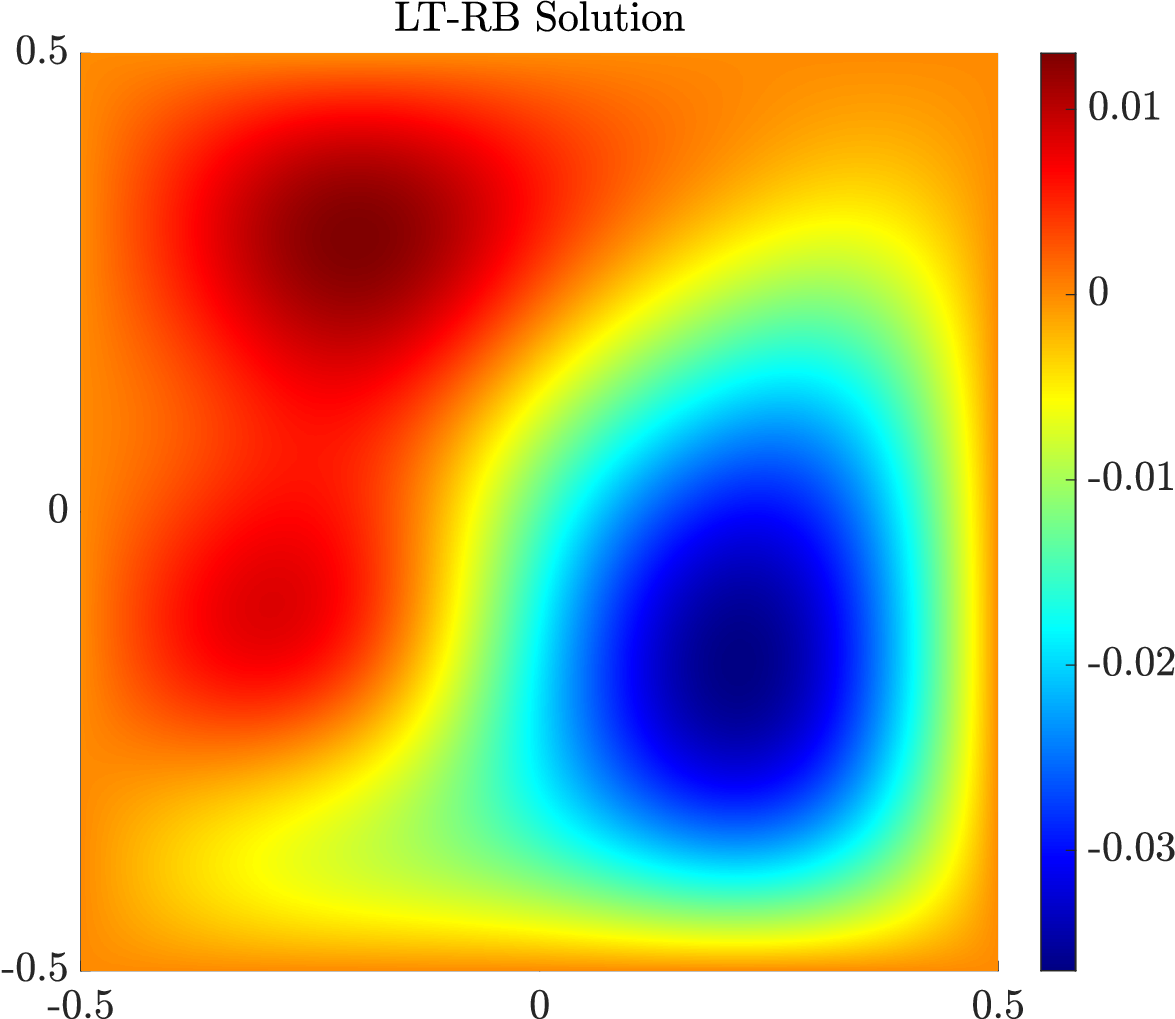}
		\subcaption{LT-MOR Solution at $t=5$.}
        \label{fig:plot_sol_3}
	\end{subfigure}
	\begin{subfigure}{0.4\linewidth}
		\includegraphics[width=\textwidth]
		{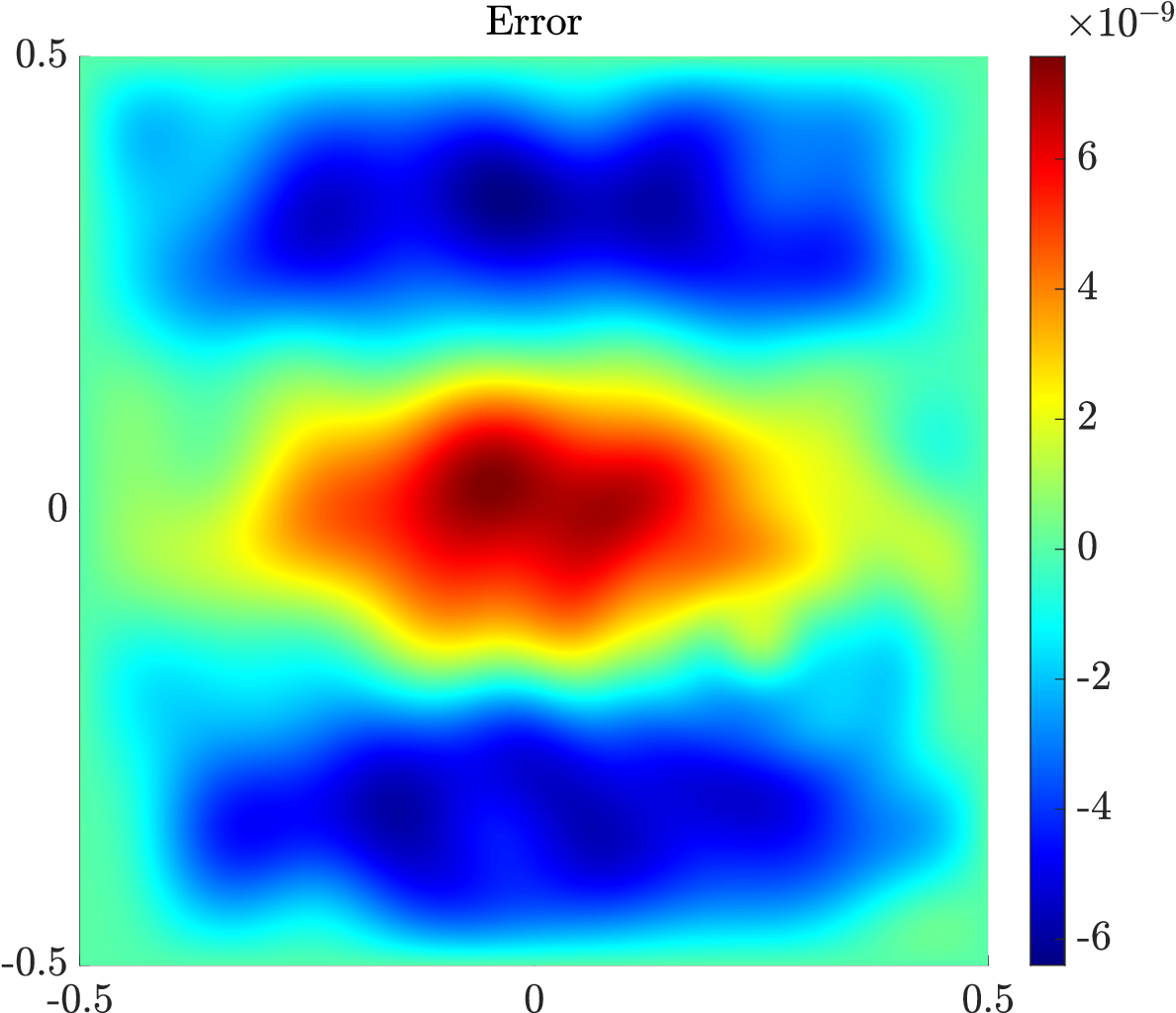}
		\subcaption{Error at $t=5$.}
		\label{fig:plot_sol_4}
	\end{subfigure}
	\centering
	\begin{subfigure}{0.4\linewidth}
		\includegraphics[width=\textwidth]
		{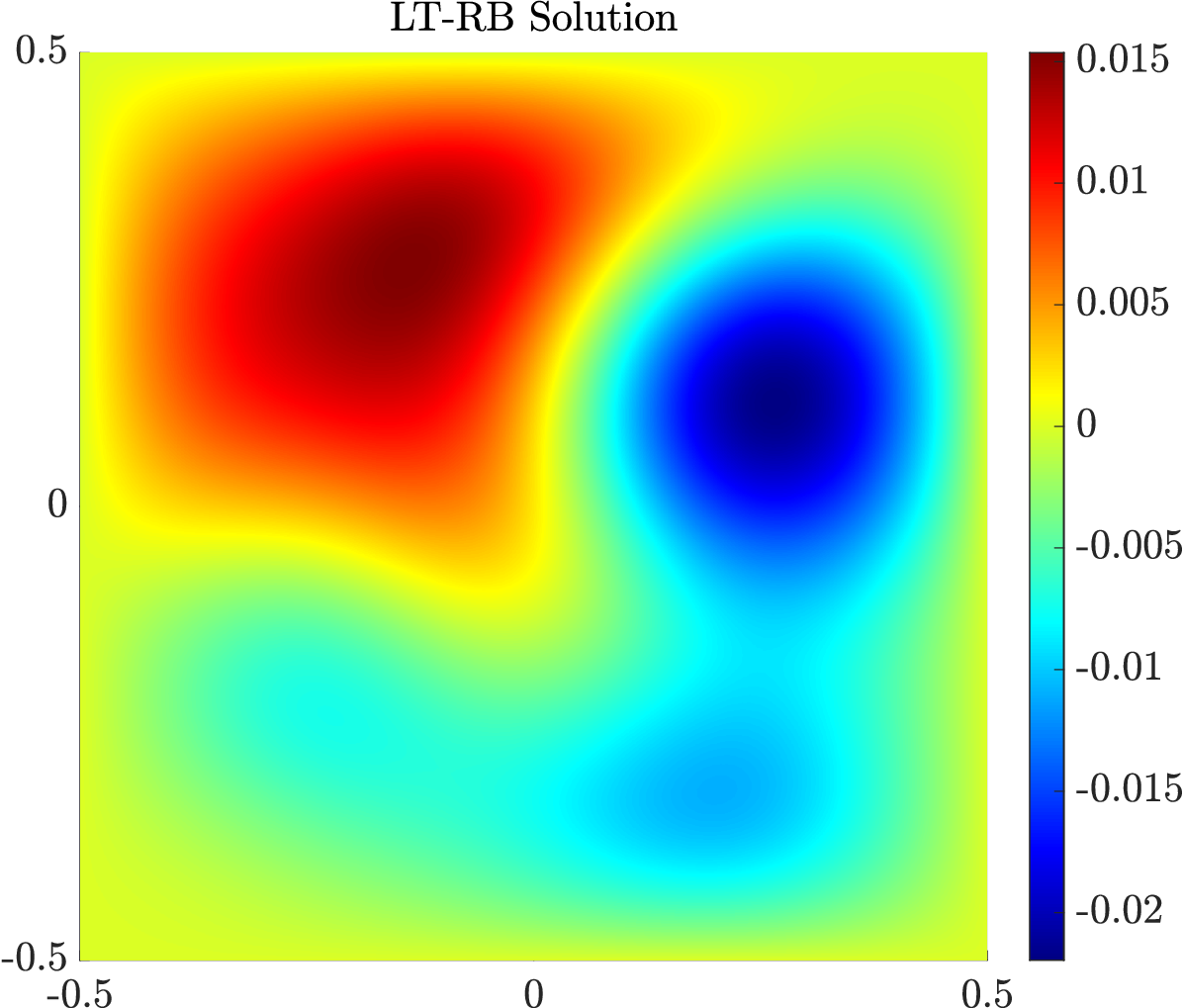}
		\subcaption{LT-MOR Solution at $t=7.5$.}
		\label{fig:plot_sol_5}
	\end{subfigure}
	\begin{subfigure}{0.4\linewidth}
		\includegraphics[width=\textwidth]
		{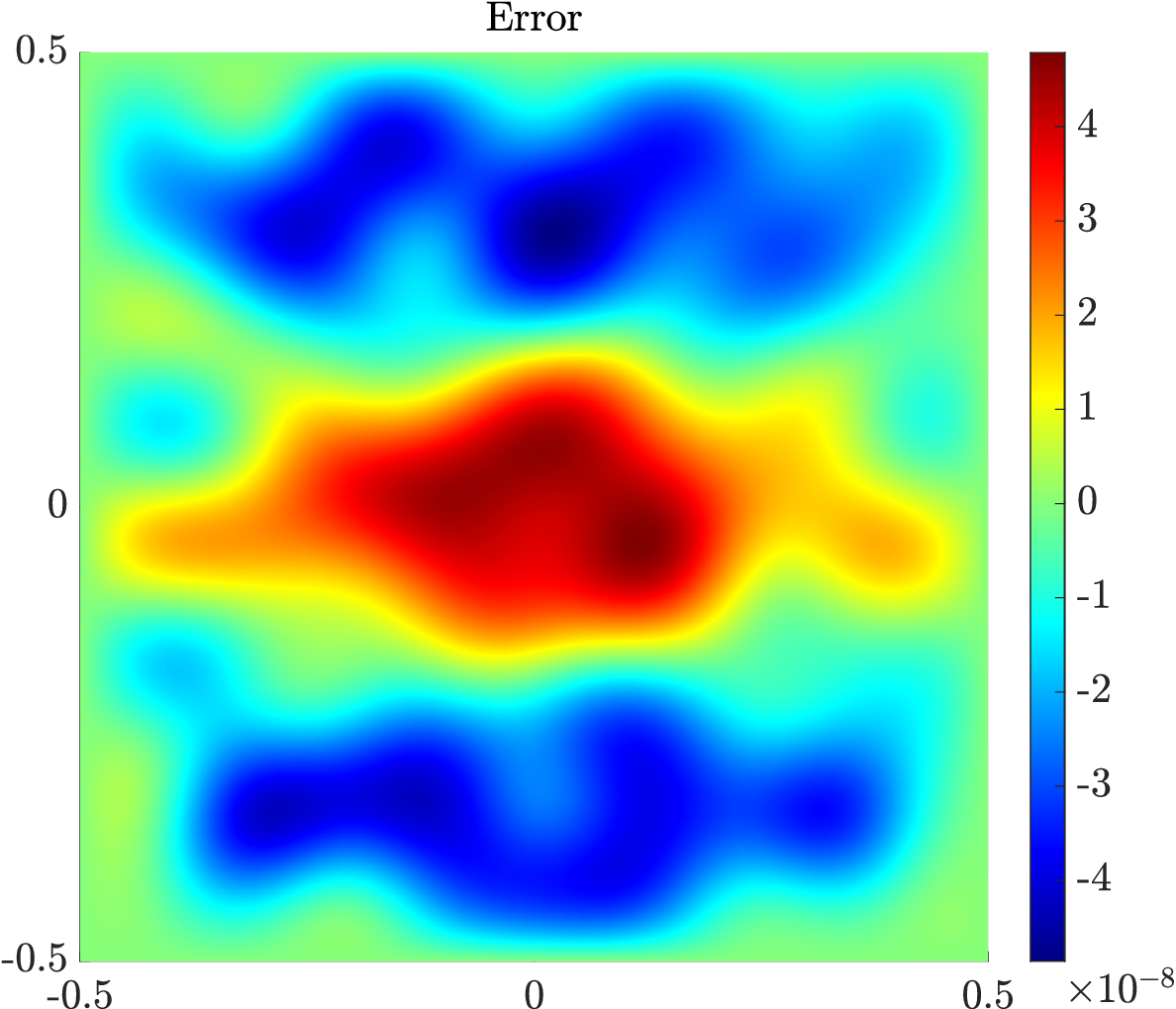}
		\subcaption{Error at $t=7.5$.}
		\label{fig:plot_sol_6}
	\end{subfigure}
	\centering
	\begin{subfigure}{0.4\linewidth}
		\includegraphics[width=\textwidth]
		{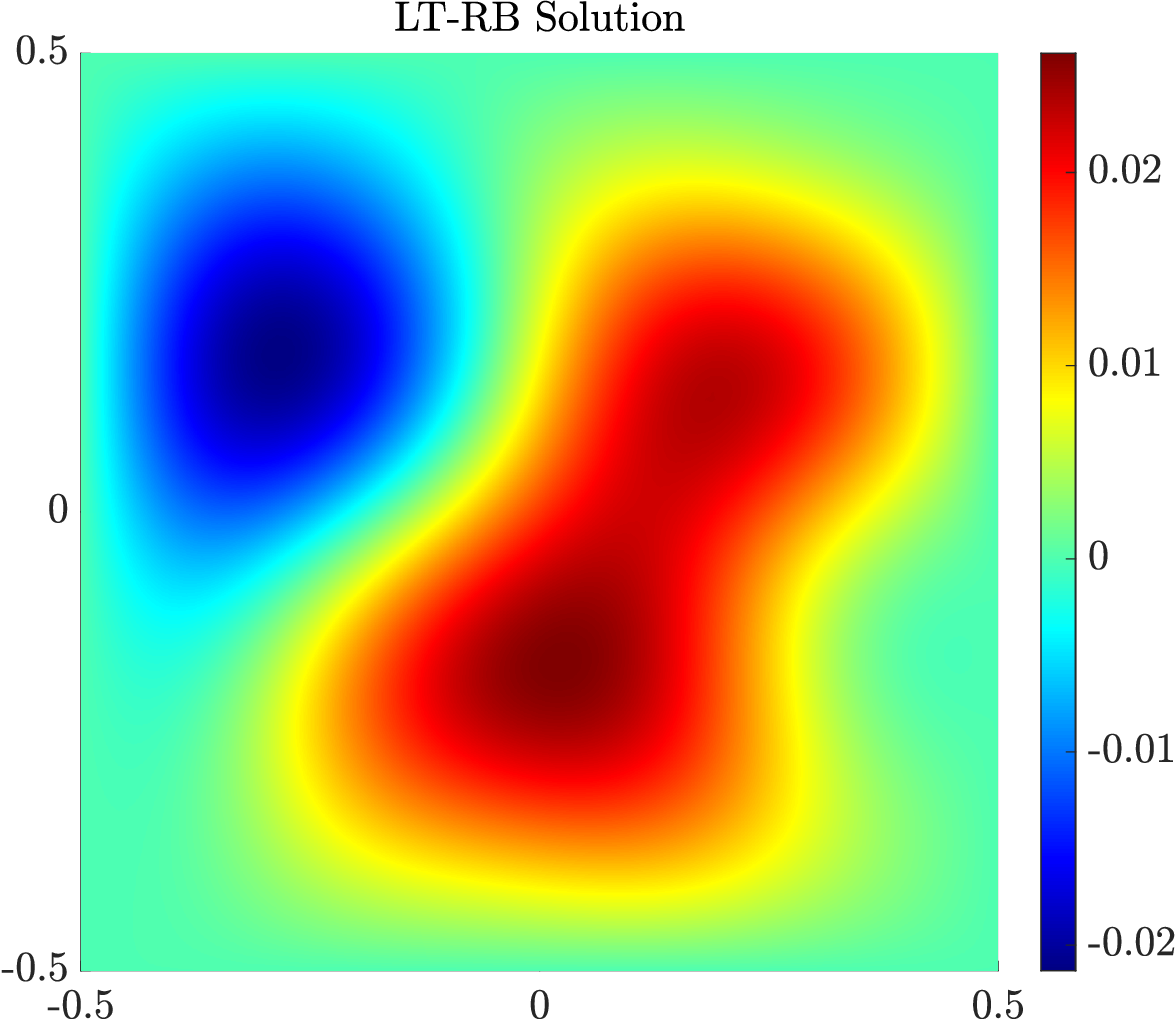}
		\subcaption{LT-MOR Solution at $t=10$.}
		\label{fig:plot_sol_7}
	\end{subfigure}
	\begin{subfigure}{0.4\linewidth}
		\includegraphics[width=\textwidth]
		{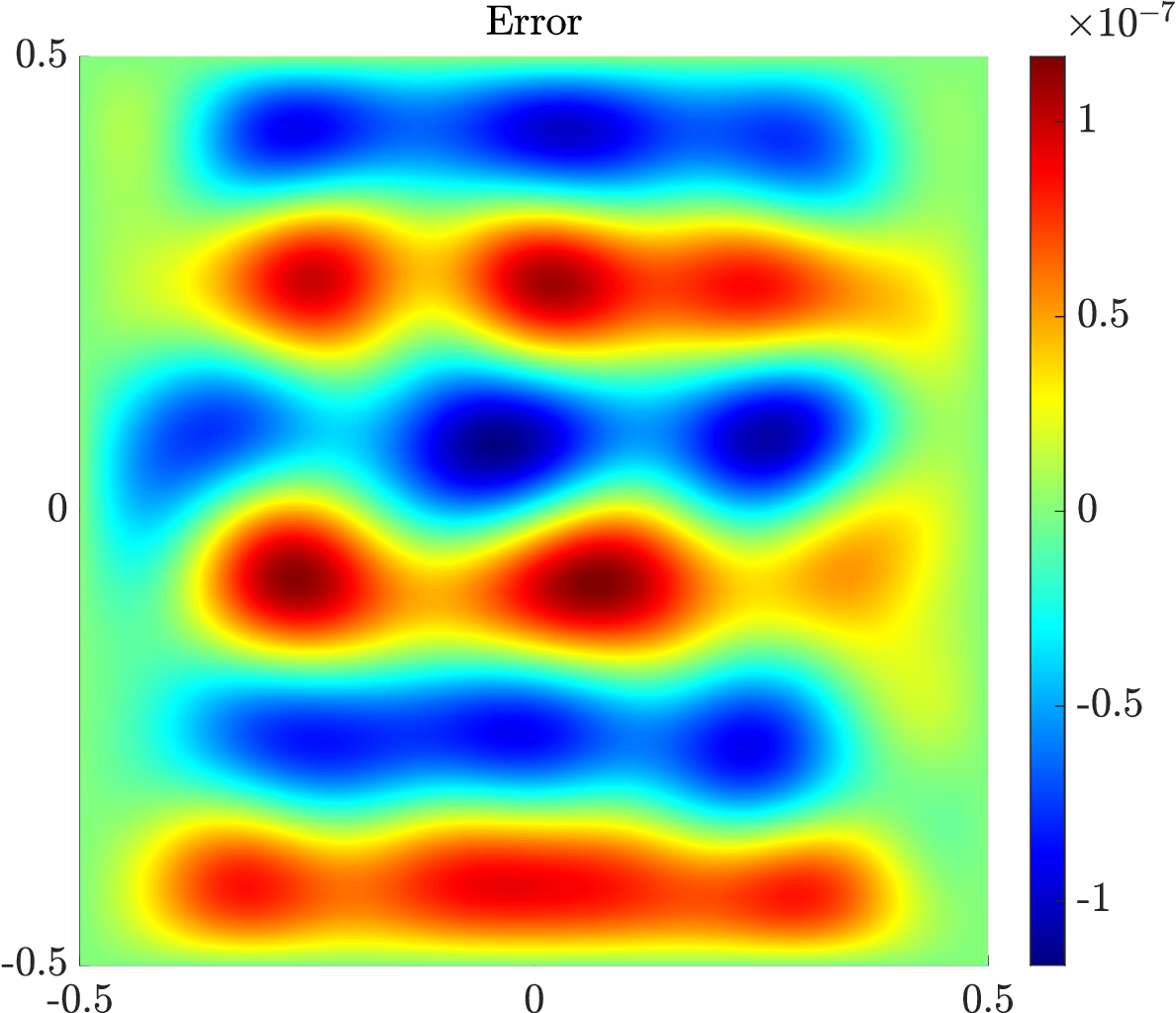}
		\subcaption{Error at $t=10$.}
		\label{fig:plot_sol_8}
	\end{subfigure}
	\caption{\label{fig:plot_Sol_1} 
	LT-MOR solution and the difference between the LT-MOR solution and the high-fidelity solution
	at different times. Figures~\ref{fig:plot_sol_1}--\ref{fig:plot_sol_8} show the LT-MOR solution and the
	corresponding error at $t\in\{2.5,5,7.5,10\}$ for $t_0 = 2.5$, $\alpha = \frac{5}{2} \pi$, $M = 175$, and $R = 50$.
	}
\end{figure}

\subsection{Discussion}
We conclude by discussing the three performance criteria introduced at the beginning of Section~\ref{sec:numerical_results}.

\begin{itemize}
    \item[(i)]
    {\bf Accuracy with respect to the high-fidelity solution.}
    Figures~\ref{fig:error_rel_U1}--\ref{fig:error_rel_U4} show a rapid decay of the relative error as the reduced dimension $R$ increases, in line with the exponential best-approximation bound established in Lemma~\ref{lem:approximation_Hardy_spaces_2}
    and in Theorem~\ref{thm:fully_discrete_error} when considering the fully discrete setting. In practice, this regime is observed once the Laplace variable is sampled sufficiently fine; otherwise, the error due to the discretization of the Laplace-domain quantities becomes dominant (see Section~\ref{sec:fully_discrete}, in particular Lemma~\ref{eq:error_discrete_norm} and subsequently Theorem~\ref{thm:fully_discrete_error}).

    \item[(ii)]
    {\bf Accuracy with respect to the number of snapshots in the offline phase.}
    The same figures also highlight that the attainable accuracy for fixed $R$ is influenced by the number of Laplace-domain samples. Section~\ref{sec:fully_discrete} quantifies this effect: Lemma~\ref{eq:error_discrete_norm} provides a concrete sampling strategy and shows an (essentially) exponential decay of the sampling error as $M$ increases, until the reduced-basis truncation error becomes dominant. We recall that, following the symmetry considerations discussed in \cite[Section~5.4]{henriquez2024fast}, we effectively compute only $M+1$ samples instead of $2M+1$.

    \item[(iii)]
    {\bf Speed-up.}
    The timing breakdown in Figure~\ref{fig:speed_square} indicates that the offline phase (cf.~Section~\ref{sec:laplace_transform_MOR}) is dominated by the independent Laplace-domain solves, while the online stage amounts to solving the reduced problem (Problem~\ref{pr:sdpr})---that is, dense systems of size $R\times R$ at each time step---and reconstructing the solution in $\mathcal{V}_h$. In the configurations reported in Figure~\ref{fig:speed_square} (total bars), the overall wall-clock time is reduced from $50.8\,$s (high-fidelity) to $3.56\,$s, $4.75\,$s, $5.71\,$s, $6.61\,$s, and $8.24\,$s (LT-MOR), corresponding to speed-ups of approximately $14.3\times$, $10.7\times$, $8.9\times$, $7.7\times$, and $6.2\times$, respectively. The reduced solve has cost $\mathcal{O}(R^3)$ per time step and the reconstruction $\mathcal{O}(R N_t)$ overall, which is small compared with the cost of computing the Laplace-domain samples. Since the Laplace-domain solves are embarrassingly parallel, additional wall-clock savings can be expected if these computations are distributed.
\end{itemize}

\section{Concluding Remarks}
\label{sec:concluding_remarks}
In this work we introduced and analyzed a Laplace-transform-based reduced basis (LT-MOR)
strategy for abstract second-order wave problems with vanishing initial conditions and a
separable forcing term given by the product of a temporal Ricker wavelet and a spatial
profile. Working in the Laplace domain turns the evolution problem into a family of
parameter-dependent, stationary problems in the complex Laplace variable, which allows one
to construct reduced spaces using standard snapshot/POD technology and to recover the
time-domain solution via any suitable time-stepping scheme.

On the theoretical side, we proved exponential convergence rates for the LT-MOR
approximation and established bounds that are robust with respect to the parameters
controlling the Ricker wavelet's shape and width. Our analysis also identifies an intrinsic
accuracy barrier governed by the value of the wavelet at the initial time, thereby providing
a computable criterion for the regime in which exponential decay of the error can be
expected. Importantly, the resulting estimates are independent of the underlying Galerkin
discretization space and hence decouple the MOR approximation properties from the specific
high-fidelity solver.

The numerical experiments for the scalar wave equation corroborate the theoretical
predictions: the reduced solutions achieve high accuracy with comparatively small reduced
dimensions, and the offline--online decomposition yields substantial speed-ups relative to
the high-fidelity time stepping while maintaining controlled errors over the full time
interval of interest.

Several extensions are natural and will be addressed in future work. First, it is of
practical interest to incorporate additional physical and/or geometric parameters (e.g.,
material coefficients or source locations), leading to a genuinely parametric setting and
to reduced space constructions based on POD--Greedy or related certified greedy strategies.
Second, it would be worthwhile to generalize the approach to alternative seismic source
models beyond the Ricker wavelet, for instance exponentially decaying wavelets such as the
K\"upper wavelet used in mantle-structure simulations \cite{igel1993p}. Finally, extending
the methodology and its analysis to more complex wave systems (e.g., elastic formulations),
and to adaptive sampling strategies in the Laplace domain, are directions of current and
future research on this topic. 

\begin{acknowledgements}
FH's work is funded by the Deutsche Forschungsgemeinschaft (DFG, German Research Foundation) – Project-ID 258734477 – SFB 1173 and the Austrian Science Fund (FWF) under the project I6667-N. 
Funding was also received from the European Research Council (ERC) under the Horizon 2020 research and innovation program 
of the European Union (Grant agreement No. 101125225, Project Name: ``New Frontiers in Optimal Adaptivity", PI: Prof. Dr. Michael Feischl).
\end{acknowledgements}


\bibliographystyle{spmpsci} 
\bibliography{ref.bib}

@book{quarteroni2015reduced,
	author = {Quarteroni, Alfio and Manzoni, Andrea and Negri, Federico},
	publisher = {Springer},
	title = {{Reduced basis methods for partial differential equations: an introduction}},
	volume = {92},
	year = {2015}}

@article{henriquez2024fast,
	author = {Henr{\'\i}quez, Fernando and Hesthaven, Jan S},
	journal = {arXiv preprint arXiv:2403.02847v3},
	title = {Fast Numerical Approximation of Parabolic Problems Using Model Order Reduction and the {L}aplace Transform},
	year = {2024}}

@incollection{sauter2010boundary,
  title={{Boundary element methods}},
  author={Sauter, Stefan A and Schwab, Christoph},
  booktitle={{Boundary Element Methods}},
  pages={183--287},
  year={2010},
  publisher={Springer}
}

@book{RR97,
	author = {Rosenblum, Marvin and Rovnyak, James},
	publisher = {Oxford University Press},
	title = {{Hardy classes and operator theory}},
	year = {1985}}

@book{hille1996functional,
	author = {Hille, Einar and Phillips, Ralph Saul},
	publisher = {American Mathematical Soc.},
	title = {{Functional analysis and semi-groups}},
	volume = {31},
	year = {1996}}

@book{bracewell2000fourier,
  author       = {Bracewell, Ronald N.},
  title        = {{The Fourier Transform and Its Applications}},
  edition      = {3rd},
  publisher    = {McGraw-Hill},
  address      = {Boston},
  year         = {2000},
  isbn         = {978-0-07-303938-1}
}

@book{stenger2012numerical,
  title={{Numerical methods based on Sinc and analytic functions}},
  author={Stenger, Frank},
  volume={20},
  year={2012},
  publisher={Springer Science \& Business Media}
}

@article{HBN86,
	author = {et T. Ha Duong, A Bamberger and Nedelec, JC},
	journal = {Mathematical methods in the applied sciences},
	number = {1},
	pages = {405--435},
	publisher = {Wiley Online Library},
	title = {Formulation variationnelle espace-temps pour le calcul par potentiel retard{\'e} de la diffraction d'une onde acoustique (I)},
	volume = {8},
	year = {1986}}

@article{MST20,
	author = {Melenk, Jens M and Sauter, Stefan A and Torres, C{\'e}line},
	journal = {SIAM Journal on Numerical Analysis},
	number = {4},
	pages = {2119--2143},
	publisher = {SIAM},
	title = {Wavenumber explicit analysis for {G}alerkin discretizations of lossy Helmholtz problems},
	volume = {58},
	year = {2020}}

@unknown{HenriquezHesthaven2024,
	author = {Henriquez, Fernando and Hesthaven, Jan},
	doi = {10.48550/arXiv.2405.19896},
	month = {05},
	title = {Fast Numerical Approximation of Linear, Second-Order Hyperbolic Problems Using Model Order Reduction and the {L}aplace Transform},
	year = {2024}}

@article{hawkins2025model,
author = {Hawkins, Rhys and Khalid, Muhammad Hamza and Schlottbom, Matthias and Smetana, Kathrin},
title = {{Model Order Reduction for Seismic Applications}},
journal = {SIAM Journal on Scientific Computing},
volume = {47},
number = {5},
pages = {B1045-B1076},
year = {2025},
doi = {10.1137/24M1667737}
}

@article{Hawkinsetal2023,
  author  = {Hawkins, R. and Khalid, M. H. and Smetana, K. and Trampert, J.},
  title   = {{Model order reduction for seismic waveform modelling: inspiration from normal modes}},
  journal = {Geophysical Journal International},
  year    = {2023},
  volume  = {234},
  number  = {3},
  pages   = {2255--2283},
  doi     = {10.1093/gji/ggad195}
}

@article{alford1974accuracy,
  title={{Accuracy of finite-difference modeling of the acoustic wave equation}},
  author={Alford, R. M. and Kelly, K. R. and Whitmore, N. D.},
  journal={Geophysics},
  volume={39},
  number={6},
  pages={834--842},
  year={1974},
  publisher={Society of Exploration Geophysicists},
  doi={10.1190/1.1440470}
}

@book{fichtner2010full,
  title={{Full Seismic Waveform Modelling and Inversion}},
  author={Fichtner, Andreas},
  year={2010},
  publisher={Springer-Verlag},
  address={Berlin, Germany},
  series={Advances in Geophysical and Environmental Mechanics and Mathematics},
  doi={10.1007/978-3-642-15807-0}
}

@misc{schneider2025understanding,
  title={Understanding the {Finite-Difference Time-Domain} Method},
  author={Schneider, John B.},
  year={2025},
  publisher={Washington State University},
  url={https://eecs.wsu.edu/~schneidj/ufdtd/},
  note={Online Open-Source Textbook}
}

@inproceedings{schneider2004plane,
  title={Plane waves in {FDTD} simulations and a nearly perfect total-field/scattered-field boundary},
  author={Schneider, John B.},
  booktitle={{2004 IEEE Antennas and Propagation Society International Symposium}},
  volume={2},
  pages={1927--1930},
  year={2004},
  organization={IEEE},
  doi={10.1109/APS.2004.1330580}
}

@article{SchneiderChen2011,
  author  = {Schneider, John B. and Chen, Zhen},
  title   = {Incorporating the {G-TFSF} Concept into the Analytic Field Propagation {TFSF} Method},
  journal = {IEEE Transactions on Antennas and Propagation},
  year    = {2011},
  volume  = {59},
  number  = {9},
  pages   = {3296--3304},
  doi     = {10.1109/tap.2011.2161452}
}

@article{ricker1943further,
  title={{Further developments in the wavelet theory of seismogram structure}},
  author={Ricker, Norman},
  journal={Bulletin of the Seismological Society of America},
  volume={33},
  number={3},
  pages={197--228},
  year={1943},
  publisher={Seismological Society of America},
  doi={10.1785/BSSA0330030197}
}

@article{ricker1944wavelet,
  title={{Wavelet functions and their polynomials}},
  author={Ricker, Norman},
  journal={Geophysics},
  volume={9},
  number={3},
  pages={314--323},
  year={1944},
  publisher={Society of Exploration Geophysicists},
  doi={10.1190/1.1445082}
}

@article{igel1993p,
  title={{P-SV} wave propagation in the {Earth's} mantle using finite differences: {Application} to heterogeneous lowermost mantle structure},
  author={Igel, Heiner and Weber, Michael},
  journal={Geophysical Research Letters},
  volume={20},
  number={7},
  pages={531--534},
  year={1993},
  doi={10.1029/93GL00431}
}

@article{komatitsch1999introduction,
  title={{Introduction to the spectral element method for three-dimensional seismic wave propagation}},
  author={Komatitsch, Dimitri and Tromp, Jeroen},
  journal={Geophysical Journal International},
  volume={139},
  number={3},
  pages={806--822},
  year={1999},
  publisher={Blackwell Publishing Ltd Oxford, UK},
  doi={10.1046/j.1365-246x.1999.00967.x}
}

@article{duan2023machine,
	author = {Duan, Junming and Wang, Qian and Hesthaven, Jan S},
	journal = {arXiv preprint arXiv:2305.09199},
	title = {{Machine learning enhanced real-time aerodynamic forces prediction based on sparse pressure sensor inputs}},
	year = {2023}}

@article{duan2023non,
	author = {Duan, Junming and Hesthaven, Jan S},
	journal = {arXiv preprint arXiv:2303.02986},
	title = {{Non-intrusive data-driven reduced-order modeling for time-dependent parametrized problems}},
	year = {2023}}

@article{peherstorfer2016data,
	author = {Peherstorfer, Benjamin and Willcox, Karen},
	journal = {Computer Methods in Applied Mechanics and Engineering},
	pages = {196--215},
	publisher = {Elsevier},
	title = {{Data-driven operator inference for nonintrusive projection-based model reduction}},
	volume = {306},
	year = {2016}}

@book{kutz2016dynamic,
	author = {Kutz, J Nathan and Brunton, Steven L and Brunton, Bingni W and Proctor, Joshua L},
	publisher = {SIAM},
	title = {{Dynamic mode decomposition: data-driven modeling of complex systems}},
	year = {2016}}

@article{schmid2010dynamic,
	author = {Schmid, Peter J},
	journal = {Journal of fluid mechanics},
	pages = {5--28},
	publisher = {Cambridge University Press},
	title = {{Dynamic mode decomposition of numerical and experimental data}},
	volume = {656},
	year = {2010}}

@article{hesthaven2023adaptive,
	author = {Hesthaven, Jan and Pagliantini, Cecilia and Ripamonti, Nicol{\`o}},
	journal = {Mathematics of Computation},
	title = {{Adaptive symplectic model order reduction of parametric particle-based Vlasov--Poisson equation}},
	year = {2023}}

@article{hesthaven2021structure,
	author = {Hesthaven, Jan and Pagliantini, Cecilia},
	journal = {Mathematics of Computation},
	number = {330},
	pages = {1701--1740},
	title = {{Structure-preserving reduced basis methods for Poisson systems}},
	volume = {90},
	year = {2021}}

@article{hesthaven2022rank,
	author = {Hesthaven, Jan S and Pagliantini, Cecilia and Ripamonti, Nicol{\`o}},
	journal = {ESAIM: Mathematical Modelling and Numerical Analysis},
	number = {2},
	pages = {617--650},
	publisher = {EDP Sciences},
	title = {Rank-adaptive structure-preserving model order reduction of {H}amiltonian systems},
	volume = {56},
	year = {2022}}

@article{afkham2017structure,
	author = {Afkham, Babak Maboudi and Hesthaven, Jan S},
	journal = {SIAM Journal on Scientific Computing},
	number = {6},
	pages = {A2616--A2644},
	publisher = {SIAM},
	title = {Structure preserving model reduction of parametric {H}amiltonian systems},
	volume = {39},
	year = {2017}}

@article{xiao2017parameterized,
	author = {Xiao, Dunhui and Fang, F and Pain, CC and Navon, IM},
	journal = {Computer Methods in Applied Mechanics and Engineering},
	pages = {868--889},
	publisher = {Elsevier},
	title = {{A parameterized non-intrusive reduced order model and error analysis for general time-dependent nonlinear partial differential equations and its applications}},
	volume = {317},
	year = {2017}}

@article{audouze2013nonintrusive,
	author = {Audouze, Christophe and De Vuyst, Florian and Nair, Prasanth B},
	journal = {Numerical Methods for Partial Differential Equations},
	number = {5},
	pages = {1587--1628},
	publisher = {Wiley Online Library},
	title = {{Nonintrusive reduced-order modeling of parametrized time-dependent partial differential equations}},
	volume = {29},
	year = {2013}}

@article{lieu2006reduced,
	author = {Lieu, Thuan and Farhat, Charbel and Lesoinne, Michel},
	journal = {Computer methods in applied mechanics and engineering},
	number = {41-43},
	pages = {5730--5742},
	publisher = {Elsevier},
	title = {{Reduced-order fluid/structure modeling of a complete aircraft configuration}},
	volume = {195},
	year = {2006}}

@incollection{rozza2014fundamentals,
	author = {Rozza, Gianluigi},
	booktitle = {{Separated Representations and PGD-Based Model Reduction: Fundamentals and Applications}},
	pages = {153--227},
	publisher = {Springer},
	title = {{Fundamentals of reduced basis method for problems governed by parametrized PDEs and applications}},
	year = {2014}}

@article{prud2002reliable,
	author = {Prud'Homme, Christophe and Rovas, Dimitrios V and Veroy, Karen and Machiels, Luc and Maday, Yvon and Patera, Anthony T and Turinici, Gabriel},
	journal = {J. Fluids Eng.},
	number = {1},
	pages = {70--80},
	title = {{Reliable real-time solution of parametrized partial differential equations: Reduced-basis output bound methods}},
	volume = {124},
	year = {2002}}

@book{hesthaven2016certified,
	author = {Hesthaven, Jan S and Rozza, Gianluigi and Stamm, Benjamin and others},
	publisher = {Springer},
	title = {{Certified reduced basis methods for parametrized partial differential equations}},
	volume = {590},
	year = {2016}}

@article{guglielmi2021pseudospectral,
	author = {Guglielmi, Nicola and L{\'o}pez-Fern{\'a}ndez, Mar{\'\i}a and Manucci, Mattia},
	journal = {Journal of Scientific Computing},
	pages = {1--31},
	publisher = {Springer},
	title = {{Pseudospectral roaming contour integral methods for convection-diffusion equations}},
	volume = {89},
	year = {2021}}

@article{guglielmi2023model,
	author = {Guglielmi, Nicola and Manucci, Mattia},
	journal = {SIAM Journal on Scientific Computing},
	number = {4},
	pages = {A1711--A1740},
	publisher = {SIAM},
	title = {{Model order reduction in contour integral methods for parametric PDEs}},
	volume = {45},
	year = {2023}}

@article{haasdonk2008reduced,
	author = {Haasdonk, Bernard and Ohlberger, Mario},
	journal = {ESAIM: Mathematical Modelling and Numerical Analysis},
	number = {2},
	pages = {277--302},
	publisher = {EDP Sciences},
	title = {{Reduced basis method for finite volume approximations of parametrized linear evolution equations}},
	volume = {42},
	year = {2008}}

@article{hesthaven2014efficient,
	author = {Hesthaven, Jan S and Stamm, Benjamin and Zhang, Shun},
	journal = {ESAIM: Mathematical Modelling and Numerical Analysis},
	number = {1},
	pages = {259--283},
	publisher = {EDP Sciences},
	title = {{Efficient greedy algorithms for high-dimensional parameter spaces with applications to empirical interpolation and reduced basis methods}},
	volume = {48},
	year = {2014}}

@article{liang2002proper,
	author = {Liang, YC and Lee, HP and Lim, SP and Lin, WZ and Lee, KH and Wu, CG1237},
	journal = {Journal of Sound and vibration},
	number = {3},
	pages = {527--544},
	publisher = {Elsevier},
	title = {{Proper orthogonal decomposition and its applications---Part I: Theory}},
	volume = {252},
	year = {2002}}

@inproceedings{OhlbergerRave2016,
  author    = {Ohlberger, Mario and Rave, Stephan},
  title     = {Reduced Basis Methods: Success, Limitations and Future Challenges},
  booktitle = {Proceedings of the Conference Algoritmy},
  year      = {2016},
  pages     = {1--12},
  url       = {http://pc2.iam.fmph.uniba.sk/institute/amuc/ojs/index.php/algoritmy/article/view/389},
  doi       = {10.48550/arxiv.1511.02021}
}

@article{GU19,
	author = {Greif, Constantin and Urban, Karsten},
	doi = {10.1016/j.aml.2019.05.013},
	journal = {Applied Mathematics Letters},
	pages = {216--222},
	publisher = {Elsevier},
	title = {{Decay of the Kolmogorov N-width for wave problems}},
	volume = {96},
	year = {2019}}

@article{AGU25,
	author = {Arbes, Florian and Greif, Constantin and Urban, Karsten},
	doi = {10.1007/s10444-025-10224-0},
	journal = {Advances in Computational Mathematics},
	number = {2},
	pages = {13},
	title = {{The Kolmogorov N-width for linear transport: exact representation and the influence of the data}},
	volume = {51},
	year = {2025}}

@article{Bigoni2020a,
	author = {Bigoni, Caterina and Hesthaven, Jan S.},
	doi = {10.1016/j.cma.2020.112896},
	issn = {00457825},
	journal = {Computer Methods in Applied Mechanics and Engineering},
	pages = {112896},
	publisher = {Elsevier B.V.},
	title = {{Simulation-based Anomaly Detection and Damage Localization: an application to Structural Health Monitoring}},
	volume = {363},
	year = {2020}}

@article{huynh_laplace_2011,
	author = {Huynh, D.B. Phuong and Knezevic, David J. and Patera, Anthony T.},
	doi = {10.1016/j.crma.2011.02.003},
	issn = {1631073X},
	journal = {Comptes Rendus Mathematique},
	langid = {english},
	number = {7},
	pages = {401--405},
	shortjournal = {Comptes Rendus Mathematique},
	title = {A {L}aplace transform certified reduced basis method; application to the heat equation and wave equation},
	volume = {349},
	year = {2011}}

@article{bui-thanh_model_2008,
	author = {Bui-Thanh, T. and Willcox, K. and Ghattas, O.},
	doi = {10.1137/070694855},
	issn = {1064-8275},
	journal = {{SIAM} Journal on Scientific Computing},
	number = {6},
	pages = {3270--3288},
	shortjournal = {{SIAM} J. Sci. Comput.},
	title = {{Model Reduction for Large-Scale Systems with High-Dimensional Parametric Input Space}},
	volume = {30},
	year = {2008}}

@article{Hesthaven2022,
	author = {Hesthaven, Jan S and Pagliantini, Cecilia and Rozza, Gianluigi},
	doi = {10.1017/S0962492922000058},
	issn = {0962-4929},
	journal = {Acta Numerica},
	pages = {265--345},
	title = {{Reduced basis methods for time-dependent problems}},
	volume = {31},
	year = {2022}}

@article{buffa_priori_2012,
	author = {Buffa, Annalisa and Maday, Yvon and Patera, Anthony T. and Prud'homme, Christophe and Turinici, Gabriel},
	doi = {10.1051/m2an/2011056},
	issn = {0764-583X, 1290-3841},
	journal = {{ESAIM}: Mathematical Modelling and Numerical Analysis},
	langid = {english},
	number = {3},
	pages = {595--603},
	rights = {{\copyright} {EDP} Sciences, {SMAI}, 2012},
	shortjournal = {{ESAIM}: M2AN},
	title = {{A priori convergence of the Greedy algorithm for the parametrized reduced basis method}},
	volume = {46},
	year = {2012}}

@article{devore_greedy_2013,
	author = {{DeVore}, Ronald and Petrova, Guergana and Wojtaszczyk, Przemyslaw},
	doi = {10.1007/s00365-013-9186-2},
	issn = {1432-0940},
	journal = {Constructive Approximation},
	langid = {english},
	number = {3},
	pages = {455--466},
	shortjournal = {Constr Approx},
	title = {Greedy Algorithms for Reduced Bases in {B}anach Spaces},
	volume = {37},
	year = {2013}}

@article{kuhlman_review_2013,
	author = {Kuhlman, Kristopher L.},
	doi = {10/f4xznh},
	issn = {1572-9265},
	journal = {Numerical Algorithms},
	langid = {english},
	number = {2},
	pages = {339--355},
	shortjournal = {Numer Algor},
	title = {Review of inverse {L}aplace transform algorithms for {L}aplace-space numerical approaches},
	urldate = {2022-12-23},
	volume = {63},
	year = {2013}}

@article{Rawlinson2014SeismicTomography,
  author  = {Nicholas Rawlinson and Andreas Fichtner and Malcolm Sambridge and Mallory K. Young},
  title   = {{Seismic Tomography and the Assessment of Uncertainty}},
  journal = {Advances in Geophysics},
  volume  = {55},
  pages   = {1--76},
  year    = {2014},
  doi     = {10.1016/bs.agph.2014.08.001},
  publisher = {Academic Press Inc.},
  issn = {0065-2687}
}

@article{Tromp2005AdjointBanana,
  author  = {Tromp, Jeroen and Tape, Carl and Liu, Qinya},
  title   = {{Seismic tomography, adjoint methods, time reversal and banana-doughnut kernels}},
  journal = {Geophysical Journal International},
  volume  = {160},
  number  = {1},
  pages   = {195--216},
  year    = {2005},
  doi     = {10.1111/j.1365-246X.2004.02453.x}
}

@article{Bozdag2011MisfitPhaseEnvelope,
  author  = {Bozda{\u g}, Ebru and Trampert, Jeannot and Tromp, Jeroen},
  title   = {{Misfit functions for full waveform inversion based on instantaneous phase and envelope measurements}},
  journal = {Geophysical Journal International},
  volume  = {185},
  number  = {2},
  pages   = {845--870},
  year    = {2011},
  doi     = {10.1111/j.1365-246X.2011.04970.x}
}

@book{Tromp2013FullWaveform,
  author    = {Tromp, Jeroen},
  title     = {{Full Seismic Waveform Modelling and Inversion}},
  series    = {Lecture Notes},
  publisher = {Princeton University},
  year      = {2013}
}

@article{berkooz1993proper,
  title={{The proper orthogonal decomposition in the analysis of turbulent flows}},
  author={Berkooz, Gal and Holmes, Philip and Lumley, John L},
  journal={Annual review of fluid mechanics},
  volume={25},
  number={1},
  pages={539--575},
  year={1993},
  publisher={Annual Reviews 4139 El Camino Way, PO Box 10139, Palo Alto, CA 94303-0139, USA}
}

@article{kunisch2001galerkin,
  title={{G}alerkin proper orthogonal decomposition methods for parabolic problems},
  author={Kunisch, Karl and Volkwein, Stefan},
  journal={Numerische mathematik},
  volume={90},
  number={1},
  pages={117--148},
  year={2001},
  publisher={Springer}
}

@article{haasdonk2013convergence,
  title={Convergence rates of the {POD}--greedy method},
  author={Haasdonk, Bernard},
  journal={ESAIM: Mathematical modelling and numerical Analysis},
  volume={47},
  number={3},
  pages={859--873},
  year={2013},
  publisher={EDP Sciences}
}

@incollection{siena2023introduction,
  title={An introduction to {POD}-greedy-{G}alerkin reduced basis method},
  author={Siena, Pierfrancesco and Girfoglio, Michele and Rozza, Gianluigi},
  booktitle={{Reduced Order Models for the Biomechanics of Living Organs}},
  pages={127--145},
  year={2023},
  publisher={Elsevier}
}

@article{peherstorfer2022breaking,
  title={{Breaking the Kolmogorov barrier with nonlinear model reduction}},
  author={Peherstorfer, Benjamin},
  journal={Notices of the American Mathematical Society},
  volume={69},
  number={5},
  pages={725--733},
  year={2022},
  publisher={American Mathematical Society}
}

@article{peherstorfer2020model,
  title={{Model reduction for transport-dominated problems via online adaptive bases and adaptive sampling}},
  author={Peherstorfer, Benjamin},
  journal={SIAM Journal on Scientific Computing},
  volume={42},
  number={5},
  pages={A2803--A2836},
  year={2020},
  publisher={SIAM}
}

@article{koch2007dynamical,
  title={{Dynamical low-rank approximation}},
  author={Koch, Othmar and Lubich, Christian},
  journal={SIAM Journal on Matrix Analysis and Applications},
  volume={29},
  number={2},
  pages={434--454},
  year={2007},
  publisher={SIAM}
}

@article{sapsis2009dynamically,
  title={{Dynamically orthogonal field equations for continuous stochastic dynamical systems}},
  author={Sapsis, Themistoklis P and Lermusiaux, Pierre FJ},
  journal={Physica D: Nonlinear Phenomena},
  volume={238},
  number={23-24},
  pages={2347--2360},
  year={2009},
  publisher={Elsevier}
}

@article{musharbash2015error,
  title={{Error analysis of the dynamically orthogonal approximation of time dependent random PDEs}},
  author={Musharbash, Eleonora and Nobile, Fabio and Zhou, Tao},
  journal={SIAM Journal on Scientific Computing},
  volume={37},
  number={2},
  pages={A776--A810},
  year={2015},
  publisher={SIAM}
}

@article{iollo2014advection,
  title={{Advection modes by optimal mass transfer}},
  author={Iollo, Angelo and Lombardi, Damiano},
  journal={Physical Review E},
  volume={89},
  number={2},
  pages={022923},
  year={2014},
  publisher={APS}
}

@article{lee2020model,
  title={{Model reduction of dynamical systems on nonlinear manifolds using deep convolutional autoencoders}},
  author={Lee, Kookjin and Carlberg, Kevin T},
  journal={Journal of Computational Physics},
  volume={404},
  pages={108973},
  year={2020},
  publisher={Elsevier}
}

@article{reiss2018shifted,
  title={{The shifted proper orthogonal decomposition: A mode decomposition for multiple transport phenomena}},
  author={Reiss, Julius and Schulze, Philipp and Sesterhenn, J{\"o}rn and Mehrmann, Volker},
  journal={SIAM Journal on Scientific Computing},
  volume={40},
  number={3},
  pages={A1322--A1344},
  year={2018},
  publisher={SIAM}
}

@article{taddei2020registration,
  title={{A registration method for model order reduction: data compression and geometry reduction}},
  author={Taddei, Tommaso},
  journal={SIAM Journal on Scientific Computing},
  volume={42},
  number={2},
  pages={A997--A1027},
  year={2020},
  publisher={SIAM}
}

@article{urban2012new,
  title={{A new error bound for reduced basis approximation of parabolic partial differential equations}},
  author={Urban, Karsten and Patera, Anthony T},
  journal={Comptes Rendus. Math{\'e}matique},
  volume={350},
  number={3-4},
  pages={203--207},
  year={2012}
}

@article{yano2014space,
  title={{A space-time hp-interpolation-based certified reduced basis method for Burgers' equation}},
  author={Yano, Masayuki and Patera, Anthony T and Urban, Karsten},
  journal={Mathematical Models and Methods in Applied Sciences},
  volume={24},
  number={09},
  pages={1903--1935},
  year={2014},
  publisher={World Scientific}
}

@article{peng2016symplectic,
  title={Symplectic model reduction of {H}amiltonian systems},
  author={Peng, Liqian and Mohseni, Kamran},
  journal={SIAM Journal on Scientific Computing},
  volume={38},
  number={1},
  pages={A1--A27},
  year={2016},
  publisher={SIAM}
}

@inproceedings{hesthaven2022structure,
  title={Structure-preserving model order reduction of {H}amiltonian systems},
  author={Hesthaven, Jan S and Pagliantini, Cecilia and Ripamonti, Nicol{\`o} and others},
  booktitle={{Proc Int Cong Math}},
  volume={7},
  pages={5072--97},
  year={2022}
}

@article{wang2012proper,
  title={{Proper orthogonal decomposition closure models for turbulent flows: a numerical comparison}},
  author={Wang, Zhu and Akhtar, Imran and Borggaard, Jeff and Iliescu, Traian},
  journal={Computer Methods in Applied Mechanics and Engineering},
  volume={237},
  pages={10--26},
  year={2012},
  publisher={Elsevier}
}

@article{pan2018data,
  title={{Data-driven discovery of closure models}},
  author={Pan, Shaowu and Duraisamy, Karthik},
  journal={SIAM Journal on Applied Dynamical Systems},
  volume={17},
  number={4},
  pages={2381--2413},
  year={2018},
  publisher={SIAM}
}

@article{iollo2000stability,
  title={Stability properties of {POD}--{G}alerkin approximations for the compressible Navier--Stokes equations},
  author={Iollo, Angelo and Lanteri, St{\'e}phane and D{\'e}sid{\'e}ri, J-A},
  journal={Theoretical and Computational Fluid Dynamics},
  volume={13},
  number={6},
  pages={377--396},
  year={2000},
  publisher={Springer}
}

@article{hesthaven2022reduced,
  title={{Reduced basis methods for time-dependent problems}},
  author={Hesthaven, Jan S and Pagliantini, Cecilia and Rozza, Gianluigi},
  journal={Acta Numerica},
  volume={31},
  pages={265--345},
  year={2022},
  publisher={Cambridge University Press}
}

@article{feischl2025optimal,
  title={{Optimal Time-Adaptivity for Parabolic Problems with applications to Model Order Reduction}},
  author={Feischl, Michael and Henr{\'\i}quez, Fernando and Niederkofler, David},
  journal={arXiv preprint arXiv:2512.05676},
  year={2025}
}

@article{hesthaven2026nonlinear,
  title={{Nonlinear model reduction for transport-dominated problems}},
  author={Hesthaven, Jan S and Peherstorfer, Benjamin and Unger, Benjamin},
  journal={arXiv preprint arXiv:2602.01397},
  year={2026}
}

@article{alouges2018fem,
  title={{FEM and BEM simulations with the Gypsilab framework}},
  author={Alouges, Fran{\c{c}}ois and Aussal, Matthieu},
  journal={The SMAI journal of computational mathematics},
  volume={4},
  pages={297--318},
  year={2018}
}

@article{henriquez2021shape,
  title={{Shape holomorphy of the Calder{\'o}n projector for the Laplacian in $\mathbb{R}^2$}},
  author={Henr{\'\i}quez, Fernando and Schwab, Christoph},
  journal={Integral Equations and Operator Theory},
  volume={93},
  number={4},
  pages={43},
  year={2021},
  publisher={Springer}
}

\end{document}